\documentclass[10pt, a4paper]{amsart}
\usepackage{amssymb}
\usepackage{mathrsfs, bbm}
\usepackage{a4wide}
\usepackage{tikz}
\usetikzlibrary{arrows.meta,calc}

\newcommand{\R}{\mathbb{R}}
\newcommand{\C}{\mathbb{C}}
\newcommand{\N}{\mathbb{N}}
\newcommand{\bB}{\mathbb{B}}
\newcommand{\D}{\mathbb{D}}

\newcommand{\sD}{\mathscr{D}}
\newcommand{\sL}{\mathscr{L}}
\newcommand{\sH}{\mathscr{H}}
\newcommand{\sK}{\mathscr{K}}
\newcommand{\sM}{\mathscr{M}}

\newcommand{\fs}{\mathfrak{s}}
\newcommand{\co}{\mathfrak{c}_{0}}

\def\Xint#1{\mathchoice
{\XXint\displaystyle\textstyle{#1}}
{\XXint\textstyle\scriptstyle{#1}}
{\XXint\scriptstyle\scriptscriptstyle{#1}}
{\XXint\scriptscriptstyle\scriptscriptstyle{#1}}
\!\int}
\def\XXint#1#2#3{{\setbox0=\hbox{$#1{#2#3}{\int}$}
\vcenter{\hbox{$#2#3$}}\kern-.5\wd0}}
\def\dashint{\Xint-}

\newcommand{\bint}{\ensuremath{\dashint}}
\newcommand{\gbint}{\sideset{^g\!\!}{}\bint}
\newcommand{\hbint}{\sideset{^{h}\!\!\!}{}\bint}

\newcommand{\op}{\operatorname}

\newcommand{\Vol}{\op{Vol}}
\newcommand{\Tr}{\op{Tr}}
\newcommand{\Trw}{\Tr_{\omega}}
\newcommand{\ran}{\op{ran}}

\newcommand{\Com}{\op{Com}}

\newcommand{\Res}{\op{Res}}
\newcommand{\car}{\mathbbm{1}}

\newcommand{\dom}{\op{dom}}

\numberwithin{equation}{section}

\newtheorem{theorem}{Theorem}[section]
\newtheorem{proposition}[theorem]{Proposition}
\newtheorem{corollary}[theorem]{Corollary}

\newtheorem{theoremalph}{Theorem}
\newtheorem{lemmalph}[theoremalph]{Lemma}
\newtheorem{propositionalph}[theoremalph]{Proposition}

\newtheorem{lemma}[theorem]{Lemma}

\theoremstyle{definition}
\newtheorem{definition}[theorem]{Definition}

\theoremstyle{remark}
\newtheorem{example}[theorem]{Example}
\newtheorem{remark}[theorem]{Remark}

\newcommand{\RV}{\textup{RV}}

\newcommand{\BL}{\op{BL}}
\newcommand{\NPT}{\op{NPT}}

\newcommand{\Diag}{\op{Diag}}

\title{Nonclassical Weyl Laws and\\ Connes' Integration for weak Lorentz Ideals, II}
\date{\today}

\author{Rapha\"el Ponge}
 \address{Department of Mathematics and Statistics, University of Ottawa, Canada}
 \curraddr{Department of Mathematics and Statistics, University of Wyoming, Laramie, WY, USA}
 \email{ponge.math@icloud.com}

\author{Yongqiang Tian}
 \address{School of Mathematics and Statistics, Central South University, Changsha, China}
 \email{tianyongqiang@csu.edu.cn}



\begin{document}
\begin{abstract}
This is the second in a series of papers on Connes' integration in weak Lorentz ideals. Building on the Dixmier trace theory, Birman--Solomyak perturbation theory, and strong measurability developed in Part 1, we establish a spectral form of Pietsch's correspondence for traces on these ideals, extending to this setting results of Semenov--Sukochev--Usachev--Zanin for the weak trace-class. The correspondence describes the positive normalized traces in terms of Banach limits and yields a complete spectral characterization of strong measurability. We also introduce hypermeasurability (measurability with respect to every normalized trace). Unlike the ambient ideal, it depends on the specific regularly varying function chosen, and we characterize it spectrally by means of eigenvalue sums. We further show that hypermeasurability and spectral measurability are incomparable. Finally, we apply these results to examples arising from nonclassical Weyl laws in the sense of Simon, including the logarithm of the Laplacian on a closed manifold, the double Laplacian, multi-tensor products of Laplacians, and an example arising from Connes' approach to the Riemann Hypothesis.
\end{abstract}
\maketitle


\section{Introduction}

This paper is the second of a series of papers whose main goal is the extension to weak Lorentz ideals of Connes' integration, along with a systematic description of traces on such ideals. 

In the first paper~\cite{PT:Part1} we developed Dixmier traces and Connes' integration in the general setting of weak Lorentz ideals $\sL_g$ (where $g$ is a regularly varying function of index $-1$), extended the Birman--Solomyak perturbation theory to this setting, and studied strong measurability. 

In the current paper, our first aim is to revisit strong measurability and obtain a spectral characterization thereof. This characterization is based on an elegant correspondence, due to Pietsch~\cite{Pi:IEOT17,Pi:AM20}, between traces on an operator ideal and shift-invariant linear functionals on a suitable sequence space. What we actually need here is a spectral version of this correspondence, which we establish explicitly in Section~\ref{sec:Pietsch}. 

Our second aim is to develop the theory of hypermeasurability and to apply these results to a systematic treatment of examples arising from nonclassical Weyl laws. Roughly speaking, hypermeasurability concerns whether an operator in $\sL_g$ admits a common value with respect to all $g$-normalized traces. Although equivalent regularly varying functions generate the same ideal $\sL_g$, the definition of hypermeasurability depends in a highly delicate way on the specific choice of the function $g$. For this reason, we shall undertake a detailed investigation of $g$-hypermeasurability and provide several criteria (see Section~\ref{sec:hypermeasurability} for more information). 

We also relate $g$-hypermeasurability to previously existing notions of measurability. A quick comparison of the definitions immediately yields a clean hierarchy relating Dixmier-measurability, strong measurability and $g$-hypermeasurability. What remains unclear is the relationship between $g$-hypermeasurability and spectral measurability. As we shall see, $g$-hypermeasurability and spectral measurability are in fact \emph{incomparable}. 
It turns out that $g$-hypermeasurability is a rather strong notion: some operators arising from nonclassical Weyl laws are not $g$-hypermeasurable, though they are $g_1$-hypermeasurable for a certain asymptotically equivalent correction $g_1$ of $g$.

We refer to~\cite{GS:JFA14, GU:JMAA20, LU:JMAA25, McD:arXiv26, SS:JFA13, TU:JMAA27} and the references therein for related results.

\subsection{Measurability in weak Lorentz ideals}
Throughout this paper $\sH$ is a separable Hilbert space and $\sK$ is the ideal of compact operators on $\sH$. We also let 
 $g:[0,\infty)\rightarrow(0,\infty)$ be a  function of regular variation of index $-1$ (i.e., an  $\RV_{-1}$-function in the notation of~\cite{PT:Part1}). We further assume that $\int_0^\infty g(t)\,dt=\infty$, and set $G(t)=\int_0^t g(s)\,ds$, $t\geq 0$. Recall that the weak Lorentz ideal associated with $g$ is 
\[
  \sL_g := \bigl\{T \in \sK;\; \mu_j(T) = \op{O}(g(j))\bigr\},
\]
where $\mu_0(T)\geq \mu_1(T)\geq \cdots$ are the singular values of $T$. This is a quasi-Banach ideal. For $g(t)=(t+1)^{-1}$ we recover the weak trace-class $\sL_{1,\infty}$. The closure of finite rank operators in $\sL_g$ is the separable ideal 
\[
(\sL_g)_0:=  \big\{T \in \sK;\; \mu_j(T) = \op{o}(g(j))\big\}.
\]

In~\cite{PT:Part1} we introduced three notions of measurability that extend to weak Lorentz ideals the corresponding notions for the weak trace-class $\sL_{1,\infty}$. 

The first notion, which is the weakest, is \emph{Dixmier-measurability} (or simply \emph{measurability}). An operator $T\in \sL_g$ is Dixmier-measurable if it takes on the same value on \emph{all}  Dixmier traces on $\sL_g$. The common value then is the \emph{noncommutative integral} of $T$ and is denoted $\gbint T$. This extends to $\sL_g$ the well-known notion of measurability of Connes for operators on $\sL_{1,\infty}$~\cite{Co:NCG}. 

The space of Dixmier-measurable operators $\sM(\sL_g)$ is a closed subspace of $\sL_g$ containing the commutator subspace $\Com(\sL_g)$ and all finite rank operators. 
It can be shown that 
\begin{equation*}
 \Big(\text{$T$ is measurable and}\ \gbint T=L\Big) \Longleftrightarrow \lim_{N\to\infty}\frac{1}{G(N)}\sum_{j<N}\lambda_j(T)=L, 
\end{equation*}
where $\lambda(T)=(\lambda_j(T))$ is any eigenvalue sequence for $T$.\footnote{By an eigenvalue sequence we shall mean that each eigenvalue $\lambda_j(T)$ is repeated according to multiplicity and $|\lambda_0(T)|\geq |\lambda_1(T)|\geq |\lambda_2(T)|\geq \cdots$.} This provides a spectral characterization of $\sM(\sL_g)$. 

Dixmier traces are instances of positive $g$-normalized traces. A trace $\varphi$ on $\sL_g$ is $g$-normalized if $\varphi(T_g)=1$ for any positive operator 
$T_g$ whose eigenvalues are the $g(j)$, $j \geq 0$. There are many positive $g$-normalized traces that are not Dixmier traces. Following~\cite{PT:Part1} we shall say that an operator $T\in \sL_g$ is \emph{strongly measurable} (or \emph{positively measurable}) if it takes on the same value on \emph{all} positive $g$-normalized traces. 

Any strongly measurable operator is Dixmier-measurable. The space $\sM_s(\sL_g)$ of strongly measurable operators is given by 
\begin{equation*}
 \sM_s\left(\sL_g\right)= \C T_g \oplus \overline{\Com\left(\sL_g\right)}. 
\end{equation*}
In particular, this is a closed subspace of $\sL_g$. In this paper we will obtain a spectral characterization of $\sM_s(\sL_g)$ (see Theorem~\ref{thm:Intro.strong measurable--ac} below). 

The third notion of measurability arises from the extension to all weak Lorentz ideals of the perturbation theory of Birman--Solomyak of eigenvalues of operators in weak Schatten classes (see~\cite{PT:Part1, Ro:MUS74} and \S~\ref{sec:prelim}). An operator $T=T^*\in \sL_g$ is called \emph{spectrally measurable} if 
\begin{equation*}
 \lim_{j\rightarrow \infty} g(j)^{-1} \lambda_j^+(T) \quad \text{and} \quad  \lim_{j\rightarrow \infty} g(j)^{-1} \lambda_j^-(T)\quad \text{both exist}, 
\end{equation*}
 where $\lambda_0^\pm(T) \geq \lambda_1^\pm(T) \geq \cdots $ are the positive/negative eigenvalues of $T$. If $T\in \sL_g$ is not selfadjoint, we say that $T$ is spectrally measurable if its real and imaginary parts are. 
 
The perturbation theory for eigenvalues of operators in $\sL_g$ ensures that spectrally measurable operators form a closed subset of $\sL_g$ which is invariant under perturbation by operators in $(\sL_g)_0$. Moreover, it can be shown that if $T=T^*\in \sL_g$ is spectrally measurable, then $T$ is strongly measurable, with 
 \begin{equation*}
 \gbint T =  \lim_{j\rightarrow \infty} g(j)^{-1} \lambda_j^+(T) -  \lim_{j\rightarrow \infty} g(j)^{-1} \lambda_j^{-}(T). 
\end{equation*}
This result extends to non-selfadjoint operators (see~\cite{PT:Part1}). It provides a further spectral characterization of the NC integral on $\sL_g$. We refer to~\cite[\S6]{PT:Part1} for various examples of spectrally measurable operators. 

To sum up we have a hierarchy of measurability properties, 
\begin{equation*}
 \text{spectral measurability} \Longrightarrow \text{strong measurability} \Longrightarrow \text{Dixmier measurability}. 
\end{equation*}
Neither of the converse implications holds. In particular, in this paper we will construct an example of a strongly measurable operator which is not spectrally measurable (see \S\ref{subsec:hyper-not-spec}). 

\subsection{Hypermeasurability}
From the point of view of operator theory, it is also natural to consider measurability with respect to the full class of $g$-normalized traces. To this end we further assume that
\begin{equation*}
 t^{-1}=\op{O}\left(g(t)\right) \qquad \text{as}\ t\rightarrow \infty.
\end{equation*}
This ensures that $\sL_{1,\infty}\subseteq \sL_g$ and every trace on $\sL_g$ annihilates $\sL_1$, and hence is singular (see Proposition~\ref{prop:Hyper.trace-singular}).

We shall say that an operator $T \in \sL_g$ is \emph{$g$-hypermeasurable} if it takes on the same value on all $g$-normalized traces. As we shall see, this notion depends on the choice of the function $g$. We denote by $\sM_h^g(\sL_g)$ the space of $g$-hypermeasurable operators. For $g(t)=(t+1)^{-1}$ this corresponds to the notion of universally measurable operators in $\sL_{1,\infty}$ considered by other authors (see, e.g.,~\cite{CRSZ:JST16, LSZ:Book, SZ:Ast23}). In addition, every $g$-hypermeasurable operator is strongly measurable, and hence is Dixmier-measurable.

We have the following spectral characterization of $g$-hypermeasurable operators (see Theorem~\ref{thm:hyper} for the full statement). This extends to weak Lorentz ideals the spectral characterization of universally measurable operators in $\sL_{1,\infty}$ in~\cite{CRSZ:JST16, LSZ:Book}.

\begin{theoremalph}
If $T\in \sL_g$, then
 \begin{equation*}
 \bigg( T\in \sM_h^g(\sL_g) \quad \text{and} \quad \gbint T=c\bigg) \Longleftrightarrow \bigg( \sum_{j < N}\lambda_j(T) = cG(N) + \op{O}\bigl(Ng(N)\bigr)\bigg).
\end{equation*}
\end{theoremalph}

As mentioned above, spectral measurability implies strong measurability. The relationship between spectral measurability and hypermeasurability is more involved. On the one hand, we have examples of $g$-hypermeasurable operators that are not spectrally measurable (see Proposition~\ref{prop:hyper non-weyl}). On the other hand, we have a whole class of spectrally measurable operators that are not $g$-hypermeasurable (see below). Therefore, these two notions of measurability are incomparable.

The relations between the various notions of measurability are summarized by the following diagram:

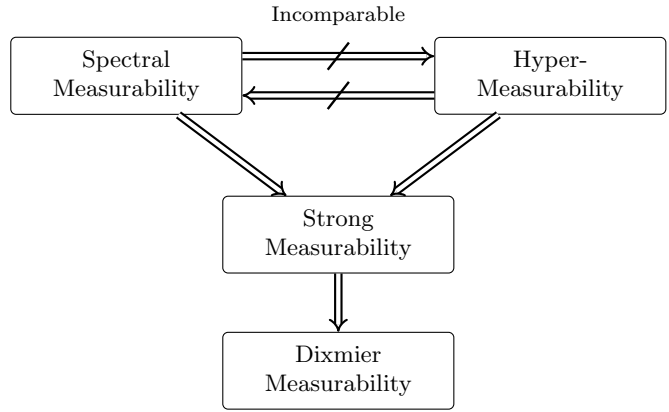
\begin{center}
\begin{tikzpicture}[
  box/.style={draw, rounded corners=2pt, align=center, inner sep=5pt,
              minimum height=10mm, text width=27mm, font=\small},
  imp/.style={-{Implies}, double, double distance=1.6pt, thick},
]
  \node[box] (spec)   at (0,0)      {Spectral\\ Measurability};
  \node[box] (hyper)  at (5.6,0)    {Hyper-\\ Measurability};
  \node[box] (strong) at (2.8,-2.1) {Strong\\ Measurability};
  \node[box] (dix)    at (2.8,-3.9) {Dixmier\\ Measurability};

  \draw[imp] (spec)   -- (strong);
  \draw[imp] (hyper)  -- (strong);
  \draw[imp] (strong) -- (dix);

  \coordinate (mid) at ($(spec.east)!0.5!(hyper.west)$);
  \draw[imp] ([yshift=2.6mm]spec.east)  -- ([yshift=2.6mm]hyper.west);
  \draw[imp] ([yshift=-2.6mm]hyper.west) -- ([yshift=-2.6mm]spec.east);
  \draw[line width=0.9pt] ($(mid)+(-1.4mm,0.8mm)$)  -- ($(mid)+(1.4mm,4.4mm)$);
  \draw[line width=0.9pt] ($(mid)+(-1.4mm,-4.4mm)$) -- ($(mid)+(1.4mm,-0.8mm)$);
  \node[font=\footnotesize] at ($(mid)+(0,8mm)$) {Incomparable};
\end{tikzpicture}
\end{center}

However, we have the following eigenvalue criterion for $g$-hypermeasurability.

\begin{propositionalph}[see Proposition~\ref{prop:hyper cond by asymptotic}] \label{prop:Intro.hyper cond by asymptotic}
 Let $T=T^*\in \sL_g$.  Assume there is a decreasing null function\footnote{By a \emph{null function} we mean a function $\varepsilon:[t_0,\infty)\rightarrow \R$ converging to $0$ as $t\rightarrow \infty$.} $\varepsilon:[t_0,\infty)\rightarrow (0,\infty)$ such that
\begin{equation*}
 \lambda_j^\pm(T)= g(j)\left(c^\pm + \op{O}\left( \varepsilon(j)\right)\right) \qquad \text{and} \qquad \int_{t_0}^t \varepsilon(s)g(s)ds=\op{O}\left(tg(t)\right).
\end{equation*}
Then the operator $T$ is $g$-hypermeasurable.
\end{propositionalph}

The following eigenvalue result provides an obstruction to $g$-hypermeasurability, and so it enables us to construct examples of spectrally measurable operators that are not $g$-hypermeasurable.

\begin{propositionalph}[see Proposition~\ref{prop:Stronger-meas.Weyk-non-hyper}] \label{prop:Intro.Stronger-meas.Weyk-non-hyper}
 Let $T\in \sL_g$, $T\geq 0$, and suppose there is a decreasing null function $\varepsilon:[t_0,\infty)\rightarrow (0,\infty)$ such that
  \begin{equation*}
 \lambda_j(T)=g(j)\big(c_0 \pm \varepsilon(j) +\op{o}\left( \varepsilon(j)\right)\big) \qquad \text{and} \qquad tg(t) =\op{o}\bigg(\int_{t_0}^t \varepsilon(s)g(s)ds\bigg).
\end{equation*}
Then the operator $T$ is not $g$-hypermeasurable.
\end{propositionalph}

\subsection{Pietsch's correspondence}
Another important approach to constructing traces on operator ideals is Pietsch's correspondence (see~\cite{Pi:IEOT17, Pi:AM20}; see also~\cite{LSZ:Book}). There are various incarnations of this correspondence. In any case we get a one-to-one correspondence between spaces of traces and spaces of shift-invariant linear functionals on suitable sequence spaces. It implies a very strong form of the spectral invariance of traces.

The beginning of Section~\ref{sec:Pietsch} is devoted to recalling the main results of~\cite{Pi:IEOT17, Pi:AM20}. We explain how they enable us to get a one-to-one correspondence between traces on $\sL_g$ and shift-invariant linear functionals on $\ell_\infty$ (see~Theorem~\ref{thm:Pietsch} for the precise statement). A salient feature of Pietsch's construction of traces is the use of suitable dyadic representations of operators in $\sL_g$ (compare~\cite{LSZ:Book}). We also refer to~\cite{LU:JMAA25} for a version of this correspondence in the semi-finite setting.

We seek a spectral reformulation of Pietsch's correspondence. By a well-known result of Dykema--Kalton~\cite{DK:Crelle98}, traces on quasi-Banach ideals are spectral, in the sense that their values depend only on the spectra of the operators at stake. We extend Pietsch's correspondence in terms of eigenvalue sequences. The correspondence is mediated by the \emph{weighted dyadic averaging operator} $\hat{\alpha}_g:\tilde{\ell}_g\rightarrow \ell_\infty$ defined by
\begin{equation*}
  \hat{\alpha}_g(x)_j= \frac{1}{2^jg(2^j)} \sum_{k=2^j-1}^{2^{j+1}-2} x_k, \qquad x=(x_j)_{j\geq 0}\in \tilde{\ell}_g, 
\end{equation*}
where $\tilde{\ell}_g$ consists of sequences $x=(x_j)_{j\geq 0}$ such that $|x_j|=\mu_j(x)=\op{O}(g(j))$ (see Section~\ref{sec:Pietsch}). In particular, a result of Dykema--Kalton~\cite{DK:Crelle98} shows that if $T\in \sL_g$, then every eigenvalue sequence for $T$ is in $\tilde{\ell}_g$. 

\begin{theoremalph}[see Theorem~\ref{thm:Pietsch.Spectral}] \label{thm:Intro.Pietsch}
 We have a one-to-one correspondence between traces on $\sL_g$ and shift-invariant linear functionals on $\ell_\infty$ as follows.
\begin{enumerate}
 \item Every shift-invariant linear functional $\ell$ on $\ell_\infty$ defines a trace on $\sL_g$ by
 \begin{equation*}
\varphi_\ell(T)= \ell\circ\hat{\alpha}_g\left(\lambda(T)\right), \qquad T\in \sL_g.
\end{equation*}

\item Conversely, for every trace $\varphi$ on $\sL_g$, there is a unique shift-invariant linear functional $\ell=\ell_\varphi$ (see~Eq.~(\ref{eq:Pietsch.ell-varphi})) such that
$\varphi=\varphi_\ell$.
\end{enumerate}
\end{theoremalph}

The correspondence can be specialized to positive and normalized positive traces (see Proposition~\ref{prop:Pietsch.positive-traces} and Corollary~\ref{cor:Pietsch.continuous-traces}). Positive traces correspond to positive multiples of Banach limits and continuous traces correspond to linear combinations of Banach limits. This appears to be new even for the weak trace-class  $\sL_{1,\infty}$.

\begin{theoremalph}[see Theorem~\ref{thm:Pietsch.PNT}]\label{thm:Intro.Pietsch-pn}
Under the correspondence provided by Theorem~\ref{thm:Intro.Pietsch}, the $g$-normalized positive traces on $\sL_g$ are exactly the traces of the form,
 \begin{equation*}
\hat{\varphi}_\theta(T):= \frac{1}{\log 2}\theta\left(\hat{\alpha}_g\left(\lambda(T)\right)\right), \qquad \qquad T\in \sL_g,
\end{equation*}
where $\theta$ ranges over all Banach limits on $\ell_\infty$.
\end{theoremalph}

Theorem~\ref{thm:Intro.Pietsch-pn} enables us to prove Theorem~\ref{thm:Intro.strong measurable--ac}. The arguments that lead to it further enable us to construct an explicit example of a $g$-hypermeasurable operator which is not spectrally measurable (see Proposition~\ref{prop:hyper non-weyl}).

Finally, under the above correspondence Dixmier traces correspond to $C_g$-factorizable Banach limits, where $C_g$ is a weighted version of the Ces\`aro operator introduced in~\cite{LU:JMAA25} (see Proposition~\ref{prop:Pietsch-Dixmier} for the precise statement).

\subsection{Spectral characterization of strong measurability}
Recall that a bounded sequence \emph{almost converges} to $c$ if every Banach limit takes the value $c$ on it. 

The following spectral characterization of strong measurability is proved in Section~\ref{sec:Pietsch} via Pietsch's correspondence for traces on $\sL_g$. It extends to weak Lorentz ideals the result of Semenov--Sukochev--Usachev--Zanin~\cite{SSUZ:AIM15} for $\sL_{1,\infty}$.

\begin{theoremalph}[see Theorem~\ref{thm:strong measurable--ac}]\label{thm:Intro.strong measurable--ac}
 If $T\in \sL_g$, then
 \begin{equation*}
 \bigg(T\in \sM_s(\sL_g) \quad \text{and} \quad \gbint T=c\bigg) \Longleftrightarrow \bigg( \hat{\alpha}_g(\lambda(T))_k \ \text{almost converges to}\ c \log 2\bigg).
\end{equation*}
\end{theoremalph}

As a consequence we see that if two $\sL_g$-operators $S$ and $T$ have the same non-zero spectrum up to multiplicity, then $S$ is strongly measurable if and only if $T$ is (see Corollary~\ref{cor:Strong.spectral-invariance}). 

\subsection{Examples arising from nonclassical Weyl laws}
In Section~\ref{sec:examples} we present several examples of hypermeasurable operators arising from \emph{sharp} nonclassical Weyl laws.

The general strategy is as follows. Suppose that $A$ is a selfadjoint operator with non-negative spectrum such that $0$ is isolated in the spectrum, and the positive part of the spectrum consists of eigenvalues with finite multiplicity. In~\cite{PT:Part1} it was shown that if $A$ satisfies a semiclassical Weyl law of the form
\begin{equation*}
 N(A;\lambda)\sim c \lambda^p (\log \lambda)^q, \qquad c>0, \quad p>0, \quad q\geq -1, 
\end{equation*}
and we put $g(t)=(t+1)^{-1}(\log (t+2))^q$, $t\geq 0$, then $A^{-p}$ is spectrally measurable, with 
\begin{equation*}
 \gbint A^{-p} = cp^{-q}.
\end{equation*}

This result is refined in Section~\ref{sec:examples} in the case of sharper Weyl laws of the form, 
\begin{equation}\label{eq:Intro.A-Weyl-law-remainder}
  N(A;\lambda)= c \lambda^p (\log \lambda)^q \left( 1+ \op{O}\big((\log \lambda)^{-\alpha}\big)\right), \qquad c>0,  \quad q\geq 0, \quad \alpha\geq 1.
\end{equation}

\begin{lemmalph}[see Lemma~\ref{lem:Examples.NAL-remainder1} and Lemma~\ref{lem:distribution asymptotic with remainder}]\label{lem:Intro.Ap-hypermeas}
 Assume that $A$ satisfies a Weyl law~(\ref{eq:Intro.A-Weyl-law-remainder}). The following hold.
 \begin{enumerate}
 \item[(i)] If $q=0$ and $\alpha>1$, then $A^{-p}$ is $(t+1)^{-1}$-hypermeasurable.

 \item[(ii)] If $q>0$ and $\alpha =1$, then $A^{-p}$ is not $g$-hypermeasurable, but it is hypermeasurable with respect to any (continuous) $\RV_{-1}$-function
 $g_1(t)$ such that
 \begin{equation}\label{eq:Intro.ex.g1}
 g_1(t)=\frac{1}{t}(\log t)^q\left(1-q^2\frac{\log\log t}{\log t} +  \op{O}\left((\log t)^{-1}\right)\right).
\end{equation}
\end{enumerate}
\end{lemmalph}

This general result enables us to deal with the following examples. 

\subsubsection{Riemann Hypothesis (RH)}
Assuming RH holds, let $\sD$ be the Dirac operator whose eigenvalues are the imaginary parts of the zeros of $\zeta(s)$ on the critical line $\Re s=1/2$ (see, e.g.,~\cite{Co:AFA24, CM:PNAS22}). The Riemann--von Mangoldt formula implies the nonclassical Weyl law,
\[
N(|\sD|;\lambda) = \frac{1}{\pi}\,\lambda\log\lambda \left(1 + \op{O}\left((\log\lambda)^{-1}\right)\right).
\]

Set $g(t)=(t+1)^{-1}\log(t+2)$, $t\geq 0$. It is shown in~\cite{PT:Part1} that $|\sD|^{-1}$ is spectrally measurable in $\sL_g$ with $\gbint |\sD|^{-1}=\pi^{-1}$. Using Lemma~\ref{lem:Intro.Ap-hypermeas} it can be further shown that $|\sD|^{-1}$ is not $g$-hypermeasurable, and is instead hypermeasurable with respect to any $\RV_{-1}$-function $g_1(t)$ of the form~(\ref{eq:Intro.ex.g1}) with $q=1$ (see Proposition~\ref{prop:RH}). 

\subsubsection{Logarithm of the Laplacian}
Given a closed Riemannian manifold $(M^n,g)$, set  $T = \log(\Delta_g)\,\Delta_g^{-n/2}$. This is a selfadjoint operator with discrete spectrum and at most finitely many negative eigenvalues. Moreover, the Weyl law for $\Delta_g$ (see, e.g., \cite{Ho:ActaM68, DG:IM75}) implies that
\[
\lambda_j^+(T) = \frac{2}{n}c(n)\Vol_g(M)\frac{\log j}{j}\Bigl(1 + \op{O}\bigl(j^{-1/n}\bigr)\Bigr), \qquad c(n):=(2\pi)^{-n}|\bB^n|.
\]
It follows (see Proposition~\ref{prop:log Laplacian}) that, if we set $h(t)= (t+1)^{-1}\log(t+2)$, $t\geq 0$, then $ \log(\Delta_g)\,\Delta_g^{-n/2}$  is spectrally measurable  and $h$-hypermeasurable in $\sL_h$, with
\[
\hbint \log(\Delta_g)\,\Delta_g^{-n/2}= \frac{2}{n}c(n)\Vol_g(M). 
\]

\subsubsection{Double Laplacian} 
Let $(M^n,g)$ be a closed Riemannian manifold with Laplacian $\Delta_g$. We equip the double manifold $M\times M$ with the product metric $g\otimes g$. On $M\times M$ we consider the double Laplacian, 
\begin{equation*}
 \Delta_g\otimes \Delta_g. 
\end{equation*}
This is a non-hypoelliptic selfadjoint differential operator of order $4$ with non-negative spectrum. It has an infinite-dimensional nullspace. The positive part of its spectrum is given by products of the positive eigenvalues of $\Delta_g$.

The product structure of the spectrum of $\Delta_g\otimes \Delta_g$ implies that its zeta function is the square of the zeta function of $\Delta_{g}$.  Together with the recent extension of Ikehara's Tauberian theorem of Pierce~\emph{et al.}~\cite{PTZ:arXiv25} this enables us to get a sharp nonclassical Weyl law for $\Delta_g\otimes \Delta_g$. Namely, for any $\alpha \in (0, \frac{1}{6})$, we have 
 \begin{equation}\label{eq:Intro.Weyl-Double}
 N\big(\Delta_g\otimes \Delta_g;\lambda\big) = b_1 \lambda^{\frac{n}{2}} \log \lambda + b_0 \lambda^{\frac{n}{2}} + \op{O}\big(\lambda^{\frac{n}{2}-\alpha}\big), 
\end{equation}
where $b_0$ is an explicit coefficient and $b_1=\frac{n}{2} c(n)^2\Vol_g(M)^2$. Therefore, if we set $h(t)=(t+1)^{-1}\log (t+2)$, $t\geq0$, then (see Proposition~\ref{prop:Examples.Double-Laplacian}) the following hold: 
\begin{itemize}
 \item The operator $(\Delta_g\otimes\Delta_g)^{-n/2}$ is spectrally measurable in $\sL_h$, with 
 \[
 \hbint ( \Delta_g\otimes \Delta_g)^{-n/2}=c(n)^2 \Vol_{g}(M)^2. 
 \]

\item It is not $h$-hypermeasurable, but it is hypermeasurable with respect to any  
 $\RV_{-1}$-function $h_1(t)$ of the form~(\ref{eq:Intro.ex.g1}) with $q=1$. 
\end{itemize}

The Weyl law~(\ref{eq:Intro.Weyl-Double}) is a refinement of a Weyl law of  Battisti~\cite{Ba:MZ12} in the special case of the double Laplacian. Battisti's result was the main inspiration behind this example. 

\subsubsection{Multi-tensor products of Laplacians}
Let $(M_i^{n_i},g_i)$, $i=1,\ldots,r$, be closed Riemannian manifolds with  $n=n_1\geq\cdots\geq n_r$ and $r\geq 2$. On $M_1\times \cdots \times M_r$ consider the tensor product of Laplacians,
\begin{equation*}
A:= \Delta_{g_1}\otimes \cdots \otimes \Delta_{g_r}.
\end{equation*}
This is a non-hypoelliptic differential operator of order $2r$. As with the double Laplacian, $A$ is essentially selfadjoint with a non-negative spectrum and an infinite-dimensional nullspace. The positive part of its spectrum is discrete and consists of the products of the eigenvalues of the Laplacians $\Delta_{g_i}$, $i=1,\ldots,r$.

As with the double Laplacian, we have a sharp nonclassical Weyl law. Namely, if we set $q:=\#\{i;\,n_i=n\}$, then, for any $\alpha \in (0, \frac{1}{2}(q+1)^{-1})$, we have 
\begin{equation}\label{eq:Intro.Examples.WL-multiL}
 N(A; \lambda) = \sum_{k=0}^{q-1} b_k \lambda^{\frac{n}{2}} (\log \lambda)^k + \op{O}\big( \lambda^{\frac{n}{2}-\alpha}\big), 
  \end{equation}
where the $b_k$, $k=0, \ldots, q-1$, are explicit coefficients, with
\[
b_{q-1} = \frac{1}{(q-1)!} \left(\frac{n}{2}\right)^{q-1}c(n)^q \Vol_{g_1}(M_1)\cdots \Vol_{g_q}(M_q). 
\]
Therefore, if we set $h(t)=(t+1)^{-1}(\log (t+2))^{q-1}$, $t\geq 0$, then (see Proposition~\ref{prop:tensor Laplacians}) the following hold: 
 \begin{itemize}
 	\item The operator $( \Delta_{g_1}\otimes \cdots \otimes \Delta_{g_r})^{-\frac{n}{2}}$ is spectrally measurable in $\sL_h$, with \[ 
	\hbint  \left(\Delta_{g_1}\otimes \cdots \otimes \Delta_{g_r}\right)^{-\frac{n}{2}}=\frac{1}{(q-1)!}c(n)^{q}\,\Vol_{g_1}(M_1)\cdots \Vol_{g_q}(M_q). 
	\]\smallskip 

          \item If $q=1$, then $( \Delta_{g_1}\otimes \cdots \otimes \Delta_{g_r})^{-\frac{n}{2}}$ is $(t+1)^{-1}$-hypermeasurable.\smallskip 

 	\item If $q\geq 2$, then $( \Delta_{g_1}\otimes \cdots \otimes \Delta_{g_r})^{-\frac{n}{2}}$ is not $h$-hypermeasurable, but it is hypermeasurable with respect to any  
 	$\RV_{-1}$-function $h_1(t)$ such that
 	\begin{equation*}
 		h_1(t)=\frac{1}{t}(\log t)^{q-1}\left(1-(q-1)^2\frac{\log\log t}{\log t} +  \op{O}\left((\log t)^{-1}\right)\right).
 	\end{equation*}
 \end{itemize}

Note that only the manifolds of maximum dimension $n$ contribute to the asymptotic~(\ref{eq:Intro.Examples.WL-multiL}) and the measurability properties of $(\Delta_{g_1}\otimes\cdots\otimes\Delta_{g_r})^{-n/2}$.

\subsection{Organization of the paper}
This paper is organized as follows. In Section~\ref{sec:prelim}, we recall the background and notation from~\cite{PT:Part1} that is needed in this paper.
In Section~\ref{sec:hypermeasurability} we study $g$-hypermeasurability and give a spectral characterization for this property. We also examine in more detail its relationship with spectral measurability.
In Section~\ref{sec:Pietsch} after deriving a spectral version of Pietsch's correspondence for traces on $\sL_g$, we specialize to $g$-normalized traces and positive $g$-normalized traces. As an application we exhibit an example of a $g$-hypermeasurable operator that is not spectrally measurable.  
In Section~\ref{sec:strong-hyper}, we use these results to give a spectral characterization of strongly measurable operators.
In Section~\ref{sec:examples} we describe the concrete examples arising from nonclassical Weyl laws mentioned above, including the Riemann Hypothesis example.
In Appendix~\ref{app:counting} we collect technical lemmas relating  sharp asymptotics for non-increasing non-negative null sequences to sharp asymptotics for their counting functions. Finally, for the reader's convenience, we include in Appendix~\ref{app:proof-Pietsch-Dixmier} an alternative proof of the characterization of Dixmier traces under Pietsch's correspondence of~\cite{LU:JMAA25}. 

\subsection*{Acknowledgements} This paper originated from discussions with Alexander Usachev. We would like to thank him warmly for sharing his insights with us. The first-named author also wishes to thank Magnus Goffeng, Nigel Higson, Galina Levitina, Edward McDonald, and Grigori Rozenblum for discussions related to the subject matter of this paper. His research was partially supported by NSFC grant No.~11971328 (China).

\section{Preliminaries}\label{sec:prelim}
In this section we recall the main definitions and results from the prequel paper~\cite{PT:Part1} that are used throughout this paper. We refer to~\cite{PT:Part1} for further details. 

\subsection{Functions of regular variation}
As in~\cite{PT:Part1} we shall say that a continuous function $g:[a,\infty)\rightarrow (0,\infty)$ has \emph{regular variation of index} $\rho\in \R$, or is $\RV_\rho$, if it is 
 ultimately monotonic,  and
\begin{equation}\label{cond:rv}
 \lim_{t\rightarrow \infty} \frac{g(\lambda t)} {g(t)}= \lambda^\rho \qquad \forall \lambda>0.
\end{equation}
Examples of $\RV_\rho$-functions include the functions, 
\begin{equation*}
 g(t)=(t+1)^{\rho}, \qquad g(t)=(t+1)^\rho\left(\log(t+2)\right)^q, \quad q\in \R^*. 
\end{equation*}

Note that~(\ref{cond:rv}) implies that 
\begin{equation*}
 \lim_{t\rightarrow \infty} t^{-\alpha} g(t)=\infty \quad \text{if $ \alpha <\rho$}, \qquad \textup{and}\qquad
 \lim_{t\rightarrow \infty} t^{-\alpha} g(t)=0 \quad \text{if $ \alpha >\rho$}.
\end{equation*}

In addition, by the Uniform Convergence Theorem (UCT) (see, e.g., \cite[Theorem~1.5.2]{BGT:Cambridge87}) the convergence~(\ref{cond:rv}) holds uniformly with respect to $\lambda$ as it varies within compact subsets of $(0,\infty)$. If $\rho>0$ (resp., $\rho<0$), we even have uniform convergence on each interval $(0,a]$ (resp., $[a,\infty)$). In what follows we will make frequent use of the following consequence of the UCT.

\begin{lemma}\label{lem:UCT-consequence}
 Let $g:[a,\infty)\rightarrow (0,\infty)$ be $\RV_\rho$, $\rho\in \R$.
 \begin{enumerate}
 \item If $t_\alpha \rightarrow \infty$ and $\lambda_\alpha \rightarrow \lambda>0$, then $g(\lambda_\alpha t_\alpha)\sim \lambda^\rho g(t_\alpha)$.

 \item If $\lim u(s)=\lim v(s)=\infty$ and $u(s)\sim v(s)$ for $s$ large, then $g(u(s))\sim g(v(s))$.
 \end{enumerate}
 \end{lemma}

In particular, this implies that 
 \begin{equation*}
 \lim_{t\rightarrow \infty} \frac{g(t+b)}{g(t)}=1 \qquad \forall b\in \R.
\end{equation*}

\subsection{Weak Lorentz ideals}
Throughout this paper $\sH$ is a separable Hilbert space, and $\sK$ denotes the ideal of compact operators on $\sH$. For $T\in \sK$ we let \[
\mu_0(T)\geq \mu_1(T)\geq \cdots\] be its sequence of \emph{singular values}, i.e., the eigenvalues of $|T|=(T^*T)^{1/2}$ arranged in non-increasing order and repeated according to multiplicity. 

Let $g:[0,\infty) \rightarrow (0,\infty)$ be an $\RV_\rho$-function with $\rho<0$. The \emph{weak Lorentz ideal} $\sL_g$ associated with $g$ is 
\begin{equation*}
\sL_g:= \left\{T \in \sK; \ \mu_j(T)= \op{O}\left(g(j)\right)\right\}.
\end{equation*}
This is a quasi-Banach ideal with respect to the quasi-norm, 
\begin{equation*}
 \|T\|_g := \sup_{j\geq 0} g(j)^{-1}\mu_j(T).
\end{equation*}
For $g(t)=(t+1)^{-1/p}$, $t\geq 0$, we recover the weak Schatten class $\sL_{p,\infty}$. 

The quasi-Banach ideal is not separable; the closure of finite rank operators in $\sL_g$ is the closed ideal, 
\begin{equation*}
 (\sL_{g})_0= \left\{T \in \sK; \ \mu_j(T)= \op{o}\left(g(j)\right)\right\} \subsetneq \sL_g.
\end{equation*}

\subsection{Inequalities for eigenvalues} 
From now on, we let  $g:[0,\infty) \rightarrow (0,\infty)$ be an $\RV_{-1}$-function such that $\int_0^\infty g(t)\,dt=\infty$. In addition, we set
\begin{equation*}
 G(t):= \int_0^t g(s)ds, \qquad t\geq 0. 
\end{equation*}
By Karamata's theorem (see, e.g., {\cite{BGT:Cambridge87, BIKS:Springer18}}), the $\RV_{-1}$-property implies that
\begin{equation}\label{eq:Karamata}
 tg(t)=\op{o}(G(t)) \qquad \text{as}\ t\longrightarrow \infty.  
\end{equation}
This is the main property that is needed to extend Connes' integration beyond the weak trace-class $\sL_{1,\infty}$ (which is the ideal $\sL_g$ with $g(t)=(t+1)^{-1}$). 

Given an operator $T \in \sL_g$ we shall denote by $\lambda(T)=(\lambda_j(T))_{j\geq 0}$ any eigenvalue sequence, where each eigenvalue is repeated according to its (algebraic) multiplicity with the convention that
\begin{equation*}
 |\lambda_0(T)|\geq |\lambda_1(T)|\geq  |\lambda_2(T)|\geq \cdots \geq 0.
\end{equation*}
Such a sequence need not be unique, but the results of this paper do not depend on the choice of a specific sequence. 

The main impetus for the approach of~\cite{PT:Part1} and this paper is the following asymptotic additivity lemma. 

\begin{lemma}[\cite{LSZ:Book, PT:Part1}]  
 For all $S,T\in \sL_g$, we have
\begin{equation}\label{eq:prelim.asymptotic-additivity}
 \sum_{j<N}\lambda_j(S+T)=\sum_{j<N}\lambda_j(S)+\sum_{j<N}\lambda_j(T)+\op{o}(G(N)).
\end{equation}
\end{lemma}

This result follows by combining a general asymptotic additivity result of~\cite{LSZ:Book} for quasi-Banach ideals with Karamata's theorem~(\ref{eq:Karamata}) (see~\cite{PT:Part1}).

By definition the commutator space $\Com(\sL_g)$ is spanned by commutators $[A,T]$ with $A\in \sL(\sH)$ and $T\in \sL_g$. For weak Lorentz ideals, the spectral characterization of the commutator spaces of quasi-Banach ideals of Kalton~\cite{Ka:Crelle98} gives
\begin{equation}\label{eq:spectral1}
 \Com\left(\sL_g\right)= \bigg\{T\in \sL_g; \ \sum_{j<N} \lambda_j(T)= \op{O}(Ng(N))\bigg\}. 
\end{equation}
More generally (see~\cite{LSZ:Book}), we have
\begin{equation}\label{eq:com difference}
S-T \in \Com\left(\sL_g\right) \ \Longleftrightarrow \  \sum_{j<N} \lambda_j(S)=\sum_{j<N} \lambda_j(T)+ \op{O}(Ng(N)).
\end{equation}

\subsection{Dixmier-measurability} 
An extended limit is any positive linear functional on $\ell_\infty$ which is an extension of the limit, i.e., a state that annihilates the null sequence ideal $\co$. In particular, given any sequence $a=(a_j)_{j\geq 0}$ in $\ell_\infty$, we have
\begin{equation*}
 \lim_{j\rightarrow \infty} a_j =L \ \Longleftrightarrow \ \big(\omega(a)=L \ \text{for every extended limit}\ \omega\big). 
\end{equation*}
 
If $\omega$ is any extended limit, then~(\ref{eq:prelim.asymptotic-additivity}) implies that we obtain a positive linear trace on $\sL_g$ by letting
\begin{equation*}
 \Tr_\omega(T):=\omega\Big(\Big\{\tfrac{1}{G(N)}\sum_{j<N}\lambda_j(T)\Big\}_{N\geq 1}\Big), \qquad T\in \sL_g. 
\end{equation*}
We call $\Trw$ the \emph{Dixmier trace} on $\sL_g$ associated with $\omega$. 

An operator $T\in \sL_g$ is called \emph{Dixmier-measurable} (or simply \emph{measurable}) if the value of $\Trw(T)$ does not depend on $\omega$. This unique value is the \emph{noncommutative integral of $T$} and is denoted $\gbint T$. 
The space of Dixmier-measurable operators is denoted $\sM(\sL_g)$. This is a closed subspace of $\sL_g$ containing $\Com(\sL_g)$ and $(\sL_g)_0$ (see~\cite{PT:Part1}). 

In~\cite{PT:Part1} the following spectral characterization of the NC integral is established: given any $T\in \sL_g$, we have
\begin{equation*}
 \Big(\text{$T$ is measurable and}\ \gbint T=L\Big) \Longleftrightarrow \lim_{N\to\infty}\frac{1}{G(N)}\sum_{j<N}\lambda_j(T)=L. 
\end{equation*}

\subsection{Strong measurability}
In what follows, we let $T_g$ be any operator for which there is an orthonormal eigenbasis $\xi_j$, $j\geq 0$, such that $T_g \xi_j=g(j)\xi_j$ for all $j\geq 0$. 
A trace $\varphi$ on $\sL_g$ is called \emph{$g$-normalized} if $\varphi(T_g)=1$. Dixmier traces are instances of such traces. 

An operator $T\in \sL_g$ is  called \emph{strongly measurable} if all $g$-normalized positive traces take the same value on $T$. Such an operator is automatically Dixmier-measurable. We denote by $\sM_s(\sL_g)$ the space of strongly measurable operators in $\sL_g$. As shown in~\cite{PT:Part1}, we then have
\begin{equation*}
 \sM_s(\sL_g)=\C T_g\oplus \overline{\Com(\sL_g)}.
\end{equation*}
A spectral characterization of strong measurability is given in Section~\ref{sec:strong-hyper}. It extends to all weak Lorentz ideals the spectral characterization of strong measurability in $\sL_{1,\infty}$ of Semenov--Sukochev--Usachev--Zanin~\cite{SSUZ:AIM15}. 

\subsection{Spectral measurability} 
If $T=T^*\in \sK$, then we let 
\begin{equation*}
 \lambda_0^\pm(T)\geq \lambda_1^\pm(T)\geq \cdots 
\end{equation*}
be its sequences of positive and negative eigenvalues, i.e., $\lambda_j^\pm =\lambda_j(T^\pm)$, where $T^\pm=\frac{1}{2}(|T|\pm T)$. 

In~\cite{PT:Part1} the perturbation theory of Birman--Solomyak~\cite{BS:JFAA70, BS:Book} for positive/negative eigenvalues and singular values of operators in weak Schatten classes is extended to all weak Lorentz ideals (see also~\cite{Ro:MUS74}). In particular, if $T=T^*\in \sL_g$ and $(T_\ell)_{\ell\geq 0}$ is a selfadjoint family in $\sL_g$ such that $\lim g(j)^{-1}\lambda_j^\pm(T_\ell)$ exist and $T_\ell \rightarrow T$ in $\sL_g/(\sL_g)_0$, then 
\begin{equation*}
 \lim_{j \rightarrow \infty} g(j)^{-1}\lambda_j^\pm(T) = \lim_{\ell \rightarrow \infty} \big( \lim_{j \rightarrow \infty}  g(j)^{-1} \lambda_j^\pm(T_\ell) \big). 
\end{equation*}
There is a similar perturbation result for singular values (see~\cite{PT:Part1}). 

An operator $T=T^*\in \sL_g$ is called \emph{spectrally measurable} if $\lim g(j)^{-1}\lambda_j^\pm(T)$ both exist. More generally, we say that a non-selfadjoint operator $T\in \sL_g$ is spectrally measurable if its real and imaginary parts are. It is shown in~\cite{PT:Part1} that if $T=T^*\in \sL_g$ is spectrally measurable, then $T$ is strongly measurable, with 
\begin{equation*}
 \gbint T =  \lim_{j \rightarrow \infty} g(j)^{-1}\lambda_j^+(T) -  \lim_{j \rightarrow \infty} g(j)^{-1}\lambda_j^{-}(T). 
\end{equation*}
There is a version of this result in the non-selfadjoint case (see~\cite{PT:Part1}). This provides a further spectral interpretation of the NC integral. Moreover, it follows from the perturbation theory of~\cite{PT:Part1} that the set of spectrally measurable operators is a closed subset of $\sL_g$ which is invariant under perturbations by operators in $(\sL_g)_0$. 

In this article, we will exhibit an example of a strongly measurable operator which is not spectrally measurable (see~\S\ref{subsec:hyper-not-spec}). Therefore, spectral measurability is a much stronger notion of measurability than strong measurability. 

We refer to~\cite{PT:Part1} for various examples of spectrally measurable operators arising from nonclassical Weyl laws (see also~\cite{GS:JFA14, GU:JMAA20, LU:JMAA25, McD:arXiv26, TU:JMAA27} for related examples). Further examples are presented in Section~\ref{sec:examples} of this paper. 

\section{Hypermeasurability}\label{sec:hypermeasurability}
From the point of view of operator theory it also makes sense to consider a version of measurability with respect to the class of all $g$-normalized traces, not just Dixmier traces or positive normalized traces as in the previous section. In this section, we look at this notion and compare it to spectral measurability. Even in the case of the weak trace-class $\sL_{1,\infty}$ this comparison seems to be new.

Throughout this section we further require that
\begin{equation}\label{cond:minimal}
 t^{-1}=\op{O}\left(g(t)\right) \qquad \text{as}\ t\rightarrow \infty.
\end{equation}

\subsection{$g$-Hypermeasurable operators}
The condition~(\ref{cond:minimal}) ensures that $\sL_{1,\infty}\subseteq \sL_g$. Moreover, we have the following result.

\begin{proposition}\label{prop:Hyper.trace-singular}
 If~(\ref{cond:minimal}) holds, then every trace on $\sL_g$ annihilates the trace-class $\sL_1$. In particular, every trace on $\sL_g$ is singular.
\end{proposition}
\begin{proof}
The condition~(\ref{cond:minimal}) ensures that $1=\op{O}(tg(t))$. Thus, if $T\in \sL_1$, then~(\ref{cond:minimal}) ensures that
\begin{equation*}
 \sum_{j<N}\lambda_j(T) =\op{O}(1)=\op{O}(Ng(N)).
\end{equation*}
It then follows from~(\ref{eq:spectral1}) that $T\in \Com(\sL_g)$, and so $T$ is annihilated by all traces on $\sL_g$. This gives the result.
\end{proof}

\begin{definition}\label{def:hyper}
 We say that an operator $T\in \sL_g$ is \emph{$g$-hypermeasurable} if it takes on the same value on all $g$-normalized traces.
 \end{definition}

We shall denote by $\sM_h^g(\sL_g)$ the space of operators in $\sL_g$ that are $g$-hypermeasurable.

\begin{remark}
 For $g(t)=(t+1)^{-1}$, in which case $\sL_g$ is the weak trace-class, $(t+1)^{-1}$-hypermeasurability is also called universal measurability by some authors
(see, e.g.,  \cite{CRSZ:JST16, LSZ:Book, SZ:Ast23}). We refrain from using the same terminology for weak Lorentz ideals, since the property is sensitive to the choice of the function $g$ to define the ideal $\sL_g$. Namely, two equivalent $\RV_{-1}$-functions $g(t)$ and $g_1(t)$ define the same weak Lorentz ideal,
but a $g$-hypermeasurable operator need not be $g_1$-hypermeasurable (see Proposition~\ref{prop:dependence} below). Therefore, hypermeasurability does not seem
to be a universal property.
\end{remark}

\begin{remark}\label{rmk:hyper implies strong}
 If $T$ is $g$-hypermeasurable, then it is strongly measurable, and so $T$ is measurable. Thus, as noted in the previous section, for every normalized trace $\varphi:\sL_g\rightarrow \C$, we have
 \begin{equation*}
 \varphi(T) =\gbint T.
\end{equation*}
\end{remark}

\subsection{Spectral characterization}
We have the following spectral criterion for $g$-hypermeasurability. It extends to weak Lorentz ideals the spectral characterization of universal measurability for the weak trace-class $\sL_{1,\infty}$ in~\cite{CRSZ:JST16, LSZ:Book}.

\begin{theorem}\label{thm:hyper}
  Let $T\in \sL_g$. The following are equivalent:
 \begin{enumerate}
 \item[(i)] $T$ is $g$-hypermeasurable and $\gbint T=c$.\smallskip

 \item[(ii)] $T=cT_{g} \ \bmod \Com(\sL_g)$.\smallskip

 \item[(iii)] We have
  \begin{equation*}
 \sum_{j<N} \lambda_j(T) = cG(N) +\op{O}(Ng(N)).
\end{equation*}
\end{enumerate}
In particular, we have
\begin{equation*}
 \sM_h^g(\sL_g)=\C T_{g} \oplus \Com(\sL_g).
\end{equation*}
\end{theorem}

\begin{proof} Using~(\ref{cond:minimal}), we get
\begin{equation*}
 \sum_{j<N} \lambda_j(T_{g})=\sum_{j<N}g(j)=G(N)+\op{O}(1)=G(N)+\op{O}(Ng(N)).
\end{equation*}
Therefore, we see that~(iii) is equivalent to
\begin{equation*}
  \sum_{j<N} \lambda_j(T) = \sum_{j<N} \lambda_j(cT_{g})+\op{O}(Ng(N)),
\end{equation*}
which is equivalent to (ii) by~(\ref{eq:com difference}). 

It is immediate that (ii)$\Rightarrow$(i). It remains to show that (i)$\Rightarrow$(ii). Suppose that (ii) does not hold. If $T\in \C T_{g} \oplus \Com(\sL_g)$, then $T=c_1T_{g} \bmod \Com(\sL_g)$ with $c_1\neq c$, and so $T$ is $g$-hypermeasurable with $\gbint T =c_1\neq c$, meaning (i) fails. If  $T\not \in \C T_{g} \oplus \Com(\sL_g)$, then, for $i=0,1$,  we can find linear forms $\varphi_i:\sL_g \rightarrow \C$ such that
\begin{equation*}
 \varphi_i(T)=i, \qquad \varphi_i(T_{g})=1, \qquad \varphi_{i}=0 \quad \text{on}\ \Com(\sL_g).
\end{equation*}
We thus have two $g$-normalized traces such that $\varphi_0(T)\neq \varphi_1(T)$, and so $T$ is not $g$-hypermeasurable. Therefore, we see that in either case (i) fails.  By contraposition, it follows that (i)$\Rightarrow$(ii). The proof is complete.
 \end{proof}

The characterization of $g$-hypermeasurable operators in terms of eigenvalue sequences provided by Theorem~\ref{thm:hyper} implies the following spectral invariance of $g$-hypermeasurability.

\begin{corollary}\label{cor:Hyper.spectral-invariance}
If  $S\in \sL_g(\sH)$ and $T\in \sL_g(\sH')$ have the same non-zero spectrum up to multiplicity, then $S$ is $g$-hypermeasurable if and only if $T$ is $g$-hypermeasurable. \end{corollary}

In the special case where $\sH'$ is the Hilbert space $\sH$ with an equivalent inner product we obtain the following result.

\begin{corollary}
 The space $\sM_h^g(\sL_g(\sH))$ does not depend on the inner product of $\sH$.
\end{corollary}

\subsection{Hypermeasurability vs.\ spectral measurability}
We know from~\cite[Theorem~5.8]{PT:Part1} that spectral measurability implies strong measurability. The relationship between spectral measurability and $g$-hypermeasurability is much less transparent.

On the one hand, there are examples of spectrally measurable operators that are not $g$-hypermeasurable (see Proposition~\ref{prop:Stronger-meas.Weyk-non-hyper} below). On the other hand, in Section~\ref{sec:Pietsch} we will exhibit an example of a $g$-hypermeasurable operator that is not spectrally measurable (see Proposition~\ref{prop:hyper non-weyl}). Therefore, those two notions of measurability are incomparable. 

Nevertheless, as we shall now see, if Weyl laws hold with a remainder term of sufficiently lower order, then we obtain $g$-hypermeasurability.

\begin{proposition}\label{prop:hyper cond by asymptotic}
Let $T=T^*\in \sL_g$ be such that
\begin{equation}\label{cond:hyper with remainder h}
 \lambda_j^\pm(T)= g(j)\left(c^\pm + \op{O}\left( \varepsilon(j)\right)\right),
\end{equation}
where $\varepsilon:[t_0,\infty)\rightarrow (0,\infty)$ is a decreasing null function such that
\begin{equation}\label{cond: h}
\int_{t_0}^t \varepsilon(s)g(s)ds=\op{O}\left(tg(t)\right).
\end{equation}
 Then $T$ is $g$-hypermeasurable.
\end{proposition}
\begin{proof}
By linearity we may assume that $T\geq 0$. Set $c=c^+$. We know from ~\cite[Theorem~5.8]{PT:Part1} that $T$ is strongly measurable and $\gbint T=c$. Therefore, by Theorem~\ref{thm:hyper} we have
\begin{equation*}
 T\ \text{is $g$-hypermeasurable}\ \Longleftrightarrow \ \sum_{j<N} \lambda_j(T)=cG(N)+\op{O}\left(Ng(N)\right).
\end{equation*}
Set $a_j =\lambda_j(T)-cg(j)$. It follows from~(\ref{cond:minimal}) that
\begin{equation*}
 \sum_{j<N} \lambda_j(T)=c\sum_{j<N}g(j) + \sum_{j<N}a_j = cG(N)+ \sum_{j<N}a_j +\op{O}\left(Ng(N)\right).
\end{equation*}
Thus,
\begin{equation*}
 T\ \text{is $g$-hypermeasurable}\ \Longleftrightarrow \ \sum_{j<N} a_j=\op{O}\left(Ng(N)\right).
\end{equation*}

The asymptotic~(\ref{cond:hyper with remainder h}) implies that
\begin{equation*}
\sum_{j<N}a_j= \op{O}\big(\sum_{t_0\leq j<N}\varepsilon(j)g(j)\big).
\end{equation*}
By using~(\ref{cond: h}) and the fact that $1=\op{O}(tg(t))$ we get
\begin{equation*}
 \sum_{t_0\leq j<N}\varepsilon(j)g(j)= \int_{t_0}^N \varepsilon(s)g(s)ds +\op{O}(1)=\op{O}\left(Ng(N)\right).
\end{equation*}
This gives the result.
\end{proof}

\begin{corollary}
 If $T\in \sL_g$ is such that
\begin{equation*}
 \mu_j(T)= g(j)\left(c + \op{O}\left( \varepsilon(j)\right)\right),
\end{equation*}
where $\varepsilon:[t_0,\infty)\rightarrow (0,\infty)$ is a decreasing null function satisfying~(\ref{cond: h}), then $|T|$ is $g$-hypermeasurable.
\end{corollary}

\begin{remark}\label{rmk:Hyper.RV-neg-eps-hyper}
 If $\varepsilon(t)$ is $\RV_{\rho}$ with $\rho<0$, then $\varepsilon(t)$ is always a null function that satisfies~(\ref{cond: h}). Indeed, as $\varepsilon(t)g(t)$ is $\RV_{-1+\rho}$, 
 by using~(\ref{cond:minimal}) we get 
 \begin{equation*}
 \int_0^t \varepsilon(s)g(s)ds = \op{O}(1)= \op{O}(tg(t)).
\end{equation*}
In particular, the condition~(\ref{cond: h}) is always satisfied by the power functions $\varepsilon(t)=t^{-\alpha}$ with $\alpha>1$.
\end{remark}

\begin{remark}\label{rmk:2nd h}
 The condition~(\ref{cond: h}) may also be satisfied by some $\RV_0$-functions. Let $\varphi:[G(t_0),\infty)\rightarrow (0,\infty)$ be a monotonic $C^1$-function such that $|\varphi'(t)|$ is a decreasing null function, and
 \begin{equation*}
 \varphi(G(t))=\op{O}(tg(t)).
\end{equation*}
If we set $\varepsilon(t)=|\varphi'(G(t))|$, $t\geq t_0$, then $\varepsilon(t)$ is a decreasing null function, and we have
\begin{equation*}
 \int_{t_0}^t \varepsilon(s)g(s)ds= \pm \int_{t_0}^t \varphi'(G(s))g(s)ds= \pm \int_{G(t_0)}^{G(t)} \varphi(u)du=\op{O}\left(\varphi(G(t))\right)=\op{O}(tg(t)).
\end{equation*}
That is, the condition~(\ref{cond: h}) holds.
\end{remark}

Proposition~\ref{prop:hyper cond by asymptotic} shows that the remainder term in~(\ref{cond:hyper with remainder h}) need not be of very low order to have $g$-hypermeasurability. Nevertheless, as the following shows, if the order of the remainder term is not low enough, then $g$-hypermeasurability may fail.

\begin{proposition}\label{prop:Stronger-meas.Weyk-non-hyper}
 Let $T\in \sL_g$, $T\geq 0$, be such that
 \begin{equation}\label{cond:non-hyper-remainder}
 \lambda_j(T)=g(j)\big(c_0 + c_1 \varepsilon(j) +\op{o}\left( \varepsilon(j)\right)\big),  \qquad c_1\neq 0,
\end{equation}
where $\varepsilon:[t_0,\infty)\rightarrow (0,\infty)$ is a decreasing null function such that
\begin{equation}\label{cond: non-h}
tg(t)  =\op{o}\bigg( \int_{t_0}^t \varepsilon(s)g(s)ds\bigg).
\end{equation}
 Then $T$ is not $g$-hypermeasurable.
\end{proposition}
\begin{proof}
 In the same way as in the proof of Proposition~\ref{prop:hyper cond by asymptotic}, if we set $a_j=\lambda_j(T)-c_0g(j)$, $j\geq 0$, then we have
\begin{equation*}
 T\ \text{is $g$-hypermeasurable}\ \Longleftrightarrow \ \sum_{j<N} a_j=\op{O}\left(Ng(N)\right).
\end{equation*}
It follows from~(\ref{cond:non-hyper-remainder}) that
\begin{equation*}
\bigg| \sum_{j<N} a_j\bigg|  \sim  \sum_{j<N} |c_1|\varepsilon(j)g(j) =|c_1| \int_{t_0}^N \varepsilon(s)g(s)ds +\op{O}(1).
\end{equation*}
The condition~(\ref{cond: non-h}) then ensures that $ \sum_{j<N} a_j$ is not $\op{O}(Ng(N))$, and so $T$ is not $g$-hypermeasurable. The proof is complete.
\end{proof}

\begin{remark}\label{rmk:non-hyper}
 Let $\varphi:[G(t_0),\infty)\rightarrow (0,\infty)$ be a monotonic $C^1$-function such that $|\varphi'(t)|$ is a decreasing null function, and
 \begin{equation}\label{eq:non-hyper-condition}
\lim_{t\rightarrow \infty} \frac{ \varphi(G(t))}{tg(t)}=\infty.
\end{equation}
As in Remark~\ref{rmk:2nd h}, if we set $\varepsilon(t)=|\varphi'(G(t))|$, $t\geq t_0$, then $\varepsilon(t)$ is a decreasing null function, and we have
\begin{equation*}
 \frac{\pm 1}{tg(t)}\int_{t_0}^t \varepsilon(s)g(s)ds=   \frac{1}{tg(t)} \int_{G(t_0)}^{G(t)} \varphi(u)du\sim \frac{\varphi(G(t))}{tg(t)}\longrightarrow\infty.
\end{equation*}
That is, the condition~(\ref{cond: non-h}) holds.
 \end{remark}

It is worth specializing Proposition~\ref{prop:hyper cond by asymptotic}  and Proposition~\ref{prop:Stronger-meas.Weyk-non-hyper} to the case $g(t)\sim t^{-1}$ (in which case $\sL_g$ agrees with the weak trace-class $\sL_{1,\infty}$) and to the case $g(t)\sim t^{-1}(\log t)^q$, $q>0$.

\begin{corollary}\label{cor:Hyper.Weyl-hyper.sLp}
 Suppose that $g(t)\sim t^{-1}$.
 \begin{enumerate}
 \item If $T=T^*\in \sL_{1,\infty}$ is such that
 \begin{equation*}
 \lambda_j^\pm(T) = g(j)\bigg(c^{\pm} + \op{O}\left(\frac{1}{(\log j)^\alpha}\right) \bigg), \qquad \alpha>1,
\end{equation*}
then $T$ is $g$-hypermeasurable.

\item If $0\leq T\in \sL_{1,\infty}$ is such that
   \begin{equation*}
 \lambda_j(T) = g(j)\bigg(c_0 + \frac{c_1}{\log j}+ \op{o}\left(\frac{1}{\log j}\right) \bigg), \qquad c_1\neq 0,
\end{equation*}
then $T$ is not $g$-hypermeasurable.
\end{enumerate}
\end{corollary}
\begin{proof}
If $g(t)\sim t^{-1}$, then $tg(t)\sim 1$ and $G(t)\sim \log t$. Therefore, the conditions of Remark~\ref{rmk:2nd h} are satisfied by any power function $\varphi(t)=t^{-\alpha}$, $\alpha>0$, in which case $|\varphi'(G(t))|\sim (\log t)^{-(1+\alpha)}$. Proposition~\ref{prop:hyper cond by asymptotic} then gives the first part.

Moreover, the condition~(\ref{cond: non-h}) is satisfied by $\varphi(t)=\log t$. In this case $\varphi'(G(t))=(\log (t+1))^{-1}\sim (\log t)^{-1}$. We then get the second part by applying Proposition~\ref{prop:Stronger-meas.Weyk-non-hyper}. The proof is complete.
 \end{proof}

\begin{corollary}\label{cor:Hyper.Weyl-hyper.sLpq}
 Suppose that  $g(t)\sim t^{-1}(\log t)^q$, $q>0$.
 \begin{enumerate}
 \item If $T=T^*\in \sL_g$ is such that
   \begin{equation*}
 \lambda_j^\pm(T) =g(j)\bigg(c^{\pm} + \op{O}\left(\frac{1}{\log j}\right) \bigg),
\end{equation*}
then $T$ is $g$-hypermeasurable.

\item If $0\leq T\in \sL_g$ is such that
  \begin{equation*}
 \lambda_j(T) = \frac{(\log j)^q}{j}\bigg(c_0 + \frac{c_1}{(\log j)^\alpha}+ \op{o}\left(\frac{1}{(\log j)^\alpha}\right) \bigg), \qquad c_1\neq 0,  \quad 0<\alpha<1,
\end{equation*}
then $T$ is not $g$-hypermeasurable.
\end{enumerate}
\end{corollary}
\begin{proof}
 If $g(t)\sim t^{-1}(\log t)^q$, $q>0$, then $tg(t)\sim (\log t)^q$ and $G(t)\sim (q+1)^{-1}(\log t)^{q+1}$. Thus, the conditions of Remark~\ref{rmk:2nd h} are satisfied by
the power function $\varphi(t)=t^{\frac{q}{q+1}}$. In this case $\varphi'(t)=c_q t^{-\frac{1}{q+1}}$ and $\varphi'(G(t))\sim c_q'(\log t)^{-1}$, where $c_q$ and $c'_q$ are some positive constants. We then can apply Proposition~\ref{prop:hyper cond by asymptotic} to get the first part.

Moreover, if $\varphi(t)=t^\beta$, then $\varphi'(t)$ is a decreasing null function if and only if $\beta<1$. In addition, $\frac{1}{tg(t)}\varphi(G(t))\sim c_{q\beta}(\log t)^{\beta(q+1)-q}$, where $c_{q\beta}$ is some positive constant, and so the condition~(\ref{eq:non-hyper-condition}) is satisfied if and only if $\beta> q(q+1)^{-1}$. It follows that the assumptions of Remark~\ref{rmk:non-hyper} are satisfied if and only if $\beta\in (q(q+1)^{-1},1)$. This then ensures that the function $\varepsilon(t)=|\varphi'(G(t))|$ satisfies~(\ref{cond: non-h}).  

Note that $\varphi'(G(t))\sim c_{q\beta}'(\log t)^{-\alpha}$, where  $\alpha:=(1-\beta)(q+1)$ and $c_{q\beta}'$ is some positive constant. As $\beta$ ranges over $(q(q+1)^{-1},1)$, the exponent $\alpha$ ranges over $(0,1)$. We then get the second part by applying Proposition~\ref{prop:Stronger-meas.Weyk-non-hyper} with $\varepsilon(t)$ as above. The proof is complete.
\end{proof}

\subsection{Dependence on $g(t)$}
We end this section by looking at the dependence of $\sM_h^g(\sL_g)$ on the choice of $g$. In what follows we let $g_1:[0,\infty)\rightarrow (0,\infty)$ be a continuous
$\RV_{-1}$-function of the form,
\begin{equation*}
 g_1(t)=g(t)(1+\varepsilon(t)), \qquad t\geq 0,
\end{equation*}
where $\varepsilon:[0,\infty)\rightarrow (0,\infty)$ is a continuous decreasing null function. In particular, as $g_1(t)\sim g(t)$ the ideals $\sL_g$ and $\sL_{g_1}$ agree.

The following shows that $g$-hypermeasurability is sensitive to the choice of the function $g(t)$.

\begin{proposition}\label{prop:dependence}
 The following hold.
 \begin{enumerate}
\item If $\varepsilon(t)$ satisfies~(\ref{cond: h}), then the normalization conditions $\varphi(T_g)=1$ and $\varphi(T_{g_1})=1$ are equivalent, and so
$\sM_h^{g_1}(\sL_g)= \sM_h^{g}(\sL_g)$.

\item  If $\varepsilon(t)$ satisfies~(\ref{cond: non-h}), then $T_{g_1}$ (resp., $T_g$) is not $g$-hypermeasurable (resp., $g_1$-hypermeasurable), and so $\sM_h^{g_1}(\sL_g)\neq \sM_h^g(\sL_g)$.
\end{enumerate}
\end{proposition}
\begin{proof}
 As for $j$ large enough we have $\lambda_j(T_{g_1})=g_1(j)=g(j)(1+\varepsilon(j))$, we see that
 \begin{equation*}
 \lambda_j(T_{g_1})=g(j)\left(1+\varepsilon(j)+\op{o}\left(\varepsilon(j)\right) \right).
\end{equation*}
As in the proof of~\cite[Lemma~5.11]{PT:Part1} this implies that $T_{g_1}$ is strongly measurable and $\gbint T_{g_1}=1$.

Suppose now that $\varepsilon(t)$ satisfies~(\ref{cond: h}). Proposition~\ref{prop:hyper cond by asymptotic} then ensures that $T_{g_1}$ is $g$-hypermeasurable. Thus, by Remark~\ref{rmk:hyper implies strong}, for every $g$-normalized trace on $\sL_g$, we have
\begin{equation*}
 \varphi(T_{g_1})= \gbint T_{g_1}=1.
\end{equation*}
Swapping the roles of $g$ and $g_1$ shows that if $\varphi$ is a $g_1$-normalized trace, then $\varphi(T_g)=1$. This shows that the normalization conditions $\varphi(T_g)=1$ and $\varphi(T_{g_1})=1$ are equivalent, proving the first part.

Finally, if $\varepsilon(t)$ satisfies~(\ref{cond: non-h}), then Proposition~\ref{prop:Stronger-meas.Weyk-non-hyper} ensures that $T_{g_1}$ is not $g$-hypermeasurable. Likewise, the operator $T_g$ is not  $g_1$-hypermeasurable. This gives the second part. The proof is complete.
\end{proof}

 \section{Pietsch's Correspondence for Traces on Weak Lorentz Ideals}\label{sec:Pietsch}

In this section, we work out in the special setting of weak Lorentz ideals Pietsch's correspondence for traces on operator ideals~\cite{Pi:IEOT17, Pi:AM20}. Our presentation merely follows the description of this correspondence given in~\cite{LSZ:Book}. We use it to give a precise description of the class of $g$-normalized positive traces and establish a spectral characterization of strongly measurable operators. We also explain how Dixmier traces fit into this framework. These results extend to weak Lorentz ideals results of Semenov--Sukochev--Usachev--Zanin~\cite{SSUZ:AIM15} for traces and strongly measurable operators on $\sL_{1,\infty}$ (see also~\cite{LSZ:Book}). Incidentally, the considerations in this section further lead to an elementary construction of $g$-hypermeasurable operators that are not spectrally measurable.

Throughout this section, we let $g:[0,\infty)\rightarrow (0,\infty)$ be a continuous $\RV_{-1}$-function such that $\int_0^\infty g(t)dt=\infty$. For the sake of simplicity we further assume that $g$ is decreasing. Note that, as $g(t)$ is ultimately decreasing, we may always replace $g(t)$ by the equivalent function $g(t+a)$, with $a$ large enough so that $g(t)$ is decreasing on $[0,\infty)$. This does not affect the definition of $\sL_g$ (\emph{cf}.~\cite[Remark~2.19]{PT:Part1}).

\subsection{Pietsch's correspondence for ideals}
Pietsch's correspondence for traces grew out of his correspondence for ideals. We briefly recall this correspondence for ideals and what it entails for the ideal $\sL_g$.

In what follows, we let  $S:\ell_\infty \rightarrow\ell_\infty$ be the (forward) shift operator defined by
\begin{equation}\label{def:forward}
 Sa=(0,a_0, a_1,a_2, \ldots ), \qquad a=(a_j)\in \ell_\infty.
\end{equation}
An ideal $\fs$ of $\ell_\infty$ is called \emph{shift-invariant} if $S(\fs)\subseteq \fs$. It is called \emph{monotone} if for every sequence
$x=(x_j)_{j\geq 0}$ in $\fs$, for any sequence $y=(y_j)_{j\geq 0}$ in $\ell_\infty$ we have
\begin{equation*}
\bigg( \sup_{k \geq j}|y_k| \leq  \sup_{k \geq j}|x_k| \quad \forall j\geq 0 \bigg) \Longrightarrow y\in\fs.
\end{equation*}

Pietsch's correspondence for ideals provides a one-to-one correspondence between two-sided ideals of $\sL(\sH)$ and shift-invariant monotone ideals of $\ell_\infty$ (see~\cite{Pi:IEOT17}; see also~\cite{LSZ:Book}).

To state Pietsch's result, we need to introduce some notation. Set
\begin{equation*}
 \D_m:=\left\{j\in \N_0; \ 2^{m}-1\leq j \leq 2^{m+1}-2\right\}, \qquad m\geq 0.
\end{equation*}
Note that these sets form a partition of $\N_0$. Moreover, given any sequence $a=(a_j)\in \co$ and any orthonormal family $(\xi_j)_{j \geq 0}$ in $\sH$ we set
\begin{equation*}
 \Diag(a)= \sum_{j\geq 0} a_j \left(\xi_j^*\otimes \xi_j\right).
\end{equation*}
Note that if $J$ is any ideal of $\sL(\sH)$, then the relation $\Diag(a)\in J$ does not depend on the choice of the orthonormal family $(\xi_j)_{j \geq 0}$.

In addition, we consider the \emph{dyadic dilation} operator $\delta: \ell_\infty \rightarrow \ell_\infty$ and the dyadic average operator $\alpha:\ell_\infty \rightarrow \ell_\infty$ given by
\begin{gather*}
\delta (a)_j=a_m \qquad \text{if} \ j \in \D_m,\\
\alpha (a)_j= 2^{-j}\sum_{k\in \D_j} a_k, \qquad a=(a_j)_{j\geq 0}\in \ell_\infty.
\end{gather*}
The operator $\alpha$ is a right-inverse of $\delta$, i.e., $\alpha\circ \delta=\op{id}$. In~\cite{LSZ:Book} the operators $\delta$ and $\alpha$ are denoted $L$ and $L'$, respectively. In the notation of~\cite{Pi:IEOT17} the operators $D_a$, $a\in \ell_\infty$,  are the operators $\Diag(\delta(a))$.

We are now in a position to state Pietsch's correspondence for ideals.

\begin{proposition}[Pietsch~\cite{Pi:IEOT17}]\label{prop:corpdns on ideals}
 We have a one-to-one correspondence between two-sided ideals of $\sL(\sH)$ and shift-invariant monotone ideals of $\ell_\infty$ as follows.
 \begin{enumerate}
 \item To any ideal $J$ of $\sL(\sH)$ corresponds the shift-invariant monotone ideal,
 \begin{equation*}
 \fs_J:=\left\{ x\in \ell_\infty; \ \Diag\left(\delta(x)\right)\in J\right\}.
\end{equation*}

\item To any shift-invariant monotone ideal of $\ell_\infty$ corresponds the ideal,
\begin{equation*}
 J_\fs:= \left\{T\in \sK; \ \left(\mu_{2^m-1}(T)\right)_{m\geq 0}\in \fs\right\}.
\end{equation*}

\item The correspondences $J\rightarrow \fs_J$ and $\fs\rightarrow J_\fs$ are inverses of each other.
\end{enumerate}
\end{proposition}

In what follows given any sequence $x=(x_j)_{j\geq 0}$ in $\co$ we denote by $\mu(x)=(\mu_j(x))_{j\geq 0}$ the non-increasing re-arrangement of the sequence of absolute values $(|x_j|)_{j\geq 0}$. Equivalently, $\mu_j(x)=\mu_j(\Diag(x))$, $j\geq 0$.

If $J$ is a two-sided ideal of $\sL(\sH)$, then its \emph{Calkin space} (also called the \emph{commutative core}) is
\begin{equation*}
 \ell_J:=\{ x\in \co; \ \Diag(x)\in J\}.
\end{equation*}
The first part of Proposition~\ref{prop:corpdns on ideals} means that $\delta^{-1}(\ell_J)=\fs_J$. In particular, we see that $\delta$ induces a linear operator $\delta: \fs_J \rightarrow \ell_J$.

The second part of Proposition~\ref{prop:corpdns on ideals} can also be re-interpreted in terms of the operator $\alpha$. Let $\fs$ be a shift-invariant monotone ideal.
If $x\in \co$ is such that $x=\mu(x)$, then
\begin{equation*}
 \mu_{2^{m+1}-1}(x) \leq \alpha(x)_m \leq \mu_{2^{m}-1}(x), \qquad m\geq 0.
\end{equation*}
Combining this with the shift-invariance of $\fs$, we see that
\begin{equation*}
 (x_{2^m-1})_{m\geq 0} \in \fs \Longleftrightarrow \alpha(x)\in \fs.
\end{equation*}
It then follows that
\begin{equation}\label{def:J_s}
 J_\fs = \left\{ T\in \sK; \ \alpha(\mu(T))\in \fs\right\}.
\end{equation}
We thus recover the definition of $J_{\fs}$ given in~\cite[Chapter~4]{LSZ:Book}.

From~(\ref{def:J_s}) we obtain that if $x\in \co$, then
\begin{equation*}
 x\in \ell_{J_\fs}  \Longleftrightarrow  \Diag(\mu(x))\in J_\fs \Longleftrightarrow \alpha(\mu(x))\in \fs.
\end{equation*}
 Thus, if $x\in \ell_{J_\fs}$ and $|x_0|\geq |x_1|\geq \cdots$, then $\mu_j(x)=|x_j|$, and so we have $|\alpha(x)|_m\leq \alpha(\mu(x))_m$. As $\alpha(\mu(x))\in \fs$, the monotonicity of $\fs$ then implies that $\alpha(x)\in \fs$. Therefore, if we set
 \begin{equation*}
\tilde{\ell}_{J_\fs}=\big\{x \in  \ell_{J_\fs}; |x_0|\geq |x_1|\geq \cdots\big\},
\end{equation*}
 then we see that $\alpha$ induces a linear operator $\alpha:\tilde{\ell}_{J_\fs} \rightarrow \fs$.

For $J=\sL_g$ the corresponding shift-invariant monotone ideal is identified as follows. Set $\zeta_g= (g(2^{j}))_{j\geq 0}$, and denote by $\ell_\infty(\zeta_g)$ the ideal of $\ell_\infty$ that is generated by $\zeta_g$, i.e.,
\begin{equation*}
 \ell_\infty(\zeta_g)=\left\{x=(x_j)_{j\geq 0}\in \ell_\infty; \ |x_j|=\op{O}\big(g(2^j)\big)\right\}.
\end{equation*}
This is a monotone ideal of $\ell_\infty$. Moreover, as the $\RV_{-1}$-property of $g(t)$ implies that $g(2^{j-1})\sim 2 g(2^j)$, we see that $\ell_\infty(\zeta_g)$ is shift-invariant as well.

The $\RV_{-1}$-property of $g(t)$ further implies that if $j\in\D_m$, then
\begin{equation*}
 g(2^m)\leq \frac{g(2^m)}{g(2^{m+1}-2)}g(j)\sim 2g(j).
\end{equation*}
Thus, for any operator $T\in \sK$, we have
\begin{equation*}
 \mu_{2^m-1}(T) =\op{O}(g(2^m)) \Longleftrightarrow \mu_j(T)=\op{O}(g(j))  \Longleftrightarrow T\in \sL_g.
\end{equation*}
That is, $J_{\ell_\infty(\zeta_g)}=\sL_g$, and hence we have
\begin{equation*}
 \fs_{\sL_g}=\ell_\infty(\zeta_g).
\end{equation*}

In addition, the Calkin space of $\sL_g$ is the weak Lorentz sequence space,
\begin{equation*}
 \ell_g:=\big\{ x\in \co; \ \mu_j(x)= \op{O}(g(j))\big\}.
\end{equation*}
We thus get linear operators,
\begin{equation*}
 \delta:\ell_\infty(\zeta_g) \longrightarrow \ell_g \qquad \text{and} \qquad \alpha:\tilde{\ell}_g \longrightarrow \ell_\infty(\zeta_g).
\end{equation*}

\subsection{Pietsch's correspondence for traces on $\sL_g$}
Given any two-sided ideal $J$ Pietsch's correspondence provides a one-to-one correspondence between traces on $J$ and linear functionals on $\fs_J$ that are $1/2$-shift invariant, i.e., linear functionals $\ell:\fs_J\rightarrow \C$ such that
\begin{equation*}
\frac12 \ell(Sx)=\ell(x) \qquad \text{for all}\ x\in \fs_J.
\end{equation*}

For the ideal $\sL_g$ the corresponding shift-invariant monotone ideal is $\ell_\infty(\zeta_g)$. We shall specialize Pietsch's results to this case.

A key ingredient in Pietsch's construction of traces is the notion of \emph{$\zeta_g$-representations} for operators in $\sL_g$. These are representations $T=\sum_{m\geq 0} T_m$, where
 \begin{equation}\label{eq:zeta_g reps}
 \op{rk} (T_m)\leq 2^m \qquad \text{and} \qquad \big\|T -\sum_{m<j} T_m\big\| =\op{O}\big(g(2^j)\big).
\end{equation}

Every $T\in \sL_g$ admits such a representation. If $T=\sum_{j\geq 0} \mu_j(T) \eta_j^*\otimes \xi_j$ is any Schmidt series representation of $T$, then $T$ admits the $\zeta_g$-representation,
\begin{equation}\label{eq:2nd zeta_g reps}
 T=\sum_{m\geq 0}T_m, \qquad T_m:= \sum_{j\in \D_m} \mu_j(T)\eta_j^*\otimes \xi_j.
\end{equation}

Another key ingredient is the following observation which, according to Pietsch~\cite[p.~603]{Pi:IEOT17}, goes back to unpublished results of Figiel in the 80s. It shows that traces on $\sL_g$ are uniquely determined by their values on the operators $\Diag[\delta(x)]$, $x\in \ell_\infty(\zeta_g)$.

\begin{lemma}[see {\cite[Proposition~4.3]{Pi:IEOT17}}] \label{lem:Pietsch.Figiel}
If $T=T^*\in \sL_g$, then
\begin{equation*}
T -\Diag\big[(\delta\circ\alpha)(\lambda(T))\big] \in \Com\big(\sL_g\big).
\end{equation*}
In particular, for every trace $\varphi$ on $\sL_g$, we have
\begin{equation*}
 \varphi(T)= \varphi\big( \Diag[(\delta\circ\alpha)(\lambda(T))]\big).
\end{equation*}
 \end{lemma}

 We are now in a position to state Pietsch's correspondence for traces in $\sL_g$ in the form given in~\cite{Pi:IEOT17}.

\begin{proposition}[Pietsch~\cite{Pi:IEOT17}]\label{prop:corpdns on traces}
We have a one-to-one correspondence between traces on $\sL_g$ and  $1/2$-shift invariant linear functionals on $\ell_\infty(\zeta_g)$ as follows.
\begin{enumerate}
 \item For every $1/2$-shift invariant linear functional $\ell$ on $\ell_\infty(\zeta_g)$, the following formula defines a trace on $\sL_g$,
\begin{equation*}
 \varphi_\ell(T):=\ell\bigg( \bigg\{\frac{1}{2^{m}} \Tr(T_m)\bigg\}_{m\geq 0}\bigg), \qquad T\in \sL_g,
\end{equation*}
where $T=\sum T_m$ is any $\zeta_g$-representation of $T$, the choice of which is irrelevant.

\item  For every  trace $\varphi$ on $\sL_g$, the following formula defines a $1/2$-shift invariant linear functional on $\ell_\infty(\zeta_g)$,
\begin{equation}\label{eq:Pietsch.ell-varphi}
 \ell_\varphi(x) = \varphi\big(\Diag(\delta(x))\big), \qquad x \in \ell_\infty(\zeta_g).
\end{equation}

\item The correspondences $\ell \rightarrow \varphi_\ell$ and $\varphi \rightarrow \ell_\varphi$ are inverses of each other.
\end{enumerate}
\end{proposition}

\begin{remark}\label{rmk:zeta_g reps}
 If $T\in (\sL_g)_+$ has Schmidt series representation $T=\sum \mu_j(T) \xi_j^*\otimes \xi_j$, then the $\zeta_g$-representation~(\ref{eq:zeta_g reps}) becomes $T= \sum T_m$ with $T_m:=\sum_{j\in \D_m} \mu_j(T)\xi_j^*\otimes \xi_j$. Note that
\begin{equation*}
 \frac{1}{2^m}\Tr(T_m)=  \frac{1}{2^m}\sum_{j\in \D_m} \mu_j(T) \Tr[ \xi_j^*\otimes \xi_j] = \frac{1}{2^m}\sum_{j\in \D_m} \mu_j(T)=\alpha(\mu(T))_m.
\end{equation*}
Therefore, if $\ell$ is any $1/2$-shift invariant linear functional on $\ell_\infty(\zeta_g)$, then we get
\begin{equation*}
 \varphi_\ell(T) = \ell\circ\alpha(\mu(T)) \qquad \forall T\in \big(\sL_g\big)_+.
\end{equation*}
We thus recover the definition of $\varphi_\ell$ given in~\cite[Chapter~4]{LSZ:Book}.
\end{remark}

\begin{remark}
 If $x\in \ell_\infty(\zeta_g)$, then an  $ \ell_\infty(\zeta_g)$-representation of $\Diag(\delta(x))$ is
\begin{equation*}
 \Diag(\delta(x)) = \sum_{m\geq 0} a_m D_m, \qquad D_m:=\sum_{j\in \D_m} \xi_j^*\otimes \xi_j.
\end{equation*}
As $2^{-m}\Tr[x_mD_m]=x_m$, we see that
\begin{equation*}
\ell_{\varphi_{\ell}}(x)=\varphi_\ell\big(\Diag(\delta(x))\big) = \ell\big(\{2^{-m}\Tr[D_m]\}\big)=\ell(x).
\end{equation*}
This shows that $\varphi \rightarrow \ell_\varphi$ is a right-inverse of $\ell \rightarrow \varphi_\ell$. The fact this is a left-inverse is a consequence of Lemma~\ref{lem:Pietsch.Figiel} and the fact that $\delta$ is a right-inverse of $\alpha$.
\end{remark}

As further observed by Pietsch~\cite[Lemma~6.1]{Pi:AM20}, for principal generated ideals his correspondence for traces can be reformulated in terms of shift-invariant linear functionals on $\ell_\infty$, i.e., linear functionals $\ell: \ell_\infty \rightarrow \C$ such that $\ell(Sx)=\ell(x)$ for all $x\in \ell_\infty$.

Let $Z_g:\ell_\infty \rightarrow \ell_\infty(\zeta_g)$ be the linear isomorphism given by
\begin{equation*}
 Z_g(x)_j=(\zeta_g)_jx_j=g(2^j)x_j, \qquad x=(x_j)_{j\geq 0} \in \ell_\infty.
\end{equation*}
By duality this gives rise to a linear isomorphisms between the respective (algebraic) duals of $\ell_\infty$ and $\ell_\infty(\zeta_g)$, i.e.,
\begin{equation*}
 Z_g^*:\ell_\infty(\zeta_g)^*\longrightarrow \ell_\infty^*, \qquad Z_g^*\ell =\ell\circ Z_g.
\end{equation*}

\begin{lemma}[Pietsch~{\cite[Lemma~6.1]{Pi:AM20}}]
 The linear map  $Z_g^*:\ell_\infty(\zeta_g)^*\longrightarrow \ell_\infty^*$ induces a one-to-one correspondence between  $1/2$-shift invariant linear functionals on $\ell_\infty(\zeta_g)$ and shift-invariant linear functionals on $\ell_\infty$.
\end{lemma}

This allows us to reformulate Proposition~\ref{prop:corpdns on traces} in terms of shift-invariant linear functionals on $\ell_\infty$. Let $\hat{\delta}_g: \ell_\infty\rightarrow \ell_g$ and
$\hat{\alpha}_g:\tilde{\ell}_g \rightarrow \ell_\infty$ be the operators $\hat{\delta}_g=\delta\circ Z_g$ and $\hat{\alpha}_g:=Z_g^{-1}\circ \delta$. That is,
\begin{gather}
  \hat{\delta}_g (x)_j= g(2^m)x_m \quad \text{for} \ j\in \D_m, \qquad x=(x_j)\in \ell_\infty,\label{def:hat delta}\\
\hat{\alpha}_g(x)_j= \frac{1}{2^jg(2^j)} \sum_{k\in \D_j} x_k, \qquad x=(x_j)\in \tilde{\ell}_g.\label{def:hat alpha}
\end{gather}
Note that $\hat{\alpha}_g$ is a right-inverse of $\hat{\delta}_g$.

 If $\ell$ is a shift-invariant functional on $\ell_\infty$ and $T\in \sL_g$ has $\ell_\infty(\zeta_g)$-decomposition $T= \sum T_m$, then we have
\begin{equation*}
 \varphi_{(Z_g^*)^{-1}\ell}(T)=\varphi_{\ell\circ Z_g^{-1}}(T)= \left(\ell\circ Z_g^{-1}\right)\left(\left\{\frac{1}{2^{m}} \Tr(T_m)\right\}\right)
 =  \ell\left(\left\{\frac{1}{g(2^m)2^{m}} \Tr(T_m)\right\}\right).
\end{equation*}
If $\varphi$ is a trace on $\sL_g$, then, for every $x\in \ell_\infty$, we have
 \begin{equation*}
Z_g^*\ell_\varphi(x)= \ell_\varphi\big(Z_g(x)\big)= \varphi\big(\Diag\big[\delta(Z_g(x))\big]\big)= \varphi\big(\Diag\big[\hat{\delta}_g(x)\big]\big).
\end{equation*}
Therefore, we arrive at the following reformulation of Pietsch's correspondence for traces on $\sL_g$.

\begin{theorem}[Pietsch~\cite{Pi:IEOT17, Pi:AM20}]\label{thm:Pietsch}
 We have a one-to-one correspondence between traces on $\sL_g$ and shift-invariant linear functionals on $\ell_\infty$ as follows.
\begin{enumerate}
 \item For every shift-invariant linear functional $\ell$ on $\ell_\infty$, the following formula defines a trace on $\sL_g$,
\begin{equation}\label{def:varphi ell}
 \varphi_\ell(T):=\ell\left(\left\{\frac{1}{g(2^m)2^{m}} \Tr(T_m)\right\}_{m\geq 0}\right),  \qquad T\in \sL_g,
\end{equation}
where $T=\sum T_m$ is any $\zeta_g$-representation of $T$, the choice of which is irrelevant.

\item  For every  trace $\varphi$ on $\sL_g$, the following formula defines a shift-invariant linear functional on $\ell_\infty$,
\begin{equation}\label{def:ell varphi}
 \ell_\varphi(x) = \varphi\big(\Diag(\hat{\delta}_g(x))\big), \qquad x \in \ell_\infty.
\end{equation}

\item The correspondences $\ell \rightarrow \varphi_\ell$ and $\varphi \rightarrow \ell_\varphi$ are inverses of each other.
\end{enumerate}
\end{theorem}

\begin{remark}[see~{\cite[Lemma~8.1]{Pi:IEOT15}}] \label{rmk:Pietsch.shift-singular}
For a linear functional on $\ell_\infty$, invariance by the forward shift operator $S$ in~(\ref{def:forward}) is equivalent to invariance by the backward shift operator,
\begin{equation*}
 S^*(a_0,a_1, a_2, \ldots)= (a_1,a_2,\ldots), \qquad (a_j)_{j\geq 0}\in \ell_\infty.
\end{equation*}
It follows that every shift-invariant linear functional on $\ell_\infty$ annihilates finitely supported sequences. Combining this with Theorem~\ref{thm:Pietsch} we thus recover the fact that every trace on $\sL_g$ is singular.
\end{remark}

\begin{remark}
 As pointed out in Remark~\ref{rmk:zeta_g reps}, every $T\in (\sL_g)_+$ admits a $\zeta_g$-representation $T=\sum T_m$, with $2^{-m}\Tr(T_m)=\alpha(\mu(T))_m$. Thus, in view of~(\ref{def:hat alpha}) we have
 \begin{equation*}
 \frac{1}{g(2^m)2^m} \Tr(T_m)= \frac{1}{g(2^m)}\alpha(\mu(T))_m=\hat{\alpha}_g(\mu(T))_m.
\end{equation*}
Therefore, given any shift-invariant linear functional on $\ell_\infty$, the formula~(\ref{def:varphi ell}) gives
\begin{equation*}
  \varphi_\ell(T):= \ell\left( \big\{\hat{\alpha}_g(\mu(T))_m \big\}_{m\geq 0}\right)= \ell\circ \hat{\alpha}_g(\mu(T)).
\end{equation*}
\end{remark}

\begin{remark}\label{rmk:semi-finite}
We refer to~\cite[Theorem~3.5]{LU:JMAA25} for a version of Pietsch's correspondence for weak Lorentz ideals in the semi-finite setting. The correspondence there is stated for continuous traces. (In the terminology of~\cite{LU:JMAA25} traces are called symmetric functionals.) We refer to Corollary~\ref{cor:Pietsch.continuous-traces} for a more precise description of continuous traces on $\sL_g(\sH)$.
\end{remark}

\subsection{Pietsch's correspondence. Spectral version}
We seek a spectral reformulation of Pietsch's correspondence, i.e., in terms of eigenvalue sequences. For one thing, in ~\cite{PT:Part1} Dixmier traces were constructed directly out of eigenvalue sequences. On the other hand, by a result of Dykema--Kalton~\cite{DK:Crelle98} every trace on a quasi-Banach ideal is spectral, i.e., its values depend only on the spectrum of the operators at stake. Therefore, it is natural to ask how this fits with Pietsch's construction of traces.

For the weak trace-class $\sL_{1,\infty}$ a spectral reformulation of Pietsch is provided in~\cite[Chapter~6]{LSZ:Book}. This can also be deduced from the results of Semenov--Sukochev--Usachev--Zanin~\cite{SSUZ:AIM15} by combining~\cite[Theorem~4.1]{SSUZ:AIM15} with the Lidskii-type formula provided by~\cite[Theorem~6.3]{SSUZ:AIM15}.

We shall now extend this spectral formulation to weak Lorentz ideals. In fact, as we shall see, the result is a direct consequence of combining Dykema--Kalton's result with Pietsch's construction of traces in terms of $\zeta_g$-representations. Therefore, even in the case of $\sL_{1,\infty}$, our approach simplifies the approaches of~\cite{LSZ:Book, SSUZ:AIM15}.
 
 The aforementioned result of Dykema--Kalton~\cite{DK:Crelle98} can be formulated as follows.

\begin{proposition}[\cite{DK:Crelle98}] \label{prop:DK}
Let $T\in \sL_g$. The following hold. 
\begin{enumerate}
 \item Every eigenvalue sequence $\lambda(T)$ is in $\tilde{\ell}_g$, and we have
 \begin{equation*}
 T- \Diag\big(\lambda(T)\big)\in \Com\big(\sL_g\big). 
\end{equation*}

\item For every trace $\varphi$ on $\sL_g$, we have
\begin{equation*}
 \varphi(T)=\varphi\big(\Diag( \lambda(T))\big).
\end{equation*}
In particular, the value of $\varphi(T)$ depends only on the spectrum of $T$.
\end{enumerate}
\end{proposition}

\begin{remark}
In~\cite{DK:Crelle98} the above result is proved on geometrically stable ideals. Every quasi-Banach ideal is geometrically stable (see~\cite{Ka:Crelle98}). Dykema--Kalton's result was extended to ideals that are closed under logarithmic submajorization by Sukochev--Zanin~\cite{SZ:AIM14} (see also~\cite[Chapter~5]{LSZ:Book}).
\end{remark}

The following result extends to weak Lorentz ideals the Lidskii formula for traces on $\sL_{1,\infty}$ of Semenov--Sukochev--Usachev--Zanin~\cite[Theorem~6.3]{SSUZ:AIM15}.

\begin{proposition}\label{prop:Pietsch.Lidskii}
 For every shift-invariant linear functional $\ell$ on $\ell_\infty$, we have
 \begin{equation}\label{eq:Lidskii by hat alpha}
 \varphi_\ell(T)= \ell\circ \hat{\alpha}_g\big( \lambda(T)\big) \qquad \forall T\in \sL_g.
\end{equation}
\end{proposition}
\begin{proof}
 Let $T\in \sL_g$. As by Proposition~\ref{prop:DK} we have $\varphi_\ell(T)= \varphi_\ell(\Diag( \lambda(T)))$, it is enough to establish~(\ref{eq:Lidskii by hat alpha}) for the operator $N=\Diag( \lambda(T))$.

 By definition $N= \sum \lambda_j(T) \xi_j^*\otimes \xi_j$. As $|\lambda_0(T)|\geq |\lambda_1(T)|\geq \cdots$, we see that $\mu_j(N)=|\lambda_j(T)|$.  Thus, if we put $\lambda_j(T)=e^{i\theta_j}\mu_j(N)$ and $\eta_j:=e^{i\theta_j}\xi_j$, then $(\eta_j)_{j\geq 0}$ is an orthonormal family in $\sH$, and we have
 \begin{equation*}
 N=\sum_{j\geq 0} \lambda_j(T)\xi_j^*\otimes \xi_j= \sum_{j\geq 0} \mu_j(N)\eta_j^*\otimes \xi_j.
\end{equation*}
 In particular, we see that $N=\sum_{j\geq 0} \mu_j(N)\eta_j^*\otimes \xi_j$ is a Schmidt series representation of $N$. Therefore, by~(\ref{eq:2nd zeta_g reps}) we get the $\zeta_g$-representation,
 \begin{equation*}
 N=\sum_{m\geq 0} N_m, \qquad N_m:= \sum_{j\in \D_m} \mu_j(N)\eta_j^*\otimes \xi_j=  \sum_{j\in \D_m} \lambda_j(T)\xi_j^*\otimes \xi_j.
\end{equation*}
 As $\Tr(\xi_j^*\otimes \xi_j)=1$, we get
 \begin{equation*}
 \frac{1}{g(2^m)2^m} \Tr(N_m)=   \frac{1}{g(2^m)2^m} \sum_{j\in \D_m} \lambda_j(T) =\hat{\alpha}_g\big(\lambda(T)\big)_m.
\end{equation*}
 The formula~(\ref{def:varphi ell}) then yields
 \begin{equation*}
 \varphi_\ell(N)= \ell\bigg( \bigg\{\frac{1}{2^{m}g(2^m)} \Tr(N_m)\bigg\}\bigg) =\ell\big(\hat{\alpha}_g\big(\lambda(T)\big)\big).
\end{equation*}
 This gives~(\ref{eq:Lidskii by hat alpha}). The proof is complete.
\end{proof}

Combining Theorem~\ref{thm:Pietsch} and Proposition~\ref{prop:Pietsch.Lidskii} we then arrive at the following spectral reformulation of Pietsch's correspondence for traces on~$\sL_g$. This extends to all weak Lorentz ideals the corresponding result for traces on $\sL_{1,\infty}$ (see~\cite[Theorem~6.3.1(a)--(c)]{LSZ:Book}).

\begin{theorem}\label{thm:Pietsch.Spectral}
 For every shift-invariant linear functional $\ell$ on $\ell_\infty$, the following formula defines a trace on $\sL_g$,
 \begin{equation}\label{eq:final Lidskii}
  \varphi_\ell(T)= \ell\circ \hat{\alpha}_g\big( \lambda(T)\big),  \qquad T\in \sL_g.
\end{equation}
Conversely, given any trace $\varphi$ on $\sL_g$, there is a unique shift-invariant linear functional $\ell=\ell_\varphi$ given by~(\ref{def:ell varphi}) such that $\varphi=\varphi_\ell$.
\end{theorem}

\subsection{Positive normalized traces}
We shall now specialize Theorem~\ref{thm:Pietsch.Spectral} to positive traces and positive normalized traces. As in~\cite{SSUZ:AIM15, LSZ:Book} this involves using \emph{Banach limits}, i.e., shift-invariant extended limits.

For positive traces we have the following characterization which, in the formulation given below, seems to be new, even in the case of the weak trace-class $\sL_{1,\infty}$ (compare~\cite[Theorem~6.3.1(d)]{LSZ:Book}).

\begin{proposition}\label{prop:Pietsch.positive-traces}
 Under the correspondence provided by Theorem~\ref{thm:Pietsch.Spectral}, the (non-zero) positive traces on $\sL_g$ are in one-to-one correspondence with positive scalar multiples of Banach limits.
\end{proposition}
\begin{proof}
It is immediate from~(\ref{def:hat alpha}) that if $\ell$ is a positive shift-invariant linear functional, then~(\ref{eq:final Lidskii}) defines a positive trace. Conversely, as the operator $\hat{\delta}_g$ given by~\eqref{def:hat delta} preserves positivity we see that if $\varphi$ is a positive trace, then the linear functional $\ell_\varphi$ given by~(\ref{def:ell varphi}) is positive. Thus, Pietsch's correspondence induces a one-to-one correspondence between positive shift-invariant linear functionals on $\ell_\infty$ and positive traces on $\sL_g$.

It remains now to identify (non-zero) positive shift-invariant linear functionals on $\ell_\infty$. Every positive scalar multiple of a Banach limit is such a functional.
Conversely, let $\ell$ be a (non-zero) positive shift-invariant linear functional on $\ell_\infty$.  For one thing, as mentioned in Remark~\ref{rmk:Pietsch.shift-singular} the shift-invariance of $\ell$ implies that it annihilates finitely supported sequences. On the other hand, the positivity of $\ell$ ensures it is continuous with $\|\ell\|=\ell(1)$.  It then follows that $\ell$ annihilates the null sequence space $\co$.

Furthermore, as $\ell\neq 0$, we have $\ell(1)=\|\ell\|>0$. Set $\theta= \ell(1)^{-1}\ell$. Then $\theta$ is a positive linear functional on $\ell_\infty$ that annihilates $\co$. By definition $\theta(1)= \ell(1)^{-1}\ell(1)=1$, and so $\theta$ is an extended limit. The shift-invariance of $\ell$ is transferred to $\theta$, and so we see that $\theta$ is a Banach limit. As $\ell=\ell(1)\theta$ we then see that $\ell$ is a positive scalar multiple of a Banach limit.

To sum up, (non-zero) positive shift-invariant linear functionals on $\ell_\infty$ are exactly the positive multiples of Banach limits. As they are in one-to-one correspondence with (non-zero)  positive traces on $\sL_g$ we get the result.
 \end{proof}

As every continuous trace on $\sL_g$ is a linear combination of positive traces (\emph{cf}.~\cite[Theorem~4.1.10]{LSZ:Book}), from Proposition~\ref{prop:Pietsch.positive-traces} we also obtain the following characterization of continuous traces.

\begin{corollary}\label{cor:Pietsch.continuous-traces}
 Under the correspondence provided by Theorem~\ref{thm:Pietsch.Spectral}, the space of continuous traces on $\sL_g$ is in one-to-one correspondence with the space spanned by
 Banach limits.
 \end{corollary}

\begin{remark}
 As with Proposition~\ref{prop:Pietsch.positive-traces}, Corollary~\ref{cor:Pietsch.continuous-traces} seems to be a new result, even in the case of $\sL_{1,\infty}$ (compare~\cite[Theorem~6.3.1(f)]{LSZ:Book} and~\cite[Theorem~3.5]{LU:JMAA25}).
\end{remark}

To understand how $g$-normalized positive traces fit under Pietsch's correspondence, it is convenient to use the following lemma which, at least for $a=b=0$,  is a standard result from the theory of $\RV$-functions (see, e.g, \cite[Chapter~3]{BGT:Cambridge87}).

\begin{lemma}\label{lem:auxiliary result} Given any $a,b\in \R$, we have
 \begin{equation*}
 \lim_{t\rightarrow \infty} \frac{G(2t+a)-G(t+b)}{tg(t)}= \log 2.
\end{equation*}
If $a$ and $b$ are integers, then we also have
\begin{equation}\label{eq:uni convergence}
 \lim_{j\rightarrow \infty} \frac{1}{jg(j)} \sum_{k=j+a}^{2j+b} g(k)= \log 2.
\end{equation}
\end{lemma}
\begin{proof}
 As for $t$ large enough, we have $|G(t+a)-G(t)|\leq |a|g(t-|a|)\sim |a|g(t)=o(tg(t))$, we may assume that $a=b=0$. In that case the limit~(\ref{eq:uni convergence}) is a consequence of the uniform convergence theorem for $\RV$-functions (\emph{cf}.~\cite[Remark~2.16]{PT:Part1}). This result ensures that $g(t)^{-1}g(\lambda t)\rightarrow \lambda^{-1}$ as $t\rightarrow \infty$ uniformly in $\lambda$ on the compact interval $[1,2]$. Thus,
 \begin{equation*}
 \frac{G(2t)-G(t)}{tg(t)}= \int_t^{2t} \frac{g(s)}{g(t)}\frac{ds}{t}=\int_1^2 \frac{g(\lambda t)}{g(t)} d\lambda \longrightarrow \int_1^2 \frac{d\lambda}{\lambda}=\log 2.
\end{equation*}
 This proves the first part.

 To prove the second part we may also assume $a=b=0$. In this case, the result follows from the first part and the inequalities
 \begin{equation*}
 G(2j+1)-G(j)=\int_{j}^{2j+1}\!\!\! g(t)dt \leq \sum_{k=j}^{2j} g(k) \leq \int_{j-1}^{2j}\!\!\! g(t)dt=G(2j)-G(j-1).
\end{equation*}
The proof is complete.
 \end{proof}

In what follows we denote by $\gamma$ the sequence $(g(j))_{j\geq 0}$. As $\lambda(T_{g})=\gamma$ we see that, if $\ell$ is a linear functional on $\ell_\infty$, then
\begin{equation*}
\varphi_\ell(T_{g})= \ell\circ\hat{\alpha}_g\left(\lambda(T_{g})\right)=  \ell\circ\hat{\alpha}_g(\gamma).
\end{equation*}
Thus,
\begin{equation*}
\varphi_\ell\ \text{is a $g$-normalized trace} \ \Longleftrightarrow\  \ell\left(\hat{\alpha}_g(\gamma)\right)=1.
\end{equation*}
By Lemma~\ref{lem:auxiliary result} we have
\begin{equation*}
 \hat{\alpha}_g(\gamma)_j = g(2^j)^{-1}\alpha(\gamma)_j=  \frac{1}{g(2^j)2^j} \sum_{k\in \D_j}g(k)\longrightarrow \log 2.
\end{equation*}
Therefore, for any extended limit $\omega$ on $\ell_\infty$, we have
\begin{equation*}
 \omega( \hat{\alpha}_g(\gamma))= \log 2.
\end{equation*}
It then follows that, for every Banach limit $\theta$ on $\ell_\infty$, we have
\begin{equation}\label{eq:Pietsch.phitheta-Tg}
 \varphi_\theta(T_g)= \theta\circ \hat{\alpha}_g(\gamma)= \log 2.
\end{equation}

By combining~(\ref{eq:Pietsch.phitheta-Tg}) with Proposition~\ref{prop:Pietsch.positive-traces} we then arrive at the following result, which extends~\cite[Theorem~6.1.1(a)]{LSZ:Book} to weak Lorentz ideals (see also~\cite[Corollary~4.2]{SSUZ:AIM15}).

 \begin{theorem}\label{thm:Pietsch.PNT}
 Under the correspondence provided by Theorem~\ref{thm:Pietsch.Spectral}, the $g$-normalized positive traces on $\sL_g$ are exactly the traces of the form,
 \begin{equation*}
\hat{\varphi}_\theta(T):= \frac{1}{\log 2}\theta\circ \hat{\alpha}_g\big(\lambda(T)\big), \qquad T\in \sL_g,
\end{equation*}
where $\theta$ ranges over all Banach limits on $\ell_\infty$.
\end{theorem}

\begin{remark}
As mentioned in Remark~\ref{rmk:semi-finite}, a version of Pietsch's correspondence for weak Lorentz ideals in the semi-finite setting is given in~\cite{LU:JMAA25}.  The correspondence there is stated for continuous traces. It seems that by using Lemma~\ref{lem:characterization of ac} it should be possible to specialize their correspondence to $g$-normalized positive traces and get a semi-finite version of Theorem~\ref{thm:Pietsch.PNT}.
\end{remark}

\begin{remark}\label{rmk:Pietsch.independence-g}
 Rigorously speaking, for Proposition~\ref{prop:Pietsch.positive-traces}, Corollary~\ref{cor:Pietsch.continuous-traces}, and Theorem~\ref{thm:Pietsch.PNT} to hold the function $g(t)$ needs to be decreasing, since this is a running assumption of this section. This condition actually can be dropped. To see this let $g_1(t)$, $t\geq 0$, be any continuous $\RV_{-1}$-function such that $g_1(t)\sim g(t)$. Given any $T\in \sL_g$, we then have
 \begin{equation*}
 \hat{\alpha}_{g_1}\big(\lambda(T)\big)_j= \frac{1}{2^jg_1(2^j)} \sum_{k\in \D_j} \lambda_k (T)\sim  \frac{1}{2^jg(2^j)} \sum_{k\in \D_j}\lambda_k (T)= \hat{\alpha}_g\big(\lambda(T)\big)_j.
\end{equation*}
As $ \hat{\alpha}_{g}(\lambda(T))$ and  $ \hat{\alpha}_{g_1}(\lambda(T))$ are bounded sequences, this implies that  $ \hat{\alpha}_{g}(\lambda(T))- \hat{\alpha}_{g_1}(\lambda(T))$ is in $\co$. Therefore, for any shift-invariant linear functional $\ell$ on $\ell_\infty$ that annihilates $\co$, we have
\begin{equation*}
 \varphi_\ell\left(\lambda(T)\right)= \ell\left(\hat{\alpha}_g\big(\lambda(T)\big)\right) = \ell\left(\hat{\alpha}_{g_1}\big(\lambda(T)\big)\right).
\end{equation*}
This holds if $\ell$ is a positive multiple of a Banach limit, and more generally if $\ell$ is any linear combinations of Banach limits. Therefore, in the statements of Proposition~\ref{prop:Pietsch.positive-traces}, Corollary~\ref{cor:Pietsch.continuous-traces}, and Theorem~\ref{thm:Pietsch.PNT} we may replace the function $g(t)$ by any continuous $\RV_{-1}$-function  $g_1(t)$, $t\geq 0$, such that $g_1(t)\sim g(t)$.
\end{remark}

\subsection{A $g$-hypermeasurable operator which is not spectrally measurable}\label{subsec:hyper-not-spec}
As a further application of the considerations that lead to Theorem~\ref{thm:Pietsch.Spectral}, we produce an explicit example of a $g$-hypermeasurable operator which is not spectrally measurable.

Set $T=\Diag(\delta\circ \alpha(\gamma))$, where as above $\gamma=(g(j))_{j\geq 0}$. That is, $T\xi_j=\lambda_j(T)\xi_j$, with
\begin{equation*}
 \lambda_j(T)=\delta\circ \alpha(\gamma)_j=\alpha(\gamma)_m=2^{-m} \sum_{k\in \D_m} g(k) \qquad \text{if}\ j\in \D_m.
\end{equation*}
In particular, $T$ is a positive compact operator, and we have
\begin{equation*}
 \mu_j(T)=\lambda_j(T)\leq g(2^m)\leq g(j/2)\sim 2g(j).
\end{equation*}
Thus, $T$ is an operator in $\sL_g$.

\begin{proposition}\label{prop:hyper non-weyl}
 The operator $T$ is $g$-hypermeasurable in $\sL_g$ with $\gbint T=1$. However, $T$ is not spectrally measurable.
\end{proposition}
\begin{proof}
 As $T_{g}=\Diag(\gamma)$, by Lemma~\ref{lem:Pietsch.Figiel} we have
 \begin{equation*}
 T-T_{g}=\Diag(\delta\circ \alpha(\gamma))-\Diag(\gamma)\in \Com(\sL_g).
\end{equation*}
It then follows from Theorem~\ref{thm:hyper} that $T$  is $g$-hypermeasurable with $\gbint T=1$.

In addition, thanks to Lemma~\ref{lem:auxiliary result} we have
\begin{equation*}
 \frac{1}{g(2^m)} \lambda_{2^m}(T)= \frac{1}{g(2^m)2^m}\sum_{k\in \D_m} g(k) \longrightarrow \log 2.
\end{equation*}
However, if $j_m=2^{m+1}-2$, then $j_m\sim 2\cdot 2^m$, and so $g(j_m)\sim \frac{1}{2} g(2^m)$ by Lemma~\ref{lem:UCT-consequence}. Thus,
\begin{equation*}
  \frac{1}{g(j_m)} \lambda_{j_m}(T)\sim  \frac{2}{g(2^m)2^m}\sum_{k\in \D_m} g(k) \longrightarrow 2 \log 2.
\end{equation*}
It follows that the sequence $g(j)^{-1}\lambda_j(T)$, $j\geq 0$, does not converge, and so $T$ is not spectrally measurable. The proof is complete.
\end{proof}

\begin{remark}
 As $g$-hypermeasurability implies strong measurability, the above example is also an example of strongly measurable operator that is not spectrally measurable. As spectral measurability implies strong measurability,  this shows that spectral measurability is a stronger property than strong measurability. 
\end{remark}

\subsection{Dixmier traces revisited} It is shown in~\cite{LSZ:Book, SSUZ:AIM15} that on the weak trace-class $\sL_{1,\infty}$ under Pietsch's correspondence Dixmier traces correspond to \emph{factorizable} Banach limits. These are Banach limits of the form $\theta=\omega\circ C$, where $\omega$ ranges over all extended limits, and $C:\ell_\infty \rightarrow \ell_\infty$ is the Ces\`aro operator,
\begin{equation}\label{eq:classical Cesaro}
 (Ca)_j = \frac{1}{j+1}\sum_{k\leq j} a_k, \qquad a=(a_j)_{j \geq 0} \in \ell_\infty.
\end{equation}
We refer to~\cite[Example~6.3.6]{LSZ:Book} for the existence of non-factorizable Banach limits on $\ell_\infty$.

The aforementioned characterization of Dixmier traces was extended to weak Lorentz ideals by Levitina--Usachev~\cite{LU:JMAA25}. The role of the Ces\`aro operator is played by the ``weighted Ces\`aro operator'' (or Riesz operator) $C_g:\ell_\infty\rightarrow \ell_\infty$ given by
 \begin{equation*}
 (C_ga)_j = \frac{\sum_{k=0}^{j} 2^k g(2^k)a_k}{\sum_{k=0}^{j}  2^k g(2^k)}, \qquad a=(a_j)_{j \geq 0} \in \ell_\infty.
\end{equation*}
In particular, if $g(t)=t^{-1}$ for $t\geq 1$, then we recover the Ces\`aro operator~(\ref{eq:classical Cesaro}).

If $\omega$ is any extended limit, then it can be shown that $\omega\circ C_g$ is a Banach limit (see Lemma~\ref{lem:Cesaro -banach}). We shall say that such a Banach limit is \emph{$C_g$-factorizable}. We then have the following characterization of Dixmier traces. This extends to weak Lorentz ideals the characterization of Dixmier traces on the weak trace $\sL_{1,\infty}$ in~\cite{LSZ:Book, SSUZ:AIM15}.

 \begin{proposition}[compare {\cite[Theorem~4.7]{LU:JMAA25}}]\label{prop:Pietsch-Dixmier}
 Under the correspondence provided by Theorem~\ref{thm:Pietsch.PNT}, Dixmier traces are exactly the traces that arise from $C_g$-factorizable Banach limits.
\end{proposition}

The above result can be proved by combining \cite[Theorem~4.7]{LU:JMAA25} with Proposition~\ref{prop:Pietsch.Lidskii}. For the sake of completeness and the reader's convenience, we give a direct proof of Proposition~\ref{prop:Pietsch-Dixmier} in Appendix~\ref{app:proof-Pietsch-Dixmier}.

\section{Spectral Characterization of Strong Measurability}\label{sec:strong-hyper}
In this section, we make use of the results of the previous section to get a spectral characterization of strongly measurable operators in weak Lorentz ideals. This extends to weak Lorentz ideals the spectral characterization of strongly measurable operators in $\sL_{1,\infty}$ of Semenov--Sukochev--Usachev--Zanin~\cite[Proposition~7.2]{SSUZ:AIM15}. 

We keep on using the notation of the previous section. In particular, $g(t)$, $t\geq 0$, is a continuous $\RV_{-1}$-function such that $\int_0^\infty g(t)dt=\infty$. 

Recall that a sequence $a=(a_j)_{j\geq 0}\in \ell_\infty$ converges to $c$ if and only if $\omega(a)=c$ for every extended limit $\omega$ (see, e.g., \cite[Lemma~3.15(iii)]{PT:Part1}). We then say that the sequence $(a_j)_{j\geq 0}$ \emph{almost converges} to $c$ if $\theta(a)=c$ for every Banach limit $\theta$, i.e., for every shift-invariant extended limit. 

We have the following (spectral) characterization of almost convergence.

\begin{lemma}[see, e.g., \cite{Su:AMM67}] \label{lem:characterization of ac}
Let $a=(a_j)_{j\geq 0}\in \ell_\infty$. The following are equivalent:
\begin{enumerate}
 \item[(i)] The sequence $(a_j)_{j\geq 0}$ \emph{almost converges} to $L$.

 \item[(ii)] We have
 \begin{equation*}
 \lim_{N\rightarrow \infty} \frac1{N+1}\sum_{j=m}^{m+N} a_j =L,
\end{equation*}
uniformly with respect to $m$.
\end{enumerate}
 \end{lemma}

 \begin{example}
 The sequence $(-1)^j$, $j\geq 1$, almost converges to $0$ (although it does not have any limit).
\end{example}

We have the following spectral characterization of strongly measurable operators in $\sL_g$. This extends to weak Lorentz ideals the spectral characterization of strong measurability in $\sL_{1,\infty}$ of~\cite{SSUZ:AIM15}. 

\begin{theorem}\label{thm:strong measurable--ac}
 If $T\in \sL_g$, then
 \begin{equation*}
 \bigg(T\in \sM_s(\sL_g) \quad \text{and} \quad \gbint T=c\bigg) \Longleftrightarrow \frac{1}{\log 2} \hat{\alpha}_g\big(\lambda(T)\big)_j \ \text{almost converges to}\ c .
\end{equation*}
\end{theorem}
\begin{proof}
 Denote  by $\BL(\N_0)$ the space of Banach limits on $\ell_\infty$ and  by $\NPT(\sL_g)$ the set of $g$-normalized positive traces on $\sL_g$. 
Given any $T\in \sL_g$, by Theorem~\ref{thm:Pietsch.PNT} and Remark~\ref{rmk:Pietsch.independence-g} we have
 \begin{align*}
 \bigg(T\in \sM_s(\sL_g) \ \text{and}\ \gbint T =c\bigg) & \Longleftrightarrow \ \varphi(T) =c \ \text{for all}\ \varphi \in \NPT(\sL_g), \\
 & \Longleftrightarrow \ \frac{1}{\log 2}\theta(\hat{\alpha}_g(\lambda(T)))= c \ \text{for all}\ \theta \in \BL(\N_0),\\
 & \Longleftrightarrow \ \frac{1}{\log 2}\hat{\alpha}_g(\lambda(T))_j \ \text{almost converges to}\ c .
\end{align*}
This proves the result.
\end{proof}

The fact that Theorem~\ref{thm:strong measurable--ac} is formulated in terms of eigenvalue sequences implies the following form of spectral invariance for strong measurability.

\begin{corollary}\label{cor:Strong.spectral-invariance}
If  $S\in \sL_g(\sH)$ and $T\in \sL_g(\sH')$ have the same non-zero spectrum up to multiplicity, then $S$ is strongly measurable if and only if $T$ is strongly measurable. \end{corollary}

In the special case where $\sH'$ is the Hilbert space $\sH$ with an equivalent inner product we obtain the following result.

\begin{corollary}
 The space $\sM_s(\sL_g(\sH))$ does not depend on the inner product of $\sH$.
\end{corollary}

\section{Examples Arising from Nonclassical Weyl Laws}\label{sec:examples}
In this section, we describe various concrete examples of $g$-hypermeasurable operators, and compute their NC integrals. These examples come from a variety of settings. Most of them arise from nonclassical Weyl laws in the sense of Simon~\cite{Si:JFA83}. We also describe a general recipe to construct such examples from operators satisfying nonclassical Weyl laws.

We refer to~\cite{GS:JFA14, GU:JMAA20, LU:JMAA25, PT:Part1, TU:JMAA27} for further examples arising from nonclassical Weyl laws. We stress that in these references the focus is mostly on Dixmier measurability, as well as spectral measurability in~\cite{PT:Part1}, whereas this section exhibits hypermeasurable examples. 

\subsection{General construction of examples}
Let $A:\dom(A)\rightarrow \sH$ be a selfadjoint operator with non-negative spectrum such that
\begin{itemize}
 \item $0$ is isolated in the spectrum of $A$

 \item The non-zero part of the spectrum of $A$ consists of isolated eigenvalues with finite multiplicity.
\end{itemize}
We then can arrange the positive eigenvalues of $A$ as a non-decreasing sequence,
\begin{equation*}
 \lambda_0(A)\leq  \lambda_1(A) \leq \lambda_2(A)\leq \cdots
\end{equation*}
where each eigenvalue is repeated according to multiplicity. We then define the counting function of $A$ by
\begin{equation*}
 N(A;\lambda)=\#\left\{j; \ \lambda_j\left(A\right) <\lambda\right\},  \qquad \lambda>0.
\end{equation*}

For $p>0$ we define $A^{-p}$ to be $f(A)$ with $f(t)=\car_{(0,\infty)}(t)t^{-p}$.  Equivalently, we let $(\xi_j)_{j\geq 0}$ be any orthonormal family in $\sH$ such that $A\xi_j=\lambda_j(A)\xi_j$ for all $j\geq 0$, then we have
\begin{equation*}
 A^{-p}=0 \quad \text{on}\ \ker A \qquad \text{and} \qquad A^{-p}\xi_j= \lambda_j(A)^{-p} \xi_j \quad \text{for all}\ j\geq 0.
\end{equation*}
In particular, $A^{-p}$ is a positive compact operator, and we have
\begin{equation*}
 \lambda_{j}(A^{-p})=\lambda_{j}(A)^{-p} \qquad \forall j\geq 0.
\end{equation*}
Thus, if we define the counting function of $A^{-p}$ by
\begin{equation*}
 N\left(A^{-p};\lambda\right)=\#\left\{j; \ \lambda_j\big(A^{-p}\big) >\lambda\right\},  \qquad \lambda>0,
\end{equation*}
then we have
\begin{equation*}
 N\left(A^{-p};\lambda\right)= N\big(A;\lambda^{-\frac{1}{p}}\big) \qquad \forall \lambda>0.
\end{equation*}

We are interested in the case where $A$ satisfies a Weyl law of the form, 
\begin{equation}\label{eq:example.A-Weyl-law}
  N(A;\lambda)\sim c \lambda^p (\log \lambda)^q, \qquad c>0, \quad p>0, \quad q\geq -1.
\end{equation}
In addition, we set
\begin{equation*}
 g(t)=\frac{1}{t+1}(\log(t+2))^q, \qquad t\geq 0. 
\end{equation*}
This is a continuous $\RV_{-1}$-function. Note that for $q=0$ the ideal $\sL_g$ is just the weak trace-class $\sL_{1,\infty}$. 

The following result is proved in~\cite{PT:Part1}.

\begin{lemma}[\cite{PT:Part1}] \label{lem:example.Ap-spectral-meas}
If $A$ satisfies a Weyl law of the form~(\ref{eq:example.A-Weyl-law}), then $A^{-p}$ is spectrally measurable in $\sL_{g}$, with
 \begin{equation}\label{eq:example.Ap-int-formula}
 \gbint A^{-p} = cp^{-q}.
\end{equation}
\end{lemma}

We shall provide refinements of this result when we have sharper Weyl laws of the form, 
\begin{equation}\label{eq:Examples.A-Weyl-law-remainder}
	N(A;\lambda)= c \lambda^p (\log \lambda)^q \left( 1+ \op{O}\big((\log \lambda)^{-\alpha}\big)\right), \qquad c>0, \quad p>0, \quad q\geq 0, \quad \alpha\geq 1.
\end{equation}

\begin{lemma}\label{lem:Examples.NAL-remainder1}
 Suppose that $q=0$ and $A$ satisfies the Weyl law~(\ref{eq:Examples.A-Weyl-law-remainder}) with $\alpha>1$. Then, the following hold.
\begin{enumerate}
 \item[(i)] We have
 \begin{equation*}
  \lambda_j\big(A^{-p}\big) = \frac{c}{j}  \left( 1+ \op{O}\big((\log j)^{-\alpha}\big)\right).
\end{equation*}

\item[(ii)] The operator $A^{-p}$ is $(t+1)^{-1}$-hypermeasurable.
\end{enumerate}
\end{lemma}
\begin{proof}
As $\lambda \rightarrow 0^+$ we have
\begin{equation*}
   N\left(A^{-p};\lambda\right)= N\big(A;\lambda^{-\frac{1}{p}}\big) = c\lambda^{-1} \left( 1+ \op{O}\big(|\log \lambda|^{-\alpha}\big)\right).
\end{equation*}
Applying Lemma~\ref{lem:app.example-tp} then shows that as $j\rightarrow \infty$ we have
\begin{equation*}
  \lambda_j\big(A^{-p}\big) = \frac{c}{j}  \left( 1+ \op{O}\big((\log j)^{-\alpha}\big)\right).
\end{equation*}
This gives the first part. Furthermore, as $\alpha>1$ Corollary~\ref{cor:Hyper.Weyl-hyper.sLp} ensures that $A^{-p}$ is $(t+1)^{-1}$-hypermeasurable.  The proof is complete.
\end{proof}

\begin{lemma}\label{lem:distribution asymptotic with remainder}
 Suppose that $q>0$ and  $A$ satisfies the sharper Weyl law~(\ref{eq:Examples.A-Weyl-law-remainder}) with $\alpha=1$. The following hold.
\begin{enumerate}
 \item[(i)] We have
 \begin{equation*}
 \lambda_j\big(A^{-p}\big) =cp^{-q} \frac{1}{j} (\log j)^q\left(1- q^2\frac{\log\log j}{\log j}+ \op{O}\left (\frac{1}{\log j}\right)\right).
\end{equation*}

\item[(ii)] The operator $A^{-p}$ is not $g$-hypermeasurable, but it is hypermeasurable with respect to any (continuous) $\RV_{-1}$-function $g_1(t)$ such that
\begin{equation}\label{eq:g_1}
 g_1(t)=\frac{1}{t}(\log t)^q\left(1-q^2\frac{\log\log t}{\log t} + \op{O}\left(\frac{1}{\log t}\right)\right).
\end{equation}
 \end{enumerate}
\end{lemma}
\begin{proof}
As $\lambda \rightarrow 0^+$ we have
\begin{equation*}
N\left(A^{-p};\lambda\right)= N\big(A;\lambda^{-\frac{1}{p}}\big) = cp^{-q} \lambda^{-1} |\log \lambda|^q  \left( 1+ \op{O}\big(|\log \lambda|^{-1}\big)\right).
\end{equation*}
Combining this with Lemma~\ref{lem:app.example-tplogq} shows that as $j\rightarrow \infty$ we have
\begin{equation}\label{eq:example.Weyl-Ap-log}
 \lambda_j\big(A^{-p}\big) = cp^{-q} \frac{1}{j} (\log j)^q\left(1- q^2\frac{\log\log j}{\log j}+  \op{O}\left (\frac{1}{\log j}\right)\right).
\end{equation}
This gives the first part. Combining this with Corollary~\ref{cor:Hyper.Weyl-hyper.sLpq} further shows that $A^{-p}$ is not $g$-hypermeasurable.

Finally, if $g_1(t)$ is any continuous $\RV_{-1}$-function satisfying~(\ref{eq:g_1}), then we can rewrite the asymptotic~(\ref{eq:example.Weyl-Ap-log}) in the form,
\begin{align*}
  \lambda_j\big(A^{-p}\big) &= cp^{-q} j^{-1} (\log j)^q\left(1- q^2\frac{\log\log j}{\log j}\right) \left(1+ \op{O}\left (\frac{1}{\log j}\right)\right) \\
  &= cp^{-q} g_1(j) \left(1+ \op{O}\left (\frac{1}{\log j}\right)\right).
\end{align*}
Corollary~\ref{cor:Hyper.Weyl-hyper.sLpq} then ensures that $A^{-p}$ is $g_1$-hypermeasurable. The proof is complete.
\end{proof}

Often Weyl laws of the form~(\ref{eq:example.A-Weyl-law}) can be deduced from Tauberian theorems (see, e.g., \cite{Ar:HMJ83,  De:ASENS54, Ko:Springer04}). There are further refinements of Ikehara's Tauberian theorem that provide nonclassical Weyl laws with lower order remainders of the type considered in Lemma~\ref{lem:Examples.NAL-remainder1} and Lemma~\ref{lem:distribution asymptotic with remainder} (see, e.g., \cite{Ar:AMH96, PTZ:arXiv25} and the references therein).  In particular, by using the recent refinement of Pierce~\emph{et al.}~\cite[Theorem~B]{PTZ:arXiv25} we get the following result. 

\begin{lemma}[\cite{PTZ:arXiv25}]\label{lem:PTZ}
Assume that there are real numbers $\delta>0$ and $a_1, \ldots, a_q$ with $a_q>0$, such that the function
\begin{equation*}
 \zeta(A;s) - \sum_{1\leq \ell \leq q} \frac{a_\ell}{(s-p)^{\ell}}, \qquad \Re s>p, 
\end{equation*}
has an analytic extension to a neighborhood of the half-plane $\Re s\geq p-\delta$. Assume further there is $\kappa>0$ such that 
\begin{equation*}
\left|\zeta(A;p-\delta+iy)\right| = \op{O}\left(|y|^\kappa \left(\log |y|\right)^{q-1}\right)\qquad \text{as}\ |y|\rightarrow \infty. 
\end{equation*}
	Then, we have 
\begin{equation*}
 N(\lambda) = \sum_{0\leq k \leq q-1} b_k \lambda^p (\log \lambda)^{k} + \op{O}\left(\lambda^{p-\frac{\delta}{\kappa+1}}(\log \lambda)^{q-1}\right), 
\end{equation*}
where 
\begin{equation}\label{eq:PTZ-coeff}
 b_k = \frac{1}{k!} \sum_{m=0}^{q-k-1} \frac{1}{m!}\frac{(-1)^m}{p^{m+1}} a_{m+k+1}, \qquad k=0, \ldots, q-1. 
\end{equation}
\end{lemma}

\begin{remark}
 For $k=q-1$ the formula~(\ref{eq:PTZ-coeff}) reduces to 
 \begin{equation*}
 b_{q-1}= \frac{1}{(q-1)!}  \frac{1}{p}a_q. 
\end{equation*}
\end{remark}

\subsection{Zeros of Riemann's zeta function (Riemann Hypothesis)}
Assume the Riemann Hypothesis (RH) holds, and let $\sD$ be the ``Dirac operator'' whose spectrum consists of the imaginary parts of non-trivial zeros of the Riemann zeta function  $\zeta(s):=\sum_{n\geq 1} n^{-s}$ (see~\cite{Co:AFA24, CM:PNAS22}). Thus, $\sD$ is a selfadjoint unbounded operator on some suitable Hilbert space. Moreover, as the zeros of $\zeta(s)$ are symmetric with respect to the real axis, we see that the counting function $N(|\sD|;\lambda)$, $\lambda>0$, is two times the number of zeros of
$\zeta(s)$ with imaginary part in $[0,\lambda)$. 

Set $g(t)=(t+1)^{-1}\log (t+2)$, $t\geq 0$. The following result is shown in~\cite{PT:Part1}. 

\begin{proposition}[\cite{PT:Part1}] \label{prop:RH-spectral-meas}
 If (RH) holds, then $|\sD|^{-1}$ is spectrally measurable in $\sL_g$, with 
 \begin{equation*}
 \gbint |\sD|^{-1} = \frac{1}{\pi}.
\end{equation*}
\end{proposition}

This result is a direct consequence of the Riemann-von Mangoldt formula (see, e.g., \cite{Ti:Zeta}) and Lemma~\ref{lem:example.Ap-spectral-meas}. In fact, the Riemann-von Mangoldt formula implies that, 
 as $\lambda\rightarrow \infty$, we have
\begin{align*}
 N\big(|\sD|;\lambda\big) & = 2\frac{\lambda}{2\pi}\left( \log\left(\frac{\lambda}{2\pi}\right)-1\right)+\op{O}(\log \lambda)\\
 & = \frac{1}{\pi} \lambda \log \lambda\left( 1+ \op{O}\left(\frac{1}{\log \lambda} \right)\right).
\end{align*}
Applying Lemma~\ref{lem:distribution asymptotic with remainder} then gives the following refinement of Proposition~\ref{prop:RH-spectral-meas}. 

\begin{proposition}\label{prop:RH}
 The following hold.
 \begin{enumerate}
 \item We have
 \begin{equation}\label{eq:asymptotic of RH dirac}
\lambda_j\left(|\sD|^{-1}\right)=  \frac{1}{\pi}  \frac{\log j}{j}\left(1- \frac{\log\log j}{\log j}+ \op{O}\left(\frac{1}{\log j}\right)\right).
\end{equation}

 \item The operator $|\sD|^{-1}$ is not $g$-hypermeasurable, but it is hypermeasurable with respect to any  
continuous $\RV_{-1}$-function $g_1(t)$ such that
\begin{equation}\label{eq:Examples.g1-RH}
	g_1(t)=\frac{1}{t}\log t \left(1-\frac{\log\log t}{\log t} + \op{O}\left(\frac{1}{\log t}\right)\right).
\end{equation}
\end{enumerate}
\end{proposition}

\begin{remark}
 The symmetry between positive and negative eigenvalues of $\sD$ ensures that $\lambda_j^\pm(\sD)=\lambda_{2j}(|\sD|)$ for all $j\geq 0$. Therefore, the positive and negative eigenvalues satisfy asymptotics of the form~(\ref{eq:asymptotic of RH dirac}). It then follows that $\sD^{-1}$ is spectrally measurable in $\sL_g$ with $g(t)=(t+1)^{-1}\log (t+2)$. The symmetry further ensures that $\sH$ has an orthogonal splitting $\sH=\sH^+\oplus \sH^-$, and we can find an orthonormal eigenbasis $(\xi_j^\pm)$ of $\sH^{\pm}$ in such a way that
 $\sD^{-1}\xi_j^\pm =\pm \lambda_{2j}(|\sD|)\xi_j^\pm$. If $J$ is the involution on $\sH$ such that $J\xi_j^\pm=\xi_j^\mp$, then we have $J\sD^{-1}J=-\sD^{-1}$. Equivalently,
 \begin{equation*}
 \sD^{-1}= \frac{1}{2}\big(\sD^{-1}-J\sD^{-1}J\big)= \frac12\big[J,J\sD^{-1}\big].
\end{equation*}
This shows that $\sD^{-1}$ is in the commutator space $\Com(\sL_g)$, and so it annihilates \emph{every} trace on $\sL_g$. In particular, the operator $\sD^{-1}$ is $g$-hypermeasurable. More generally, it is $g_1$-hypermeasurable for any $\RV_{-1}$-function $g_1(t)\sim g(t)$.
\end{remark}

\subsection{Logarithm of the Laplacian}
Let $\Delta_g$ be the Laplacian on a closed manifold $(M^n,g)$. We know (see, e.g.,~\cite{Ho:ActaM68}) that as $j\rightarrow \infty$ we have
\begin{equation}\label{eq:eigenvalue of Delta_g}
  \lambda_j\left(\Delta_g\right)= \left(\frac{j}{c(n)\Vol_g(M)}\right)^{\frac{2}{n}}\left(1 +\op{O}\big(j^{-\frac1{n}}\big)\right), \qquad c(n):=(2\pi)^{-n}|\bB^n|.
\end{equation}

Set $T= (\log \Delta_g)\Delta_g^{-n/2}$. Let $(\xi_j)_{j\geq 0}$ be an orthonormal family such that $\Delta_g\xi_j=\lambda_j(\Delta_g)\xi_j$ for all $j\geq 0$. We have
\begin{equation*}
T=0 \quad \text{on}\ \ker \Delta_g \qquad \text{and} \qquad T\xi_j=\log\big( \lambda_j(\Delta_g)\big)\lambda_j(\Delta_g)^{-\frac{n}{2}}\xi_j \quad \forall j\geq 0.
\end{equation*}

Set $n_1=N(\Delta_g;1)$. We have
\begin{equation*}
 \lambda_j^{-}(T)= \left\{
 \begin{array}{cl}
 \big|\log\big( \lambda_{j}(\Delta_g)\big)\big|\lambda_{j}(\Delta_g)^{-\frac{n}{2}}
 & \text{if}\ j <n_1,\\
 0 & \text{if}\ j\geq n_1.
\end{array}\right.
\end{equation*}
In particular, as $j\rightarrow \infty$ we have
\begin{equation*}
 \lambda_j^{-}(T) = \op{O}\big(j^{-r}\big) \qquad \forall r>0.
\end{equation*}

Set $\overline{n}_1=n_1 +\dim \ker (\Delta_g-1)$. We have
\begin{equation*}
  \lambda_j^{+}(T)= \log\big( \lambda_{j+\overline{n}_1}(\Delta_g)\big) \lambda_{j+\overline{n}_1}(\Delta_g)^{-\frac{n}{2}} \qquad \forall j\geq 0.
\end{equation*}
Combining this with~(\ref{eq:eigenvalue of Delta_g}) we get
\begin{align*}
 \lambda_j^{+}(T) & = \log\left[ ((c(n)\Vol_g(M))^{-1}(j+\overline{n}_1))^{\frac{2}{n}}\right] c(n)\Vol_g(M)(j+\overline{n}_1)^{-1}\left(1 +\op{O}\big(j^{-\frac1{n}}\big)\right)\\
  &= \frac{2}{n} c(n)\Vol_g(M) \frac{\log j}{j}\left(1 +\op{O}\big(j^{-\frac1{n}}\big)\right).
\end{align*}
Therefore, by applying~\cite[Theorem~5.8]{PT:Part1} and Corollary~\ref{cor:Hyper.Weyl-hyper.sLpq} we obtain the following integration result.

\begin{proposition}\label{prop:log Laplacian}
 Set $h(t)=(t+1)^{-1}\log(t+2)$, $t\geq 0$. The operator $(\log\Delta_g)\Delta_g^{-n/2}$ is spectrally measurable  and $h$-hypermeasurable in $\sL_h$, with
 \begin{equation*}
 \hbint \left(\log\Delta_g\right)\Delta_g^{-\frac{n}{2}} =   \frac{2}{n}c(n)\Vol_g(M).
\end{equation*}
\end{proposition}

\subsection{Double Laplacian} Let $(M^n,g)$ be a closed Riemannian manifold with Laplacian $\Delta_g$. We equip the double manifold $M\times M$ with the product metric $g\otimes g$. In what follows we shall use letters $x$, $\xi$ to denote variables arising from the first factor and letters $y$, $\eta$ for variables arising from the second factor.

On $M\times M$ we consider the \emph{double Laplacian},
\begin{equation}\label{eq:double Laplacian}
 \Delta_g\otimes \Delta_g.
\end{equation}
This is a 4th order differential operator. Its principal symbol is
\begin{equation*}
 \sigma_4(x,\xi;y,\eta)= |\xi|^2_{g(x)}|\eta|^2_{g(y)}, \qquad \xi\in T_x^*M, \quad \eta\in T_y^*M.
\end{equation*}
In particular, it is a quadratic form of signature $(1,1)$ in $(|\xi|^2_{g(x)},|\eta|^2_{g(y)})$. Note also that $\Delta_g\otimes \Delta_g$ annihilates every function that does not depend on one of the two variables $x,y$. Therefore,  $ \Delta_g\otimes \Delta_g$ is not elliptic or even hypoelliptic.

General results on tensor products of selfadjoint unbounded operators (see, e.g., \cite[\S7.5]{Sc:Springer12}) ensure that  $ \Delta_g\otimes \Delta_g$ is essentially selfadjoint. It has a non-negative spectrum with an infinite-dimensional nullspace. The positive part of its spectrum is discrete and consists of the products of the eigenvalues of $\Delta_g$. More precisely, we have
\begin{equation*}
 \ker  \big( \Delta_g\otimes \Delta_g\big) = \left[\big(\ker \Delta_g\big)\otimes L^2_g(M)\right] +  \left[L^2_g(M)\otimes \big(\ker \Delta_g\big)\right].
\end{equation*}
Moreover, if  $(\xi_j)_{j\geq 0}$ is any orthonormal basis of $L_g^2(M)$ such that $\Delta_g \xi_j = \lambda_j(\Delta_g)\xi_j$ for all $j\geq 0$, then we have
\begin{equation*}
 \big( \Delta_g\otimes \Delta_g\big)(\xi_k\otimes \xi_\ell)= \lambda_k(\Delta_g) \lambda_\ell(\Delta_g)(\xi_k\otimes \xi_\ell) \qquad \forall k,\ell\geq 0.
\end{equation*}

Given any $s\in \C$, we then have
\begin{gather}
 \big( \Delta_g\otimes \Delta_g\big)^{-s}= 0 \qquad \text{on}\ \ker\big( \Delta_g\otimes \Delta_g\big),\label{eq:1st tensor product}\\
\big( \Delta_g\otimes \Delta_g\big)^{-s}(\xi_k\otimes \xi_\ell)= \lambda_k(\Delta_g)^{-s} \lambda_\ell(\Delta_g)^{-s} (\xi_k\otimes \xi_\ell) \qquad \forall k,\ell\geq 0.\label{eq:2nd tensor product}
\end{gather}
Equivalently,
\begin{equation}\label{eq:tensor identity}
 \big( \Delta_g\otimes \Delta_g\big)^{-s}=  \big( \Delta_g\big)^{-s}\otimes \big( \Delta_g\big)^{-s}.
\end{equation}
In particular, as $\Delta_g^{-s}$ is trace-class for $\Re s>n/2$, we see that $( \Delta_g\otimes \Delta_g)^{-s}$ is trace-class for $\Re s>n/2$ as well.

Let $\zeta(\Delta_g;s)=\Tr[\Delta_g^{-s}]$ and $\zeta(\Delta_g\otimes \Delta_g;s)=\Tr[(\Delta_g\otimes \Delta_g)^{-s}]$ be the respective zeta functions of
$\Delta_g$ and $\Delta_g\otimes \Delta_g$. By~(\ref{eq:1st tensor product})--(\ref{eq:2nd tensor product}) we have
\begin{align*}
 \zeta\big(\Delta_g\otimes \Delta_g;s\big) &= \sum_{k,\ell\geq n_0} \lambda_k(\Delta_g)^{-s}\lambda_\ell(\Delta_g)^{-s}\\
  &=  \bigg(\sum_{k\geq n_0} \lambda_k(\Delta_g)^{-s}\bigg) \bigg(\sum_{\ell\geq n_0} \lambda_\ell(\Delta_g)^{-s}\bigg) \\
  &=\zeta(\Delta_g;s)^2,
\end{align*}
where $n_0$ is the least integer such that $\lambda_j(\Delta_g)>0$ for all $j\geq n_0$.
We know (see, e.g., \cite{Sh:Springer01}) that $\zeta(\Delta_g;s)$ has a meromorphic extension to $\C$ with at worst \emph{simple} pole singularities at non-zero half-integers and integers~$\leq n/2$. Moreover, the residue at $s=n/2$ is given by
\begin{equation}\label{eq:Intro.Residue-Deltag}
 \Res_{s=\frac{n}{2}} \zeta(\Delta_g;s)= \frac{n}{2}c(n) \Vol_g(M). 
\end{equation}
It follows that $\zeta(\Delta_g\otimes \Delta_g;s)$  has a meromorphic extension to $\C$ with at worst \emph{double} pole singularities at non-zero integers and half-integers~$\leq n/2$. Moreover, at $s=n/2$ we have
\begin{equation*}
  \Res_{s=\frac{n}{2}} \big(s-\frac{n}{2}\big)\zeta(\Delta_g\otimes\Delta_g;s)= \left(\frac{n}{2}\right)^2c(n)^2 \Vol_g(M)^2. 
\end{equation*}

We also recall the following result. 

\begin{lemma}[see~{\cite[Corollary~2.2]{DG:IM75}}]\label{lem:Example.DG}
The following hold.
\begin{enumerate}
 \item The zeta function $\zeta(\Delta_g;s)$ is uniformly bounded on vertical lines $\Re s=c$, $c>n/2$. 
 
 \item For any $k=0,1,2,\ldots$ and $\frac12(n-k-1)<c<\frac12(n-k)$ we have
\begin{equation*}
\left| \zeta(\Delta_g;c+iy)\right| =\op{O}\left(|y|^{k+1}\right) \qquad \text{as}\ |y|\rightarrow \infty. 
\end{equation*}
\end{enumerate}
 \end{lemma}

It follows from Lemma~\ref{lem:Example.DG} that, if $\frac1{2}(n-1)<c<\frac1{2}n$, then 
\begin{equation*}
 \left| \zeta(\Delta_g\otimes\Delta_g;c+iy)\right| =\left| \zeta(\Delta_g;c+iy)\right|^2 =\op{O}\left(|y|^2\right) \qquad \text{as}\ |y|\rightarrow \infty. 
\end{equation*}
Therefore, the assumptions of Lemma~\ref{lem:PTZ} are satisfied. We then obtain that, for any $\alpha \in (0,1/6)$, we have
 \begin{equation}\label{eq:weyl law tensor case}
 N\big(\Delta_g\otimes \Delta_g;\lambda\big) = b_1 \lambda^{\frac{n}{2}} \log \lambda + b_0 \lambda^{\frac{n}{2}} + \op{O}\big(\lambda^{\frac{n}{2}-\alpha}\big), 
\end{equation}
where the coefficients $b_0$ and $b_1$ are computed by using~(\ref{eq:PTZ-coeff}). In particular, 
\begin{equation*}
b_1= \frac{n}{2}c(n)^2\Vol_g(M)^2.
\end{equation*}
It follows that 
\begin{equation*}
 N\big(\Delta_g\otimes \Delta_g;\lambda\big) =  \frac{n}{2}c(n)^2\Vol_g(M)^2 \lambda^{\frac{n}{2}}\log \lambda \left( 1+ \op{O}\left(\frac{1}{\log \lambda}\right)\right).  
\end{equation*}
Combining this with Lemma~\ref{lem:example.Ap-spectral-meas} and Lemma~\ref{lem:distribution asymptotic with remainder} we then arrive at the following statement. 

\begin{proposition}\label{prop:Examples.Double-Laplacian}
Set $h(t)=(t+1)^{-1}\log(t+2)$,  $t\geq 0$. The following hold.
\begin{enumerate}
\item We have
 \begin{equation*}
	\lambda_j((\Delta_g\otimes\Delta_g)^{-\frac{n}{2}})=c(n)^2\Vol_g(M)^2 \frac{\log j}{j}\left(1- \frac{\log\log j}{\log j}+ \op{O}\left(\frac{1}{\log j}\right)\right).
\end{equation*}
\item The operator $(\Delta_g\otimes\Delta_g)^{-\frac{n}{2}}$ is spectrally measurable in $\sL_h$, with
\begin{equation*}
	\hbint   \big( \Delta_g\otimes \Delta_g\big)^{-\frac{n}{2}}=c(n)^2 \Vol_{g}(M)^2.
\end{equation*}
\item It is not $h$-hypermeasurable,  but it is hypermeasurable with respect to any  
$\RV_{-1}$-function $h_1(t)$ of the form~(\ref{eq:Examples.g1-RH}). 
\end{enumerate}
\end{proposition}

\begin{remark}
 Battisti~\cite{Ba:MZ12} obtained a Weyl law for bi-singular operators on closed manifolds, including the double Laplacian $\Delta_g\otimes \Delta_g$. Battisti's result was actually the main impetus for this example. We stress that~(\ref{eq:weyl law tensor case}) provides us with a sharper version of Battisti's Weyl law. 
\end{remark}

\subsection{Multi-tensor products of Laplacians}
This is a generalization of the previous example.  For $i=1,\ldots, r$ let $(M_i^{n_i},g_i)$ be closed Riemannian manifolds with $n=n_1\geq n_2\geq \cdots \geq n_r$ and $r\geq 2$. We equip the product manifold $M_1\times \cdots \times M_r$ with the product metric $g_1\otimes \cdots \otimes g_r$. For $i=1,\ldots, r$ let $\Delta_{g_i}$ be the Laplacian of $(M_i,g_i)$. On $M_1\times \cdots \times M_r$ we consider the tensor product of Laplacians,
\begin{equation*}
A:=\Delta_{g_1}\otimes \cdots \otimes \Delta_{g_r}.
\end{equation*}
This is a non-hypoelliptic differential operator of order $2r$. As with the double Laplacian~(\ref{eq:double Laplacian}) this is an essentially selfadjoint operator with non-negative spectrum and an infinite-dimensional nullspace. The positive part of its spectrum is discrete and consists of the products of the positive eigenvalues of the $\Delta_{g_i}$. As in~(\ref{eq:tensor identity}) we have
\begin{equation*}
A^{-s}= \big(\Delta_{g_1}\big)^{-s}\otimes \cdots \otimes \big(\Delta_{g_r}\big)^{-s}, \qquad s\in\mathbb{C}.
\end{equation*}
The zeta function of $A$ then satisfies
\begin{equation*}
 \zeta\big(A;s\big)=  \zeta\big( \Delta_{g_1};s\big)\cdots  \zeta\big( \Delta_{g_r};s\big).
\end{equation*}
It follows that it has a meromorphic extension to $\C$ with at worst pole singularities of order~$\leq r$ at non-zero half-integers and integers~$\leq n/2$.

The order of the pole at $s=n/2$ is equal to
\begin{equation}\label{eq:def of q}
 q:=\#\{i;\ n_i=n\}.
\end{equation}
In particular, we have a simple pole if and only if $n_2<n$. In any case, we have
\begin{equation*}
 \Res_{s=\frac{n}{2}}  \big(s-\frac{n}{2}\big)^{q-1}\zeta\big(A;s\big) = \left(\frac{n}{2}\right)^q c(n)^q \Vol_{g_1}(M_1)\cdots \Vol_{g_q}(M_q),
\end{equation*}
where as in~(\ref{eq:Intro.Residue-Deltag}) we have set $c(n)=(2\pi)^{-n}|\bB^n|$. Moreover, Lemma~\ref{lem:Example.DG} ensures that if $\frac1{2}(n-1)<c<\frac1{2}n$, then 
\begin{equation*}
 \left| \zeta(\Delta_{g_i};c+iy)\right| =  \left\{ 
 \begin{array}{ll} 
  \op{O}\left(|y|\right) & \textup{for}\ 1\leq i \leq q, \smallskip \\
   \op{O}(1) & \textup{for}\ q< i \leq r. 
\end{array}
\right.  
\end{equation*}
Thus, 
\begin{equation*}
  \left| \zeta\big(A;s\big)\right|=   \left|\zeta\big( \Delta_{g_1};s\big)\right|\cdots  \left|\zeta\big( \Delta_{g_r};s\big)\right| =  \op{O}\left(|y|^q\right). 
\end{equation*}
This allows us to apply Lemma~\ref{lem:PTZ}. For any $\alpha \in (0, \frac{1}{2}(q+1)^{-1})$, we have the Weyl law, 
\begin{equation}\label{eq:Examples.WL-multiL}
 N(A; \lambda) = \sum_{k=0}^{q-1} b_k \lambda^{\frac{n}{2}} (\log \lambda)^k + \op{O}\big( \lambda^{\frac{n}{2}-\alpha}\big), 
  \end{equation}
where the coefficients $b_k$, $k=0, \ldots, q-1$, are computed by using~(\ref{eq:PTZ-coeff}). In particular, 
\begin{gather*}
 b_{q-1} =  \frac{1}{(q-1)!} \left(\frac{n}{2}\right)^{q-1}c(n)^q \Vol_{g_1}(M_1)\cdots \Vol_{g_q}(M_q). 
\end{gather*}
 It follows that 
\begin{gather*}
  N(A; \lambda) = b_{q-1} \lambda^{\frac{n}{2}} (\log \lambda)^{q-1} 
 \left(1+  \op{O}\left(\varepsilon(\lambda)\right)\right), 
\end{gather*}
where $ \varepsilon(\lambda)= \lambda^{-\alpha} $ if $q=1$,  or $ \varepsilon(\lambda)= (\log \lambda)^{-1} $ if $q\geq 2$. Therefore, by combining this with Lemma~\ref{lem:example.Ap-spectral-meas}, and either Lemma~\ref{lem:Examples.NAL-remainder1} if $q=1$, or Lemma~\ref{lem:distribution asymptotic with remainder} if $q\geq 2$, we arrive at the following statement. 

\begin{proposition}\label{prop:tensor Laplacians}
Set  $h(t)=(t+1)^{-1}(\log (t+2))^{q-1}$, $t\geq 0$, where $q$ is given by~(\ref{eq:def of q}). The following hold.
 \begin{enumerate}
 	\item The operator $( \Delta_{g_1}\otimes \cdots \otimes \Delta_{g_r})^{-\frac{n}{2}}$ is spectrally measurable in $\sL_h$, with
 	\begin{equation*}
 		\hbint  \left(\Delta_{g_1}\otimes \cdots \otimes \Delta_{g_r}\right)^{-\frac{n}{2}}=\frac{1}{(q-1)!}c(n)^{q}\,
		\Vol_{g_1}(M_1)\cdots \Vol_{g_q}(M_q). 
 	\end{equation*}

          \item If $q=1$, then $( \Delta_{g_1}\otimes \cdots \otimes \Delta_{g_r})^{-\frac{n}{2}}$ is $(t+1)^{-1}$-hypermeasurable.\smallskip 

 	\item If $q\geq 2$, then $( \Delta_{g_1}\otimes \cdots \otimes \Delta_{g_r})^{-\frac{n}{2}}$ is not $h$-hypermeasurable, but it is hypermeasurable with respect to any  
 	 $\RV_{-1}$-function $h_1(t)$ such that
 	\begin{equation*}
 		h_1(t)=\frac{1}{t}(\log t)^{q-1}\left(1-(q-1)^2\frac{\log\log t}{\log t} +  \op{O}\left(\frac{1}{\log t}\right)\right).
 	\end{equation*}
 \end{enumerate}
\end{proposition}

\begin{remark}
Battisti~\emph{et al.}~\cite{BBC:AMPA16} established (classical) Weyl laws for tensor products of elliptic pseudodifferential operators on closed manifolds under suitable assumptions. In the case of $A:=\Delta_{g_1}\otimes \cdots \otimes \Delta_{g_r}$ those assumptions amount to requiring $q=1$. We note that even in this special case the Weyl law~(\ref{eq:Examples.WL-multiL}) provides a sharper form of the Weyl law of~\cite{BBC:AMPA16}.
\end{remark}

\appendix

\section{Counting Functions and Eigenvalues} \label{app:counting}
In~\cite[Appendix~B]{PT:Part1} we explained the link between the asymptotic behaviours of non-increasing non-negative null sequences and  the asymptotic behaviours of their counting functions when the leading terms are given by $\RV$-functions. In this appendix, we further study this link in order to deal with sharp asymptotics with suitable remainder terms. 

Recall that, if $h:[0,\infty)\rightarrow [0,\infty)$ is a function converging to $\infty$ as $t\rightarrow \infty$, then an \emph{asymptotic inverse} is any function
$h^\sharp:[0,\infty)\rightarrow [0,\infty)$ converging to $\infty$ as $t\rightarrow \infty$ such that
 \begin{equation*}
 h(h^\sharp(t)) \sim t \qquad \text{and} \qquad h^\sharp(h(t)) \sim t  \qquad \text{as}\ t\rightarrow \infty.
 \end{equation*}
If $h(t)$ is an $\RV_p$-function with $p>0$, then it always admits $h^\sharp(t)$ which is $\RV_{1/p}$, and a function $\tilde{h}(t)$ is an asymptotic inverse of $h(t)$ if and only if $\tilde{h}(t)\sim h^\sharp(t)$ (see~\cite[Theorem~5.1.12]{BGT:Cambridge87}). Moreover, we always can choose $h^\sharp(t)$ to be continuous (see~\cite[Remark~B.2]{PT:Part1}). 

\subsection{Principal term asymptotics}
As in~\cite[Appendix~B]{PT:Part1} we let
\begin{equation*}
 \lambda_0\geq \lambda_1 \geq \lambda_2 \geq \cdots \geq 0
\end{equation*}
be a non-increasing non-negative sequence converging to $0$. Its counting function is given by
\begin{equation*}
 N(\lambda):= \#\{j; \lambda_j>\lambda\}, \qquad \lambda>0.
\end{equation*}
In addition, we let $h:[0,\infty)\rightarrow [0,\infty)$ be an $\RV_p$-function with $p>0$, and let $h^\sharp: [0,\infty)\rightarrow [0,\infty)$ be an
$\RV_{1/p}$ asymptotic inverse.

\begin{lemma}[see~\cite{PT:Part1}]\label{lem:app.principal-equiv}
We have
\begin{equation}\label{eq:asymptotic equivalence}
 \lim_{\lambda\rightarrow 0^+} h(\lambda^{-1})^{-1}N(\lambda)=c \ \Longleftrightarrow \ \lim_{j\rightarrow \infty} h^\sharp(j)\lambda_j =c^{\frac1p}.
\end{equation}
\end{lemma}

\begin{example}
 If $h(t)=t^p$, $p>0$, then an exact inverse of $h(t)$ is $h^\natural(t):=t^{\frac1p}$, $t\geq 0$, which is an $\RV_{\frac1p}$-continuous function. Thus, 
 \begin{equation*}
 \lim_{\lambda\rightarrow 0^+} \lambda^pN(\lambda)=c \ \Longleftrightarrow \ \lim_{j\rightarrow \infty} j^{\frac1p}\lambda_j =c^{\frac1p}.
\end{equation*}
\end{example}

\begin{example}[see~\cite{PT:Part1}]\label{ex:asymp inverse}
 If  $h(t)=t^p(\log(t+2))^q$ with $p>0$ and $q\neq 0$, then a continuous $\RV_{1/p}$ asymptotic inverse is given by
\begin{equation*}
 h^\sharp(t):= p^{\frac{q}{p}} t^{\frac{1}{p}} \left(\log(t+2)\right)^{-\frac{q}{p}}, \qquad t\geq 0.
\end{equation*}
Therefore, we have
\begin{equation}\label{eq:App.principa-asymp-tplogtq}
 \lim_{\lambda\rightarrow 0^+} \lambda^p|\log \lambda|^qN(\lambda)=c \ \Longleftrightarrow \ \lim_{j\rightarrow \infty} j^{\frac1p} (\log j)^{-\frac{q}{p}}\lambda_j =
 \left(p^{-q}c\right)^{\frac1p}.
\end{equation}
\end{example}

\subsection{Sharp asymptotics}
It follows from~(\ref{eq:asymptotic equivalence}) that if
\begin{equation*}
 N(\lambda)=h(\lambda^{-1})\left(c+\op{o}(1)\right) \qquad \textup{as}\ \lambda\rightarrow 0^+,
\end{equation*}
then, as $j\rightarrow \infty$ we have
\begin{equation*}
 \lambda_j =h^\sharp(j)^{-1}\left(c^{\frac1p}+\op{o}(1)\right).
\end{equation*}
To look at sharper asymptotic expansions with better remainder terms we make the following additional assumptions:
\begin{enumerate}
 \item[(i)] There is $q>p$ such that $t^{-q}h(t)$ is ultimately decreasing.

 \item[(ii)] For all $a>0$, we have
 \begin{equation}\label{cond:ii}
 h(at)=a^p h(t)\left(1+\op{O}(\epsilon(t))\right) \qquad \textup{as}\ t\rightarrow \infty.
\end{equation}

\item[(iii)] As $t\rightarrow \infty$, we have
\begin{equation}\label{cond:hsharp}
 h^\sharp\circ h(t)= t\left(1+\op{O}(\epsilon(t))\right).
\end{equation}
\end{enumerate}
In the conditions (ii)--(iii) $\epsilon: [t_0,\infty)\rightarrow (0,\infty)$ is a given decreasing null function which is $\RV_\rho$ for some $\rho\leq 0$.

The conditions (i)--(ii) are satisfied by the functions $h(t)=t^p(\log (t+2))^q$, $p>0$, $q\in \R$ (see below). Moreover, the condition~(iii) is automatically satisfied if we choose $h^\sharp(t)$ so that it inverts $h(t)$ on some interval $[b,\infty)$ (which is always possible; see~\cite[Remark~B.2]{PT:Part1}).

\begin{lemma}\label{lem:app.Ntolambda1}
Assume that the conditions (i)--(iii) are satisfied, and as $ \lambda\rightarrow 0^+$ we have
\begin{equation*}
 N(\lambda)= h(\lambda^{-1}) \left[ c+\op{O}\left(\epsilon(\lambda^{-1})\right)\right]  \qquad \textup{with}\ c>0. 
\end{equation*}
Then,  we have
\begin{equation}\label{eq:case c>0 result}
 \lambda_j=c^{\frac1p} h^\sharp(j)^{-1} \left[ 1+\op{O}\left((\epsilon \circ h^\sharp)(j)\right)\right] \qquad \textup{as}\ j\rightarrow \infty.
\end{equation}
\end{lemma}
\begin{proof}
Let us first assume that $c=1$. Let $h^\natural:(0,\infty)\to(0,\infty)$ be a continuous function inverting $h(t)$ on some interval $[t_0,\infty)$. This is an $\RV_{1/p}$-function.
Set \begin{equation}\label{eq:case c=1 Nlambda}
N(\lambda)=h(\lambda^{-1})[1+\hat{\varepsilon}(\lambda^{-1})].
\end{equation} Note that  $|\hat{\varepsilon}(t)|=\op{O}(\epsilon(t))$. In addition, let $(a_j)_{j\geq1}\subset(0,\infty)$ be such that
\begin{equation}\label{eq:case c=1 a_j}
a_j=1-\epsilon(\lambda_j^{-1})
\end{equation} for $j$ large enough. Combining ~(\ref{eq:case c=1 Nlambda})--(\ref{eq:case c=1 a_j}) we get
\begin{equation*}
h(\lambda_j^{-1})(1-|\hat{\varepsilon}(\lambda_j^{-1})|)\leq N(\lambda_j)\leq j \leq N(a_j\lambda_j)\leq h((a_j\lambda_j)^{-1})(1+|\hat{\varepsilon}((a_j\lambda_j)^{-1})|).
\end{equation*}
As $h^\natural$ is ultimately increasing, for $j$ large enough, we have
\begin{equation*}
h^\natural[	h(\lambda_j^{-1})(1-|\hat{\varepsilon}(\lambda_j^{-1})|)]\leq h^\natural(j)\leq h^\natural[  h((a_j\lambda_j)^{-1})(1+|\hat{\varepsilon}((a_j\lambda_j)^{-1})|)].
\end{equation*}
Recall that, by assumption there is $q>p$ such that $t^{-q}h(t)$ is ultimately decreasing, i.e., there is $t_0>0$ such that
\begin{equation*}
\lambda^{-q}h(\lambda)\leq t^{-q}h(t),\qquad\mbox{for all}\quad t_0\leq \lambda\leq t.
\end{equation*}
For $\lambda=h^\natural[	h(\lambda_j^{-1})(1-|\hat{\varepsilon}(\lambda_j^{-1})|)]$ and $t=\lambda_j^{-1}$, we have
\begin{equation*}
h^\natural[	h(\lambda_j^{-1})(1-|\hat{\varepsilon}(\lambda_j^{-1})|)]^{-q}h(\lambda_j^{-1})\leq \lambda_j^qh(\lambda_j^{-1}).
\end{equation*}
That is,
\begin{equation}\label{eq:case c=1 lower}
\lambda_j^{-1}(1-|\hat{\varepsilon}(\lambda_j^{-1})|)^{1/q}\leq h^\natural[	h(\lambda_j^{-1})(1-|\hat{\varepsilon}(\lambda_j^{-1})|)].
\end{equation}
For $\lambda=(a_j\lambda_j)^{-1}$ and $t=h^\natural[  h((a_j\lambda_j)^{-1})(1+|\hat{\varepsilon}((a_j\lambda_j)^{-1})|)]$, we get
\begin{equation*}
(a_j\lambda_j)^qh((a_j\lambda_j)^{-1})\leq h^\natural[  h((a_j\lambda_j)^{-1})(1+|\hat{\varepsilon}((a_j\lambda_j)^{-1})|)]^{-q}h((a_j\lambda_j)^{-1})(1+|\hat{\varepsilon}((a_j\lambda_j)^{-1})|).
\end{equation*}
That is,
\begin{equation*}
h^\natural[  h((a_j\lambda_j)^{-1})(1+|\hat{\varepsilon}((a_j\lambda_j)^{-1})|)]\leq (a_j\lambda_j)^{-1}(1+|\hat{\varepsilon}((a_j\lambda_j)^{-1})|)^{1/q}.
\end{equation*}
Combining this with ~(\ref{eq:case c=1 lower}), we get
\begin{equation}\label{eq:case c=1 a}
\lambda_j^{-1}(1-|\hat{\varepsilon}(\lambda_j^{-1})|)^{1/q}\leq h^\natural(j)\leq a_j^{-1}\lambda_j^{-1}(1+|\hat{\varepsilon}((a_j\lambda_j)^{-1})|)^{1/q}.
\end{equation}
Clearly, we have $1-|\hat{\varepsilon}(\lambda_j^{-1})|=1-\op{O}(|\hat{\varepsilon}(\lambda_j^{-1})|)=1-\op{O}(\epsilon(\lambda_j^{-1}))$. As  $\epsilon$ is an $\RV_{\rho}$-function and $a_j\to1$, we have $ \epsilon((a_j\lambda_j)^{-1})\sim a_j^{-\rho}\epsilon(\lambda_j^{-1})\sim \epsilon(\lambda_j^{-1})$. Similarly, for $j$ large enough, we see from ~(\ref{eq:case c=1 a_j}) that
\begin{equation*}
a_j^{-1}(1+|\hat{\varepsilon}((a_j\lambda_j)^{-1})|)^{1/q}=(1-\epsilon(\lambda_j^{-1}))^{-1}[1+\op{O}(\epsilon((a_j\lambda_j)^{-1}))]^{1/q}=1+\op{O}(\epsilon(\lambda_j^{-1})).
\end{equation*}
Combining this with ~(\ref{eq:case c=1 a}) shows that
\begin{equation}\label{eq:case c=1 b}
\lambda_j=h^\natural(j)^{-1}(1+\op{O}(\epsilon(\lambda_j^{-1}))).
\end{equation}
It remains to prove~(\ref{eq:case c>0 result}) for $h^\sharp$, instead of $h^\natural$. In view of ~(\ref{cond:hsharp}), we have
\begin{equation*}
h^\sharp(t)=h^\sharp\circ h(h^\natural(t))=h^\natural(t)(1+\op{O}(\epsilon\circ h^\natural(t))).
\end{equation*}
Thus,
\begin{equation}\label{eq:case c=1 c}
h^{\natural}(t)^{-1}=h^\sharp(t)^{-1}(1+\op{O}(\epsilon\circ h^\natural(t))).
\end{equation}
This implies
$\lambda_j^{-1}\sim h^\sharp(j)$.
As $\epsilon$ is $\RV_\rho$, it follows that $\epsilon(\lambda_j^{-1})\sim \epsilon\circ h^\natural(j)$, and hence we have
\begin{equation*}
\lambda_j=h^\natural(j)^{-1}(1+\op{O}(\epsilon\circ h^\natural(j))).
\end{equation*}
In particular, $h^\sharp(t)\sim h^\natural(t)$. As $\epsilon$ is $\RV_\rho$, it follows that $\epsilon\circ h^\sharp(t)\sim \epsilon\circ h^\natural(t)$. Combining this with ~(\ref{eq:case c=1 b})--(\ref{eq:case c=1 c}), we have
\begin{equation*}
\lambda_j=h^\sharp(j)^{-1}(1+\op{O}(\epsilon\circ h^\sharp(j)))(1+\op{O}(\epsilon\circ h^\sharp(j)))=h^\sharp(j)^{-1}(1+\op{O}(\epsilon\circ h^\sharp(j))).
\end{equation*}
This proves the result $c=1$. 

Suppose now that $c$ is any positive constant. If we apply~(\ref{cond:ii}) to $a=c^{1/p}$ we get
\begin{equation*}
h(c^{1/p}t)=c h(t)\left(1+\op{O}(\epsilon(t))\right) \qquad \textup{as}\ t\rightarrow \infty.
\end{equation*}
Thus, if we set $h_c(t):=h(c^{1/p}t)$, then 
\begin{equation*}
N(\lambda)=h(\lambda^{-1}) \left[ c+\op{O}\left(\epsilon(\lambda^{-1})\right)\right]=ch(\lambda^{-1}) \left[ 1+\op{O}\left(\epsilon(\lambda^{-1})\right)\right]=h_c(\lambda^{-1}) \left[ 1+\op{O}\left(\epsilon(\lambda^{-1})\right)\right].
\end{equation*}
Setting $h_c^\sharp(t):=c^{-1/p}h^\sharp(t)$, we have
\begin{equation*}
h_c^\sharp\circ h_c(t)=c^{-1/p}h^\sharp \circ h(c^{1/p}t)=c^{-1/p}\cdot c^{1/p}t\cdot[1+\op{O}(\epsilon(c^{1/p}t))].
\end{equation*}
As $\epsilon$ is an $\RV_\rho$-function, we have $ \epsilon(c^{1/p}t)\sim c^{\rho/p}\epsilon(t)$, and so we obtain
\begin{equation*}
h_c^\sharp\circ h_c(t)=t(1+\op{O}(\epsilon(t))).
\end{equation*}
The result for $c=1$ then gives
\begin{equation*}
\lambda_j= h_c^\sharp(j)^{-1}(1+\op{O}(\epsilon(\lambda_j^{-1})))= c^{1/p}h^\sharp(j)^{-1}(1+\op{O}(\epsilon(\lambda_j^{-1})))=h^\sharp(j)^{-1}(c^{1/p}+\op{O}(\epsilon(\lambda_j^{-1}))).
\end{equation*}
This implies that $\lambda_j^{-1}\sim c^{-1/p}h^\sharp(j)$. As $\epsilon$ is $\RV_\rho$, it follows that $\epsilon(\lambda_j^{-1})\sim c^{-\rho/p}\epsilon \circ h^\sharp(j)$ and so we get the asymptotic~(\ref{eq:case c>0 result}). The proof is complete. 
\end{proof}

We have the following extension of Lemma~\ref{lem:app.Ntolambda1} to the case $c=0$. 

\begin{lemma}\label{lem:app.Ntolambda2}
 Suppose that the conditions (i)--(iii) are satisfied, and 
  \begin{equation*}
  N(\lambda)= \op{O}\left(\epsilon(\lambda^{-1}) h(\lambda^{-1})\right)  \qquad \textup{as}\ \lambda \rightarrow 0^+.
\end{equation*}
Assume further that $\epsilon(t)$ is $\RV_\rho$ with $-p<\rho\leq 0$, and we have
\begin{equation}\label{eq:app.sharp-Potter}
 \frac{h^\sharp\left(t(\epsilon\circ h^\sharp)(t)\right)}{h^\sharp(t)}=\op{O}\big( (\epsilon\circ h^\sharp) (t)^{\frac{1}{p}}\big)\qquad \textup{as}\ t\rightarrow \infty.
\end{equation}
Then, we have
 \begin{equation*}
 \lambda_j = \op{O}\left((\epsilon \circ h^\sharp)(j)^{\frac1p}h^\sharp(j)^{-1}\right) \qquad \textup{as}\ j\rightarrow \infty.
\end{equation*}
\end{lemma}
\begin{proof}
 By assumption there is $C>0$ such that for $\lambda$ small enough we have
 \begin{equation}\label{eq:case c=0}
 N(\lambda)\leq Ch(\lambda^{-1})\epsilon(\lambda^{-1}).
 \end{equation}
Let $a\in(0,1)$. For $j$ large enough we then have
\begin{equation*}
j\leq N(a\lambda_j)\leq Ch(a^{-1}\lambda_j^{-1})\epsilon(a^{-1}\lambda_j^{-1}).
\end{equation*}
As $h^\sharp$ is ultimately increasing, it follows that for $j$ large enough we have
\begin{equation}\label{eq:appendix.hsharpj-ineq}
h^\sharp(j)\leq h^\sharp\left[ C\epsilon(a^{-1}\lambda_j^{-1})h(a^{-1}\lambda_j^{-1})\right].
\end{equation}

Note that $\epsilon(t)h(t)$ is $\RV_{p+\rho}$ with $p+\rho>0$. This ensures that $\epsilon(t)h(t)\rightarrow \infty$ as $t\rightarrow \infty$. Moreover, by 
Lemma~\ref{lem:UCT-consequence} we have $\epsilon(a^{-1}\lambda_j^{-1})h(a^{-1}\lambda_j^{-1})\sim a^{-(p+\rho)}\epsilon(\lambda_j^{-1})h(\lambda_j^{-1})$. Combining this with the fact that $h^\sharp(t)$ is $\RV_{1/p}$, and applying Lemma~\ref{lem:UCT-consequence} once again, we get
\begin{align}
 h^\sharp\left[ C\epsilon(a^{-1}\lambda_j^{-1}) h(a^{-1}\lambda_j^{-1}) \right] &\sim  h^\sharp\left[ Ca^{-(p+\rho)}\epsilon(\lambda_j^{-1})h(\lambda_j^{-1})\right] \nonumber\\
 & \sim 
\left(Ca^{-(p+\rho)}\right)^{\frac{1}{p}} h^\sharp\left[ \epsilon(\lambda_j^{-1})h(\lambda_j^{-1})\right].
\label{eq:appendix.hsharpj-ineq2} 
\end{align}

Moreover, as $h^\sharp \circ h(t)\sim t$ and $\epsilon$ is $\RV_\rho$, Lemma~\ref{lem:UCT-consequence} further ensures that 
$\epsilon(\lambda_j^{-1})\sim (\epsilon\circ h^\sharp)(h(\lambda_j^{-1}))$. As $h^\sharp$ is $\RV_{1/p}$ by using Lemma~\ref{lem:UCT-consequence} again we get
\begin{equation*}
 h^\sharp[ \epsilon(\lambda_j^{-1})h(\lambda_j^{-1})]\sim  h^\sharp\left[ (\epsilon\circ h^\sharp)\left(h(\lambda_j^{-1})\right)h(\lambda_j^{-1})\right]. 
\end{equation*}
 Thus, 
 \begin{equation*}
 \frac{\epsilon(\lambda_j^{-1})^{-\frac1p} h^\sharp\left[ \epsilon(\lambda_j^{-1})h(\lambda_j^{-1})\right]}{\lambda_j^{-1}} \sim 
 \frac{ (\epsilon\circ h^\sharp)(h(\lambda_j^{-1}))^{-\frac1p} h^\sharp\left[ (\epsilon\circ h^\sharp)\left(h(\lambda_j^{-1})\right)h(\lambda_j^{-1})\right]}{h^\sharp (h(\lambda_j^{-1}))}
\end{equation*}
Combining this with~(\ref{eq:app.sharp-Potter}) we then deduce that
\begin{equation*}
 \lambda_j h^\sharp\left[ \epsilon(\lambda_j^{-1})h(\lambda_j^{-1})\right] =\op{O}\left(\epsilon(\lambda_j^{-1})^{\frac1p}\right). 
\end{equation*}
Together with~(\ref{eq:appendix.hsharpj-ineq})--(\ref{eq:appendix.hsharpj-ineq2}) this implies that
\begin{equation}\label{eq:appendix.lambda-hsharpj-ineq}
\lambda_jh^\sharp(j)=\op{O}\left(\epsilon(\lambda_j^{-1})^{\frac1p}\right).
\end{equation}

Note that~(\ref{eq:appendix.lambda-hsharpj-ineq})  implies that $\lambda_jh^\sharp(j)=\op{O}(1)$, and so there is $\alpha>0$ such that
\begin{equation*}
\lambda_j^{-1}\geq \alpha h^\sharp(j), \qquad \text{for}\quad j\gg 1.
\end{equation*}
As $\epsilon$ is decreasing and $\RV_\rho$, it follows that
\begin{equation*}
\epsilon\left(\lambda_j^{-1}\right)\leq \epsilon(\alpha h^\sharp(j))\sim \alpha^\rho \epsilon(h^\sharp (j)).
\end{equation*}
This implies that $ \epsilon(\lambda_j^{-1})=\op{O}(\epsilon(h^\sharp (j)))$. Combining this with~(\ref{eq:appendix.lambda-hsharpj-ineq}) we then get
\begin{equation*}
 \lambda_jh^\sharp(j)=\op{O}\left(\epsilon\left(h^\sharp(j)\right)^{\frac1p}\right), 
\end{equation*}
which gives the result. 
\end{proof}

\begin{remark}\label{rmk:app.sharp-Potter1}
 The condition~(\ref{eq:app.sharp-Potter}) is always satisfied by $h^\sharp(t)=t^{1/p}$. It is also satisfied by $h^\sharp(t)=t^{1/p}(\log(t+2))^q$, $q\neq 0$, and $\epsilon(t)=(\log t)^{-1}$ (see \S\ref{eq:app.example-N2lambda-logq}).  
\end{remark}

\begin{remark}
In general, the condition~(\ref{eq:app.sharp-Potter}) need not hold, but by Potter's Theorem (see, e.g., \cite[Theorem~1.5.6]{BGT:Cambridge87}), for all $\delta\in(0,1)$, we have
 \begin{equation*}
 \frac{h^\sharp\left(\epsilon(t)t\right)}{h^\sharp(t)}=\op{O}\left( \epsilon(t)^{\frac{1-\delta}{p}}\right)\qquad \textup{as}\ t\rightarrow \infty.
\end{equation*}
By arguing as in the proof of Lemma~\ref{lem:app.Ntolambda2} it then can be shown that as $j\rightarrow \infty$ we have
\begin{equation*}
 \lambda_j = \op{O}\left((\epsilon \circ h^\sharp)(j)^{\frac{1-\delta}{p}}h^\sharp(j)^{-1}\right) \qquad \forall \delta \in (0,1). 
\end{equation*}
\end{remark}

\subsection{Example: $h(t)=t^p$}
If $h(t)=t^p$, $p>0$, then its inverse is  $h^\natural(t):=t^{\frac1p}$, $t\geq 0$, which is a continuous $\RV_{\frac1p}$-function. In addition, the condition~(\ref{eq:app.sharp-Potter}) is automatically satisfied by $h^\natural(t)=t^{\frac1p}$ (\emph{cf}.~Remark~\ref{rmk:app.sharp-Potter1}). Note also that $h(at)=a^ph(t)$ and $t^{-q}h(t)=t^{-q+p}$ is decreasing for any $q>p$. Therefore, the assumptions of Lemma~\ref{lem:app.Ntolambda1} and Lemma~\ref{lem:app.Ntolambda2} are satisfied by any decreasing null $\RV_\rho$-function $\epsilon: [t_0,\infty)\rightarrow (0,\infty)$. For instance, we may take $\epsilon(t)=t^{-\alpha}(\log t)^{\beta}$, $t\geq 2$, with $\alpha>0$ and $\beta \in \R$. This leads to the following result.

\begin{lemma}\label{lem:app.example-tp}
 Suppose there are $\alpha>0$ and $\beta \in \R$ such that
 \begin{equation*}
 N(\lambda) = \lambda^{-p}\left[ c+ \op{O}\left(\lambda^{\alpha}|\log \lambda|^{\beta}\right)\right]  \qquad \textup{as}\ \lambda\rightarrow 0^+.
\end{equation*}
\begin{enumerate}
 \item[(i)] If $c>0$, then, as $j\rightarrow \infty$, we have
 \begin{equation*}
 \lambda_j = c^{\frac1p}j^{-\frac1p}\left[ 1+ \op{O}\left( j^{-\frac{\alpha}{p}}(\log j)^\beta\right)\right].
\end{equation*}

\item[(ii)] If $c=0$ and $\alpha<p$, then
\begin{equation*}
 \lambda_j = \op{O}\left( j^{-\frac1p-\frac{\alpha}{p^2}}(\log j)^{\frac{\beta}{p}}\right)  \qquad \textup{as}\ j\rightarrow \infty.
\end{equation*}
\end{enumerate}
\end{lemma}
\begin{proof}
 Put $\epsilon(t)=t^{-\alpha}(\log t)^\beta$, $t\geq 2$. If $c>0$, then by Lemma~\ref{lem:app.Ntolambda1} we have
 \begin{equation*}
  \lambda_j = c^{\frac1p}j^{-\frac1p}\left[ 1+ \op{O}\left( \epsilon(j^{\frac1p})\right)\right].
\end{equation*}
If $c=0$ and $\alpha <p$, then $\epsilon(t)$ is $\RV_\rho$ with $\rho\in(-p,0]$. We thus may apply Lemma~\ref{lem:app.Ntolambda2} to get
\begin{equation*}
 \lambda_j = \op{O}\left( j^{-\frac1p}\epsilon(j^{\frac1p})^{\frac1p}\right)  \qquad \textup{as}\ j\rightarrow \infty.
\end{equation*}
The result follows by using the fact that
\begin{equation*}
 \epsilon(j^{\frac1p}) = j^{-\frac{\alpha}{p}} \left(\log (j^{\frac1p})\right)^\beta=p^{-\beta}  j^{-\frac{\alpha}{p}} (\log j)^\beta.
\end{equation*}
The proof is complete.
\end{proof}

\subsection{Example: $h(t)=t^p(\log(t+2))^q$}\label{eq:app.example-N2lambda-logq} Suppose that $h(t)=t^p(\log(t+2))^q$ with $p>0$ and $q\neq 0$. As mentioned in Example~\ref{ex:asymp inverse} an $\RV_{1/p}$ asymptotic inverse is  $h^\sharp(t):= p^{\frac{q}{p}} t^{\frac{1}{p}} \left(\log(t+2)\right)^{-\frac{q}{p}}$, $t\geq 0$. In particular, at the level of principal term asymptotics we have the equivalence~(\ref{eq:App.principa-asymp-tplogtq}). 

To get remainder terms we need a slight refinement of the asymptotic inverse $h^\sharp(t)$. This is provided by the following lemma.

\begin{lemma}\label{lem:app.h1sharp}
 Define
 \begin{equation*}
 h_1^\sharp(t):=p^{\frac{q}{p}} t^{\frac{1}{p}}\left[\log(t+2) - q\log \log (t+e)\right]^{-\frac{q}{p}}, \qquad t\geq 0.
\end{equation*}
\begin{enumerate}
 \item[(i)] For all $a>0$, we have
 \begin{equation*}
 h(at)=a^p h(t) \left[ 1+ \op{O}\left((\log t)^{-1}\right) \right]  \qquad \textup{as}\ t\rightarrow \infty.
\end{equation*}

\item[(ii)] As $t\rightarrow \infty$, we have
\begin{equation}\label{eq:adjustment of asymptotic inverse}
 h_1^\sharp \circ h(t)=t \left[ 1+ \op{O}\left((\log t)^{-1}\right) \right] .
\end{equation}
\end{enumerate}
\end{lemma}
\begin{proof}
Let $a>0$. We have
\begin{equation*}
h(at)=(at)^p[\log (at+2)]^q.
\end{equation*}
As $t\to\infty$, we have
\begin{equation*}
\log(at+2)=\log[at(1+\frac2{at})]=\log a+\log t+\log(1+\frac2{at})=\log t\cdot(1+\op{O}((\log t)^{-1})).
\end{equation*}
Thus,
\begin{align*}
h(at)&=a^pt^p[\log t\cdot(1+\op{O}((\log t)^{-1}))]^q=a^pt^p(\log t)^q (1+\op{O}((\log t)^{-1}))\\
&=a^ph(t)(1+\op{O}((\log t)^{-1})).
\end{align*}
This proves part~(i).

It remains to prove part ~(ii). As $t\to\infty$, we have
\begin{equation*}
h^\sharp_1(t)\sim h^\sharp(t)\sim p^{\frac{q}{p}} t^{\frac1p}(\log t)^{-\frac{q}{p}}.
\end{equation*}
As $h^\sharp(t)$ is an asymptotic inverse of $h(t)$, we deduce that so is $h_1^\sharp(t)$. We have
\begin{equation}\label{eq:computation of h_1sharp}
h_1^\sharp\circ h(t)=p^{\frac{q}{p}}h(t) \left[\log(h(t)+2) - q\log \log (h(t)+e)\right]^{-\frac{q}{p}}.
\end{equation}
One can easily verify that, for some constant $c_q>0$,
\begin{align*}
\Big(\frac{\log(t+2)}{\log t}\Big)^q\leq 1+\frac{c_q}{t},\quad \forall t\geq e.
\end{align*}
We then have
\begin{equation*}
h(t)=t^p[\log(t+2)]^q=t^p(\log t)^q\Big(\frac{\log(t+2)}{\log t}\Big)^q=t^p(\log t)^q[1+\op{O}(1/t)].
\end{equation*}
This implies that
\begin{equation}\label{eq:1/p power}
h(t)^{\frac1p}=t(\log t)^{\frac{q}{p}}[1+\op{O}(1/t)].
\end{equation}
In addition, we have
\begin{align*}
\log[h(t)+2]&=\log[t^p(\log t)^q[1+\op{O}(1/t)]+2]\\
&=\log[t^p(\log t)^q(1+\op{O}(1/{\log t}))]\\
&=p\log t+q\log\log t+\op{O}(1).
\end{align*}
Likewise, $\log[h(t)+e]=p\log t(1+\op{o}(1))$ and so
\begin{equation*}
\log\log[h(t)+e]=\log[p\log t(1+\op{o}(1))]=\log\log t+\op{O}(1).
\end{equation*}
Thus,
\begin{align*}
\log[h(t)+2]-q\log\log[h(t)+e]
&=p\log t+q\log\log t-q\log\log t+\op{O}(1)\\
&=p\log t[1+\op{O}(1/{\log t})].
\end{align*}
Combining this with ~(\ref{eq:computation of h_1sharp})--(\ref{eq:1/p power}) gives
\begin{align*}
h_1^\sharp\circ h(t)&=p^{\frac{q}{p}}[t(\log t)^{\frac{q}{p}}(1+\op{O}(1/t))] [p\log t(1+\op{O}(1/{\log t}))]^{-\frac{q}{p}}\\
&=p^{\frac{q}{p}}t(\log t)^{\frac{q}{p}}p^{-\frac{q}{p}}(\log t)^{-\frac{q}{p}}(1+\op{O}(1/{\log t}))\\
&=t(1+\op{O}(1/{\log t})).
\end{align*}
This proves ~(ii). The proof is complete.
\end{proof}

\begin{remark}
The asymptotic~(\ref{eq:adjustment of asymptotic inverse}) is not satisfied by the asymptotic inverse $h^\sharp(t)$, since it can be shown that
\begin{equation*}
  h^\sharp \circ h(t)=t \left[ 1 - \frac{q^2}{p^2}\frac{\log \log t}{\log t}+\op{O}\left(\frac{1}{\log t}\right) \right] .
\end{equation*}
We need to use  $h_1^\sharp(t)$ to get an asymptotic with a $\op{O}((\log t)^{-1})$-remainder term. Note that a remainder term of that order is needed to apply Corollary~\ref{cor:Hyper.Weyl-hyper.sLpq} to get $g$-hypermeasurability.
\end{remark}

It follows from Lemma~\ref{lem:app.h1sharp} that the conditions (ii)--(iii) of Lemma~\ref{lem:app.Ntolambda1} and  Lemma~\ref{lem:app.Ntolambda2} are satisfied by $h(t)$ and $h_1^\sharp(t)$ by taking $\epsilon(t)=(\log t)^{-1}$, $t\geq 2$, which is a  decreasing null $\RV_0$-function. If $r>p$, then $t^{-r}h(t)=t^{-(r-p)}(\log t)^q$ is ultimately decreasing, and so condition~(i) holds as well.

Moreover, as $h_1^\sharp(t)\sim p^{\frac{q}{p}} t^{\frac{1}{p}}(\log t)^{-\frac{q}{p}}$, we have $\log [h_1^\sharp(t)]\sim p^{-1}\log t$, and hence
\begin{equation}\label{eq:app.epsilonh1sharp-log-example}
( \epsilon\circ h_1^\sharp)(t)= \frac{1}{\log \big[h_1^\sharp(t)\big]}\sim \frac{p}{\log t}. 
\end{equation}
Thus, 
\begin{equation*}
 \frac{(\epsilon \circ h_1^\sharp)(t)^{-\frac{1}{p}}h_1^\sharp\big[ t (\epsilon \circ h_1^\sharp)(t)\big] }{h_1^\sharp(t)}\sim \bigg( \frac{ \log \big[p(\log t)^{-1}t\big]}{\log t} \bigg)\sim 1. 
\end{equation*}
It then follows that condition~(\ref{eq:app.sharp-Potter}) is satisfied. 

This shows that we may apply Lemma~\ref{lem:app.Ntolambda1} and Lemma~\ref{lem:app.Ntolambda2} with $h_1^\sharp(t)$ and $\epsilon(t)$ as above. This leads to the following result.

\begin{lemma}\label{lem:app.example-tplogq}
 Suppose that
 \begin{equation*}
 N(\lambda)= \lambda^{-p}|\log \lambda|^q \left[c +\op{O}\left(|\log \lambda|^{-1}\right)\right] \qquad \textup{as}\ \lambda\rightarrow 0^+.
\end{equation*}
\begin{enumerate}
 \item[(i)] If $c>0$, then, as $j\rightarrow \infty$, we have
 \begin{equation*}
 \lambda_j = \left(p^{-q}c\right)^{\frac 1p} j^{-\frac1p} (\log j )^{\frac{q}{p}} \left[ 1- \frac{q^2}{p} \frac{\log \log j}{\log j} + \op{O}\left(\frac{1}{\log j}\right) \right].
\end{equation*}

\item[(ii)] If $c=0$, then
\begin{equation*}
  \lambda_j = \op{O}\left(j^{-\frac1p}(\log j)^{\frac{q-1}{p}}\right) \qquad \textup{as}\ j\rightarrow \infty.
\end{equation*}
\end{enumerate}
\end{lemma}
 \begin{proof}
If $c=0$, then by Lemma~\ref{lem:app.Ntolambda2} we get
\begin{equation}\label{eq:case c=0 h_1sharp}
\lambda_j=\op{O}\left( h_1^\sharp(j)^{-1}(\epsilon \circ h_1^\sharp)(j)^{\frac1{p}}\right).
\end{equation}
Combining this with~(\ref{eq:app.epsilonh1sharp-log-example})  and the fact that $h_1^\sharp(t)\sim p^{\frac{q}{p}} t^{\frac{1}{p}}(\log t)^{-\frac{q}{p}}$ gives the result for $c=0$. 

Suppose now that $c>0$. By Lemma~\ref{lem:app.Ntolambda1}  and~(\ref{eq:app.epsilonh1sharp-log-example}) we have
\begin{align}\label{eq:case c>0 h_1sharp}
\lambda_j &=h_1^\sharp(j)^{-1}\left[1+\op{O}\big((\epsilon \circ h_1^\sharp)(j)\big)\right] \\
& = h_1^\sharp(j)^{-1}\left[1+\op{O}\big((\log j)^{-1}\big)\right]. 
\end{align}

As in the proof of Lemma \ref{lem:app.h1sharp}, we have
\begin{align*}
&\log(t+2)=\log t+\log(1+2/t)=\log t+\op{O}(1);\\
&\log(t+e)=(\log t)(1+\op{o}(1));\\
&\log\log(t+e)=\log\log t+\op{o}(1)=\log\log t+\op{O}(1).
\end{align*}
Setting $L(t)=\frac{\log\log t}{\log t}$, we then get
\begin{align*} h_1^\sharp(t)&= p^{\frac{q}{p}}t^{\frac1p}[\log t-q\log\log t+\op{O}(1)]^{-\frac{q}{p}}\\
&=p^{\frac{q}{p}}t^{\frac1p}(\log t)^{-\frac{q}{p}}[1-qL(t) +\op{O}((\log t)^{-1})]^{-\frac{q}{p}}\\
&=p^{\frac{q}{p}}t^{\frac1p}(\log t)^{-\frac{q}{p}}[1+q^2p^{-1}L(t) +\op{O}((\log t)^{-1})].
\end{align*}
Combining this with ~(\ref{eq:case c>0 h_1sharp}) then gives
\begin{align*}
\lambda_j&=(cp^{-q})^{\frac1p}j^{-\frac1p}(\log j)^{\frac{q}{p}}[1-q^2p^{-1}L(j) +\op{O}((\log j)^{-1})][1+\op{O}((\log j)^{-1})]\\
&=(cp^{-q})^{\frac1p}j^{-\frac1p}(\log j)^{\frac{q}{p}}[1-q^2p^{-1}L(j) +\op{O}((\log j)^{-1})].
\end{align*}
This proves the result for $c>0$. The proof is complete.
\end{proof}

\section{Proof of Proposition~\ref{prop:Pietsch-Dixmier}}\label{app:proof-Pietsch-Dixmier}
In this appendix, for the sake of completeness and the reader's convenience, we include an alternative proof of Proposition~\ref{prop:Pietsch-Dixmier}.

We first need to show that if $\omega$ is any extended limit, then $\omega\circ C_g$ is a Banach limit. To this end, we use the following lemma.

\begin{lemma}[{\cite[Proposition~4.5]{LU:JMAA25}}]\label{lem:regular summation}
The following hold.
 \begin{itemize}
 \item[(i)]  If $a_j \rightarrow L$, then $C_g(a)_j \rightarrow L$.

 \item[(ii)] $\ran C_g\circ (S-1)\subseteq \co$.
\end{itemize}
\end{lemma}

\begin{remark}
 As $C_g(1)=1$ the proof of the first part only amounts to showing that $C_g(\co)\subseteq \co$. The proof of the second part also uses this property and the fact that $g(t)$ is $\RV_{-1}$ (see~\cite{LU:JMAA25}). In the proof of~\cite[Proposition~4.5]{LU:JMAA25} the property that $C_g(\co)\subseteq \co$ is deduced
  from~\cite[Theorem~3.2.7]{BC:Oxford00}. In fact, it follows from Lemma~\ref{lem:auxiliary result} that  $2^kg(2^k)\sim (\log 2) (G(2^{k+1})-G(2^k))$, and so we have
  \begin{equation}\label{eq:lacunary partial sum}
\sum_{k\leq j} 2^kg(2^k) \sim (\log 2)\sum_{k\leq j} \left(G(2^{k+1})-G(2^k)\right)\sim (\log 2) G(2^{j+1})\sim (\log 2) G(2^j).
\end{equation}
This implies that $\sum_{k=0}^\infty 2^kg(2^k)=\infty$, and so we have
 \begin{equation*}
 x_j = \op{o}\left(2^jg(2^j)\right) \ \Longrightarrow \ \sum_{k\leq j} x_k = \op{o}\bigg(\sum_{k\leq j} 2^kg(2^k)\bigg).
\end{equation*}
It follows from this that $C_g(\co)\subseteq \co$.
 \end{remark}

We are now in a position to prove the following result.

\begin{lemma}\label{lem:Cesaro -banach}
 If $\omega:\ell_\infty \rightarrow \C$ is an extended limit, then $\omega\circ C_g$ is a Banach limit.
\end{lemma}
\begin{proof}
As $C_g$ preserves positivity and the first part of Lemma~\ref{lem:regular summation} ensures that $C_g(1)=1$ and $C_g(\co)\subseteq \co$, we see that $\omega \circ C_g$ is a positive linear functional on $\ell_\infty$ such that $\omega \circ C_g(1)=\omega(1)=1$ and $\omega\circ C_g=\omega=0$ on $\co$. That is, $\omega \circ C_g$ is an extended limit.

Moreover, if $a\in\ell_\infty$, then the second part of Lemma~\ref{lem:regular summation} asserts that $C_g(Sa-a)\in\co$. Thus,
\begin{equation*}
 \omega\circ C_g (Sa)- \omega\circ C_g (a)=\omega\left( C_g(Sa-a)\right)=0.
\end{equation*}
It follows that $\omega \circ C_g$ is a shift-invariant extended limit. That is, $\omega \circ C_g$ is a Banach limit. The proof is complete.
\end{proof}

As mentioned in Section~\ref{sec:Pietsch}, the Banach limits of the form provided by Lemma~\ref{lem:Cesaro -banach} are called \emph{$C_g$-factorizable}.

In what follows, we denote by $\ell_\infty(\gamma)$ the ideal of $\ell_\infty$ generated by the sequence $\gamma=(g(j))_{j\geq 0}$, and we let 
$\tau:\ell_\infty(\gamma)\rightarrow \ell_\infty$ be the operator defined by
\begin{equation*}
 \tau(x)_j = \frac{1}{G(j)}\sum_{k<j} x_k, \qquad x=(x_j)_{j\geq 0}\in \ell_\infty(\gamma).
\end{equation*}
Thus, given any extended limit $\omega:\ell_\infty \rightarrow \C$, we have
\begin{equation*}
 \Trw(T)=\omega\circ \tau (\lambda(T)) \qquad \forall T\in \sL_g.
\end{equation*}

Our approach is based on the following asymptotic result. This result seems to be new, even in the case of the weak trace-class, i.e., with $g$ such that $g(t)=t^{-1}$ for $t\geq 1$  (see, however,~\cite[p.\ 206]{LSZ:Book} for related considerations).

Recall that $\tilde{\ell}_g$ consists of sequences $(x_j)_{j\geq 0}$ in $\ell_g$ such that $|x_0|\geq |x_1|\geq \cdots$. In particular, for such sequences we have $|x_j|=\mu_j(x)=\op{O}(g(j))$.

\begin{lemma}\label{lem:new}
 For every sequence $x=(x_j)_{j\geq 0}\in \tilde{\ell}_g$, we have
 \begin{equation}\label{eq:C_g vs tau}
 \delta\circ C_g \circ \hat{\alpha}_g(x)_j = (\log 2)  \tau(x)_j + \op{o}(1).
\end{equation}
In particular, for every extended limit $\omega$ on $\ell_\infty$, we have
\begin{equation}\label{eq:factor by C_g}
 \omega\circ\tau (x)= \frac{1}{\log 2} \omega\circ \delta\circ C_g \circ \hat{\alpha}_g(x).
\end{equation}
\end{lemma}
\begin{proof}
As~(\ref{eq:factor by C_g}) is a direct consequence of~(\ref{eq:C_g vs tau}), we only need to prove~(\ref{eq:C_g vs tau}). For $j\in \D_m$, we have
\begin{gather*}
 (\delta\circ C_g \circ \hat{\alpha}_g)(x)_j = C_g\circ \hat{\alpha}_g(x)_m = \frac{y_m}{\sigma_m}, \qquad \text{where}\\
 \sigma_m = \sum_{k\leq m} 2^k g(2^k), \qquad y_m= \sum_{k\leq m}\sum_{i\in \D_k} x_i= \sum_{k=0}^{2^{m+1}-2} x_k.
\end{gather*}
We know from~(\ref{eq:lacunary partial sum}) that $\sigma_m\sim (\log 2) G(2^{m+1})$. As $G(2^m-1)\leq G(j) \leq G(2^{m+1})$, and we have $G(2^m-1)\sim G(2^m)\sim G(2^{m+1})$, we see that $G(j)\sim G(2^{m+1})$. Thus,
\begin{equation}\label{eq:partial sum asymptotic}
 \sigma_m \sim (\log 2) G(j).
\end{equation}

As $x \in \tilde{\ell}_g$, we have $|x_j|\leq Cg(j)$ for some constant $C>0$ independent of $j$. Thus,
\begin{equation*}
 \bigg| y_m -\sum_{k< j} x_k\bigg| \leq \sum_{k=j}^{2^{m+1}-2} |x_k| \leq C  \sum_{k=j}^{2^{m+1}-2} g(k)\leq C 2^m g(j)\leq C(j+1)g(j).
\end{equation*}
As $tg(t)=\op{o}(G(t))$, and hence $(j+1)g(j)=\op{o}(G(j))$, we see that
\begin{equation*}
 y_m -\sum_{k< j} x_k =\op{o}(G(j)).
\end{equation*}
Combining this with~(\ref{eq:partial sum asymptotic}) then gives
\begin{equation*}
  (\delta\circ C_g \circ \hat{\alpha}_g)(x)_j =  \frac{y_m}{\sigma_m}=  \frac{\log 2}{G(j)}\sum_{k< j} x_k+\op{o}(1)=(\log 2)  \tau(x)_j + \op{o}(1).
\end{equation*}
This gives~(\ref{eq:C_g vs tau}). The proof is complete.
\end{proof}

We are now in a position to prove Proposition~\ref{prop:Pietsch-Dixmier}.

\begin{proof}[Proof of Proposition~\ref{prop:Pietsch-Dixmier}]
 First, we observe that if $\omega:\ell_\infty\rightarrow \C$ is an extended limit, then $\omega\circ \alpha$ and $\omega\circ \delta$ are extended limits, since the operators $\alpha$ and $\delta$ both preserve positivity and each of the subspaces $\C1$ and $\co$. Using Lemma~\ref{lem:Cesaro -banach} we then see that $\theta:=\omega\circ \delta \circ C_g$ is a $C_g$-factorizable Banach limit. Recall also that by Proposition~\ref{prop:DK} if $T\in \sL_g$, then every eigenvalue sequence $\lambda(T)$ is contained in $\tilde{\ell}_g$. Therefore, we may apply Lemma~\ref{lem:new} to get
  \begin{equation*}
 \Trw(T)=\omega \circ \tau \left(\lambda(T)\right)= \frac{1}{\log 2}(\omega\circ \delta\circ C_g) \circ \hat{\alpha}_g \left(\lambda(T)\right)= \hat{\varphi}_\theta(T).
\end{equation*}

Conversely, let $\theta:\ell_\infty \rightarrow \C$ be a $C_g$-factorizable Banach limit, i.e., $\theta=\omega\circ C_g$ for some extended limit $\omega$. As mentioned above, $\tilde{\omega}:=\omega\circ \alpha$ is an extended limit. As $\alpha\circ \delta=\op{id}$, we have
\begin{equation*}
 \theta=\omega\circ C_g=\omega\circ \alpha \circ \delta \circ C_g=\tilde{\omega}\circ \delta \circ C_g.
\end{equation*}
Therefore, by using Lemma~\ref{lem:new}, we see that, for all $T\in \sL_g$, we have
\begin{equation*}
\hat{\varphi}_\theta (T)= \frac{1}{\log 2} \theta \circ \hat{\alpha}_g\left(\lambda(T)\right)=
 \frac{1}{\log 2} \tilde{\omega}\circ \delta \circ C_g\left(\lambda(T)\right)= \tilde{\omega}\circ \tau \left(\lambda(T)\right)= \Tr_{\tilde{\omega}}(T).
\end{equation*}
This shows that $\hat{\varphi}_\theta$ is a Dixmier trace. This completes the proof of Proposition~\ref{prop:Pietsch-Dixmier}.
\end{proof}

\end{document}